\documentclass[preprint,10pt]{elsarticle}
\usepackage{latexsym}
\usepackage{float}
\usepackage{pstricks}
\usepackage{amssymb}
\usepackage{amsmath} 
\usepackage{amsfonts}
\usepackage{mathrsfs}
\usepackage{enumitem}
\usepackage{bm}
\usepackage{hyperref}
\usepackage{listings}% USE TO DISPLAY GAP CODES
\lstdefinelanguage{GAP}{
 	basicstyle=\ttfamily,
 	keywords={true, false, function, return, fail, if, in, while, do, od, else, elif, fi, break, continue},
 	keywordstyle=\color{blue}\bfseries,
 	otherkeywords={% Operators
 		>, <, ==
 	},
 	identifierstyle=\color{black},
 	sensitive=True,
 	comment=[l]{\#},
 	commentstyle=\color{cyan},
 	stringstyle=\color{red},
 	morestring=[b]',
 	morestring=[b]"
 }
 
\usepackage[OT2,OT1]{fontenc}
\newcommand\cyr{\renewcommand\rmdefault{wncyr}\renewcommand\sfdefault{wncyss}\renewcommand\encodingdefault{OT2}\normalfont\selectfont}
\DeclareTextFontCommand{\textcyr}{\cyr}
\usepackage{mathtools}

\DeclareFontFamily{U}{wncy}{}
\DeclareFontShape{U}{wncy}{m}{it}{<->wncyi10}{}
\DeclareSymbolFont{UWCyr}{U}{wncy}{m}{it}
\DeclareMathSymbol{\Bcyr}{\mathalpha}{UWCyr}{"6}

\def\jj{{\mathfrak j}}
\def\ii{{\mathfrak i}}
\def\noi{\noindent}

\def\im{\text{\rm Im}}
\def\ker{\text{\rm Ker\,}}

\usepackage{amsthm,array}
\usepackage{array}
\usepackage{graphicx}
\providecommand{\U}[1]{\protect\rule{.1in}{.1in}}

\newcolumntype{Y}{>{\raggedleft\arraybackslash}X}

\def\itemc{\itemindent=10pt\labelsep=8pt\labelwidth10pt\itemsep=4pt}

\def\bc{{\mathbb{C}}}
\def\Ker{\text{\rm Ker\;}}
\def\Im{\text{\rm Im\;}}

\def\bn{{\mathbb{N}}}

\def\br{{\mathbb{R}}}

\def\bz{{\mathbb{Z}}}

\def\br{\mathbb R}

\def\wt{\widetilde}
\def\wh{\widehat}

\def\re{\text{\rm Re\,}}
\def\bcyr{\text{\cyr b}}

\def\vs{\vskip.3cm}
\def\noi{\noindent}
\def\gdeg{G\text{\rm -deg}}

\def\s1deg{S^1\text{deg}}

\def\Om{\Omega}
\def\id{\text{\rm Id\,}}
\def\one{^{-1}}
\def\vp{\varphi}
\def\ve{\varepsilon}
\def\cV{\mathcal V}
\def\cU{\mathcal U}
\def\cW{\mathcal W}
\def\L{\mathscr L} 
\def\zhong{\bm\Sigma}
\newrgbcolor{violet}{.6 .1 .8} 
\newrgbcolor{lightyellow}{1 1 .8} 
\newrgbcolor{lightblue}{.80 1 1}
\newrgbcolor{mygreen}{0 .66 .05} 
\definecolor{mygreen}{rgb}{0,.66,.05}
\definecolor{lightyellow}{rgb}{1,1,.80}
\newrgbcolor{orange}{1 .6 0}
\newrgbcolor{GreenYellow}{.85 1 .31}
\newrgbcolor{Yellow}{1  1  0}
\newrgbcolor{Goldenrod}{1  .90  .16}
\newrgbcolor{Dandelion}{1  .71  .16}
\newrgbcolor{Apricot}{1  .68  .48}
\newrgbcolor{Peach}{1  .50  .30}
\newrgbcolor{Melon}{1  .54  .50}
\newrgbcolor{YellowOrange}{1  .58  0}
\newrgbcolor{Orange}{1  .39  .13}
\newrgbcolor{BurntOrange}{1  .49  0}
\newrgbcolor{Bittersweet}{1.  .4300  .24}
\newrgbcolor{RedOrange}{1  .23  .13}
\newrgbcolor{Mahogany}{1.  .4475  .4345}
\newrgbcolor{Maroon}{1.  .4084  .5376}
\newrgbcolor{BrickRed}{1.  .3592  .3232}
\newrgbcolor{Red}{1  0  0}
\newrgbcolor{OrangeRed}{1  0  .50}
\newrgbcolor{RubineRed}{1  0  .87}
\newrgbcolor{WildStrawberry}{1  .04  .61}
\newrgbcolor{Salmon}{1  .47  .62}
\newrgbcolor{CarnationPink}{1  .37  1}
\newrgbcolor{Magenta}{1  0  1}
\newrgbcolor{VioletRed}{1  .19  1}
\newrgbcolor{Rhodamine}{1  .18  1}
\newrgbcolor{Mulberry}{.6668  .1180  1.}
\newrgbcolor{RedViolet}{.9538  .4060  1.}
\newrgbcolor{Fuchsia}{.5676  .1628  1.}
\newrgbcolor{Lavender}{1  .52  1}
\newrgbcolor{Thistle}{.88  .41  1}
\newrgbcolor{Orchid}{.68  .36  1}
\newrgbcolor{DarkOrchid}{.60  .20  .80}
\newrgbcolor{Purple}{.55  .14  1}
\newrgbcolor{Plum}{.50  0  1}
\newrgbcolor{Violet}{.98 .15 .95}
\newrgbcolor{RoyalPurple}{.25  .10  1}
\newrgbcolor{BlueViolet}{.84  .38  .98}
\newrgbcolor{Periwinkle}{.43  .45  1}
\newrgbcolor{CadetBlue}{.38  .43  .77}
\newrgbcolor{CornflowerBlue}{.35  .87  1}
\newrgbcolor{MidnightBlue}{.4414  .9259  1.}
\newrgbcolor{NavyBlue}{.06  .46  1}
\newrgbcolor{RoyalBlue}{0  .50  1}
\newrgbcolor{Blue}{0  0  1}
\newrgbcolor{Cerulean}{.06  .89  1}
\newrgbcolor{Cyan}{0  1  1}
\newrgbcolor{ProcessBlue}{.04  1  1}
\newrgbcolor{SkyBlue}{.38  1  .88}
\newrgbcolor{Turquoise}{.15  1  .80}
\newrgbcolor{TealBlue}{.1572  1.  .6668}
\newrgbcolor{Aquamarine}{.18  1  .70}
\newrgbcolor{BlueGreen}{.15  1  .67}
\newrgbcolor{Emerald}{0  1  .50}
\newrgbcolor{JungleGreen}{.01  1  .48}
\newrgbcolor{SeaGreen}{.31  1  .50}
\newrgbcolor{Green}{0  1  0}
\newrgbcolor{ForestGreen}{.1992  1.  .2256}
\newrgbcolor{PineGreen}{.3100  1.  .5575}
\newrgbcolor{LimeGreen}{.50  1  0}
\newrgbcolor{YellowGreen}{.56  1  .26}
\newrgbcolor{SpringGreen}{.74  1  .24}
\newrgbcolor{OliveGreen}{.6160  1.  .4300}
\newrgbcolor{RawSienna}{.53  .28  .16}
\newrgbcolor{Sepia}{1.  .7510  .70}
\newrgbcolor{Brown}{.41  .25  .18}
\newrgbcolor{TAN}{.86  .58  .44}
\newrgbcolor{Gray}{1.  1.  1.}
\newrgbcolor{Black}{1  1  1}
\newrgbcolor{White}{1  1  1}

\newtheorem{theorem}{Theorem}[section]
\newtheorem{proposition}[theorem]{Proposition}
\newtheorem{lemma}[theorem]{Lemma}

\newtheorem{definition}[theorem]{Definition}
\newtheorem{remark}[theorem]{Remark}
\newtheorem{example}[theorem]{Example}

\newtheorem{remark-definition}[theorem]{Remark and Definition}

\journal{J. Differential Equations}
\begin{document}
\begin{frontmatter}

%\title{Title\tnoteref{label1}}
%\tnotetext[label1]{} 
%\author{Name\corref{cor1}\fnref{label2}}
%\ead{email address}
 %\cortext[cor1]{Corresponding author at: Department of Mathematical Sciences,}
%\affiliation{organization={The University of Texas at Dallas},
%             city={Richardson},
 %            postcode={75080},
 %            state={TX},
 %            country={USA}}
%\fntext[label3]{}

\title{Equivariant Global Hopf Bifurcation in Abstract Nonlinear Parabolic Equations}
 \author[label1]{Zalman Balanov}
 %\ead{email address}
 \author[label1]{Wies\l aw Krawcewicz}
 \ead{wieslaw@utdallas.edu}
 \author[label1]{Arnaja Mitra\corref{cor1}}
 \ead{arnaja.mitra@utdallas.edu}
 \author[label1]{Dmitrii  Rachinskii}
 \ead{Dmitry.Rachinskiy@utdallas.edu}
 \affiliation[label1]{organization={Department of Mathematical Sciences, The University of Texas at Dallas},
             city={Richardson},
             postcode={75080},
             state={TX},
             country={USA}}

\cortext[cor1]{Corresponding author at: Department of Mathematical Sciences, The University of Texas at Dallas, Richardson, 75080, TX, USA.}  


%\author{Zalman Balanov, Wies\l aw Krawcewicz, Arnaja Mitra, Dmitrii  Rachinskii}

%\affiliation{organization={},%Department and Organization

\begin{abstract}
In this paper we study local and global  symmetric Hopf bifurcation in abstract parabolic systems by means of the twisted equivariant degree.
\end{abstract}

\begin{keyword}
Partial differential equations; Global Hopf Bifurcation; Equivariant degree theory
\end{keyword}
\end{frontmatter}


\section{Introduction} 
%\subsection{Subject and goal}
For a given parameterized family of dynamical systems, the Hopf bifurcation is the phenomenon of the occurrence of small amplitude periodic solutions from an equilibrium state, when the parameter crosses some critical value (for which the linearization of the dynamical system admits a purely imaginary eigenvalue). In order to study the Hopf bifurcation in symmetric systems of parabolic partial differential equations (PDEs), one uses appropriate
equivariant analysis tools
\cite{Fied, Gol, MM, Kiel-book, Wu, GW, F, RL, MN, KWX, Kiel2}. In this paper, we propose a general framework based on  the twisted equivariant degree theory.
\vs
Symmetric systems of parabolic PDEs are commonly used to model physical phenomena.  Studying the behavior of these systems in the vicinity of a Hopf bifurcation gives important insights into how they respond to perturbations and how they transition from one state to another. Moreover, a Hopf bifurcation can provide means for classifying the behavior of symmetric periodic solutions \cite{HWZ} in applications to control 
\cite{BF1, BF2, BF3} and for predicting the behavior of complex physical systems \cite{BKR, MM, CI, CGK, L, GL, D, JMWB, GM}. The  study of Hopf bifurcation in symmetric systems of parabolic equations is important  in   many practical from the application point of view ares, including biology (to model population dynamics \cite{LMT}, and the spread of diseases \cite{HTH},
physics (in fluid dynamics, to understand the dynamics of vortices \cite{L, AK}, engineering (to design  control systems  and to prevent instability-induced failures \cite{BHK}, medicine (to model certain heart conditions, such as atrial fibrillation \cite{TK,KSM}, environmental science (to model the spread of pollutants and to understand the dynamics of ecosystems \cite{S,ZZWW}, economics (to study the stability of economic systems \cite{LV},  neuroscience (to understand  dynamics of neural networks and the synchronization of neural oscillations \cite{W2,Wu}, etc.
\vs
%Many problems in population dynamics, neural networks, fluid dynamics, solid mechanics, elasticity, chemistry, electrical engineering, etc., leads to studying the Hopf bifurcation in networks of identical oscillators coupled symmetrically or in systems of PDEs defined on symmetric domains. 
The presence of symmetries (the so-called equivariance) results in the following paradigmatic question: {\it what is the impact of symmetries of a dynamical system on its actual dynamics, for example, on the existence, multiplicity (the number of bifurcating branches), the global behavior of the branches, and symmetric properties of periodic solutions?} 
\vs
In this paper we present a new self-contained treatment of the global symmetric Hopf bifurcation problem of the type
\begin{equation}\label{eq:Hopf-equation}
\dot u(t)+Au(t)=\mathcal F(\alpha,u(t)), \quad {\text a.e. } \;\;t\in \br,
\end{equation}
where $A:{\mathfrak{D}}(A)\subset H\to H$ is an accretive self-adjoint equivariant operator on a Hilbert representation of a compact Lie group $\Gamma$ and $\mathcal F:\br\times {\mathfrak{D}}(A)\to H$ is a family of completely continuous equivariant maps (with respect to ${\mathfrak{D}}(A)$ equipped with the graph topology).  
\vs
Our main goal is to detect unbounded branches of non-constant periodic solutions bifurcating from an equilibrium and provide a comprehensive description of their symmetric properties. Under additional assumption that $\mathcal F$ is odd with respect to second variable, it is possible to reduce the Fuller space for \eqref{eq:Hopf-equation} to a subspace where the only equilibria to \eqref{eq:Hopf-equation} are $(\alpha,\beta,0)$ and in this context, one can  provide sufficient conditions for the existence of unbounded branches. 
\vs
%\subsection{Methodology}

We use the method based on the twisted  equivariant degree theory \cite{AED,KW}. In a nutshell, the twisted equivariant degree is a topological tool, which allows ``counting" orbits of solutions to symmetric equations 
%(involving maps with one free parameter) 
according to symmetric properties of the solutions. As such, it is a counterpart of the usual Brouwer degree, which ``counts" solutions to non-symmetric equations.  
For a given dynamical system respecting some spatial symmetries $\Gamma$,  spatio-temporal symmetries of its non-constant periodic solutions are subgroups of the group $G:=\Gamma\times S^1$ of the type
\[
K^{\theta,l} :=\{(\gamma,z)\in K\times S^1:  \theta(\gamma)=z^l\},
\]
where $K$ is a subgroup of $\Gamma$ and $\theta: K\to S^1$ is a group homomorphism  (with $l\in\{0,1,2,\dots\}$), which are called {\it twisted subgroups} of $G$.  As a matter of fact, the twisted degree can detect non-constant periodic solutions with twisted spatio-temporal symmetries $K^{\theta,l}  \le G$ (when  $N_\Gamma(K)/K$ is a finite group). To be more precise, problem \eqref{eq:Hopf-equation} can be reformulated in appropriate Banach $G$-representation $\mathscr E$ as a two parameter equation 
\begin{equation}\label{eq:two-parameter}
\mathscr F(\alpha,\beta,u)=0, \quad u\in \mathscr E,
\end{equation}
where $\beta>0$ represents an unknown frequency of a periodic solution, and  $\mathscr F:\br\times \br_+\times \mathscr E\to \mathscr E$ is a $G$-equivariant completely continuous field.
Then, with a  bifurcation point $(\alpha,\beta,u_o)$, one can associate a local $G$-equivariant  bifurcation invariant 
$\omega_G(\alpha,\beta,u_o)$
provided by the twisted equivariant degree of the map $\mathscr F$ (complemented by an auxiliary function $\vp$) in a neighborhood of $(\alpha,\beta,u_o)$, containing equivariant topological information on the symmetries of bifurcating from $(\alpha,\beta,u_o)$ branches of non-constant periodic solutions. The invariant $\omega_G(\alpha,\beta,u_o)$ is represented as a formal sum (indexed by the conjugacy classes  $(H_j)$ of twisted subgroups in $G$) of the form
\[
\omega_G(\alpha,\beta,0)=n_1(H_1)+n_2(H_2)+\dots +n_m(H_m), \quad n_j\in \bz.
\]
Then, knowing that $n_j\not=0$ allows us to predict the existence of a branch of non-constant periodic solutions, bifurcating from $(\alpha,\beta,u_o)$ with the symmetries at least $H_j$.
\vs

\section{Two Parameter $G$-Equivariant Bifurcation} \label{sec:equiv-bif-Twisted}
\subsection{Abstract results}
Assume that $G:=\Gamma\times S^1$, where $\Gamma$ is a compact Lie group. In this subsection, we present a general setting for a two-parameter 
$G$-symmetric bifurcation problem and explain how the twisted $G$-equivariant degree theory (see Appendix \ref{appendix:Twist}) can be used to study the occurrence of the bifurcation and symmetric properties of the bifurcating branches of solutions (both local and global settings are treated). To be more precise, assume that $\mathbb E$ is an isometric Banach $G$-representation and take  the $G$-representation  $\br^2\times \mathbb E$, where $G$ acts trivially on the parameter space  $\br^2$. Let $\mathscr F:\br^2\times \mathbb E\to \mathbb E$  satisfy the following assumptions:
%Consider a completely continuous $G$-equivariant field $\mathscr F:\br^2\times \mathbb E\to %\mathbb E$  satisfying the following assumptions
\vs
\begin{enumerate}[label=(${\mathscr B}_\arabic*$)]\itemc
\item\label{bif0} $\mathscr F$ is a completely continuous $G$-equivariant field;
\item\label{bif1} $\mathscr F(\alpha,\beta,0)=0$ for every $(\alpha,\beta)\in \br^2$; 
\item\label{bif2} for  all $(\alpha,\beta)\in \br^2$,  the derivative $\mathscr A(\alpha,
\beta):=D_v\mathscr F(\alpha,\beta,0):\mathbb E\to \mathbb E$ exists, depends continuously on  $(\alpha,\beta)$ and, for all $(\alpha_o,\beta_o) \in \mathbb R^2$, one has:  
\begin{equation}\label{eq:A-twisted}
\lim_{(\alpha,\beta,v) \to (\alpha_o,\beta_o,0)} \frac{\|\mathscr F(\alpha,\beta,v)-\mathscr A(\alpha,\beta)v\|}{\|v\|} = 0.
\end{equation}%\item\label{bif3}
\end{enumerate}
\vs
Our objective is to study symmetric properties of the solution set to the equation %\eqref{eq:operator-eq}
\begin{equation}\label{eq:operator-eq}
\mathscr F(\alpha,\beta,v)=0, \quad (\alpha,\beta,v)\in \br^2\times \mathbb E.
\end{equation} 
We introduce the following notation. Put
%We denote by
\begin{equation}\label{eq:trivialM-twisted}
M:=\{(\alpha,\beta,0): (\alpha,\beta,0)\in \br^2\times \mathbb E\}
\end{equation}
and call its elements {\it trivial solutions} to \eqref{eq:operator-eq}. Also, denote by 
$\mathscr S$ the set of all  {\it nontrivial solutions} to  \eqref{eq:operator-eq}, i.e.
\begin{equation}\label{eq:nontr-twisted}
\mathscr S:=\{(\alpha,\beta,v)\in \br^2\times \mathbb E: \mathscr F(\alpha,\beta,v)=0 \;\text{ and }\; v\not=0\}.
\end{equation}
%and call its elements {\it nontrivial solutions} to  \eqref{eq:operator-eq}. 
\vs

\begin{definition}\label{def:branch-1}\rm

%\begin{itemize}
% \item[(i)] 
A set $\mathcal C\subset \mathscr S$ is called a {\it branch} of nontrivial solutions to  \eqref{eq:operator-eq}  if $\overline{\mathcal C}$ is a connected component of $\overline{\mathscr S}$ containing a nontrivial continuum. Moreover,  if $(\alpha_o,\beta_o,0)\in \overline{\mathcal C}$, then we say that $\mathcal C$  {\it bifurcates} from   $(\alpha_o, \beta_o,0)$ and  $(\alpha_o,\beta_o,0)$ is called a {\it bifurcation point}  for \eqref{eq:operator-eq}.
%\item[(ii)] A point $(\alpha_o,\beta_o,0) \in M$ is called a {\it bifurcation point} for  \eqref{eq:operator-eq} if any neighborhood of $(\alpha_o,\beta_o,0)$ contains a nontrivial solution to \eqref{eq:operator-eq}.  
%\end{itemize}
\end{definition}
\vs 
%In what follows, 
%Following the standard definitions (see Definition \ref{def:s1-bif}), 
%we are looking for branches of nontrivial solutions 
Let $\mathcal C\subset \overline{\mathscr S}$ be a branch of nontrivial solutions bifurcating from a given trivial solution $(\alpha_o,\beta_o,0)\in M$. Then, one can easily show that
$\mathscr A(\alpha_o,\beta_o):\mathbb E\to \mathbb E$ cannot be an isomorphism, in which case,   $(\alpha_o,\beta_o,0)$ is called a {\it critical point} for  \eqref{eq:operator-eq}, and we put
\begin{equation}\label{eq:critical-twisted}
\Lambda:=\{(\alpha_o,\beta_o,0)\in M: \mathscr A(\alpha_o,\beta_o):\mathbb E\to \mathbb E \;\text{ is not an isomorphism}\}.
\end{equation} 
The set $\Lambda$ is called the {\it critical set} for \eqref{eq:operator-eq}.

%In the case, if such a branch $\mathcal C$ exists, one can easily show that
%$\mathscr A(\alpha_o,\beta_o):\mathbb E\to \mathbb E$ cannot be an isomorphism,   $
%(\alpha_o,\beta_o,0)$ is a {\it critical point} for  \eqref{eq:operator-eq} (see Definition 
%\ref{def:s1-critical}, i.e. it belongs to the set 
%\begin{equation}\label{eq:critical-twisted}
%\Lambda:=\{(\alpha_o,\beta_o,0)\in M: \mathscr A(\alpha_o,\beta_o):\mathbb E\to \mathbb E \;%\text{ is not an isomorphism}\}.
%\end{equation}
%\vs
%For the sake of clarity, we make the following additional assumption:
%\vs
%
%\begin{enumerate}[label=(${\mathscr B}_\arabic*$)]\setcounter{enumi}{3}\itemc
%\item\label{bif4} the set $\Lambda\subset \br^2$ is discrete.
%\end{enumerate}

\vs
\subsubsection{Local bifurcation results}\label{sec:local-twisted}
One can identify $\br^2$ with $\bc$ and write  $\lambda:=\alpha+i\beta$. 
Assume that the following condition is satisfied:
%Under the assumptions \ref{bif0}--\ref{bif2}, suppose, in addition,  that 
\begin{enumerate}[label=(${\mathscr B}_\arabic*$)]\setcounter{enumi}{3}\itemc
\item\label{bif3} $(\lambda_o,0)$ is an isolated point in $ \Lambda$, i.e. there exists $\delta > 0$
such that 
\begin{equation}\label{eq:delta-twisted}
\overline {B_{\delta}}\cap \Lambda=\{(\lambda_o,0)\},
\end{equation} 
where 
$B_{\delta}:=\{(\lambda,0)\in \bc \times \mathbb E: |\lambda-\lambda_o| < \delta\}$.
\end{enumerate}
%$(\lambda_o,0)$, 
%$\lambda_o=(\alpha_o,\beta_o)$  
%is an isolated point in $ \Lambda$. 
%One can identify $\br^2$ with $\bc$ so we can write  $\lambda:=\alpha+i\beta$ and $
%\lambda_o:=\alpha_o+i\beta_o$. Then, there exists $\delta>0$ such that  
%\begin{equation}\label{eq:delta-twisted}
%\overline {B_{2\delta}}\cap \Lambda=\{(\lambda_o,0)\},
%\end{equation} 
%where 
%$B_{\delta}:=\{(\lambda,0)\in \bc : |\lambda-\lambda_o| < \delta\}$.

\medskip
\noindent
Take $\ve>0$ and define the set
\begin{equation}\label{eq:O-bif}
\mathfrak O:=\left\{(\lambda,v)\in \bc \times \mathbb E: |\lambda-\lambda_o| < \delta, \; \|v\|<\ve\right\}.
\end{equation}
Under the assumptions \ref{bif0}--\ref{bif3}, the standard compactness argument yields the existence of $\delta$ and $\ve>0$ so small that the following property is satisfied:
\begin{equation}\label{eq:bif-comp-admiss}
\{(\lambda,v)\in \overline{\mathfrak O}: \mathscr F(\lambda,v)=0,\; |\lambda-\lambda_o|=\delta,\; \|v\|\le \ve \}\cap \overline{\mathscr S}=\emptyset 
\end{equation}
(cf. \eqref{eq:nontr-twisted}).
Using an equivariant analog of Tietze-Dugundji Theorem, one can 
construct an auxiliary $G$-invariant (continuous) function $\eta:\bc\times \mathbb E\to \br$ satisfying the condition 
%is called  $\mathfrak O$-complementing if 
\begin{equation}\label{eq:auxil}
\begin{cases} \eta(\lambda,0)<0 &\text{ if } \; |\lambda-\lambda_o|  = \delta,\\
\eta(\lambda,v)>0 &\text{ if } \; |\lambda-\lambda_o|\le \delta \text{ and }\; \|v\|=\ve.\\
\end{cases}
\end{equation}
In particular (cf. \eqref{eq:bif-comp-admiss} and \eqref{eq:auxil}), $\eta$ is an auxiliary function
for $\mathscr F$ in $\mathfrak O$.  
%Under the assumptions , for sufficiently small $\ve>0$, one has
%\begin{equation}\label{eq:bif-comp-admiss}
%\{(\lambda,v)\in \overline{\mathfrak O}: \mathscr F(\lambda,v)=0,\; |\lambda-\lambda_o|%=\delta,\; \|v\|\le \ve \}\cap \overline{\mathscr S}=\emptyset.
%\end{equation}
Define the map $\mathscr F_\eta:\bc\times \mathbb E\to \br\times \mathbb E$ by
\[
\mathscr F_\eta(\lambda,v) := (\eta(\lambda,v),\mathscr F(\lambda,v)), \quad (\lambda,v)\in \bc\times \mathbb E.
\]
Then, 
%by  \eqref{eq:bif-comp-admiss}, 
$(\mathscr F_\eta,\mathfrak O)$ is an admissible $G$-pair and one can define
\begin{equation}\label{eq:bif-inv-twisted}
\omega_G(\lambda_o):=\gdeg(\mathscr F_\eta,\mathfrak O).
\end{equation}
One can verify that $\omega_G(\lambda_o)$ is an element of the $A(\Gamma$)-module $A_1^t(G)$
generated by conjugacy classes  $\Phi_1^t(G)$ (see Appendix \ref{appendix:Twist} for more details) of twisted subgroups in  $G$, and that $\omega_G(\lambda_o)$ is independent of the choice of $\delta>0$, $\ve>0$ and the auxiliary function $\eta$. We will call the element $\omega_G(\lambda_o)$ the {\it local bifurcation invariant} for \eqref{eq:operator-eq}. We are now in a position to formulate the local bifurcation result for equation \eqref{eq:operator-eq}.

%By exactly the same arguments which were used to prove Theorem \ref{th:bif-cont-0}, one gets the following so-called {\it local bifurcation result}:
\vs

\begin{theorem}\label{th:bif-cont-1} Assume that $\mathscr F : \mathbb C \times \mathbb E\to \mathbb E$  satisfies \ref{bif0}--\ref{bif2}, and let $(\lambda_o,0)$ satisfy \ref{bif3}. Take $\delta>0$  and $\ve>0$ so small that condition  \eqref{eq:bif-comp-admiss} is satisfied. Suppose that $\mathfrak O$ is given by \eqref{eq:O-bif}, and let $\eta$ be a $G$-invariant auxiliary function for $\mathscr F$ in  $\mathfrak O$ satisfying 
\eqref{eq:auxil}. Assume that $\omega_G(\lambda_o)$ (given by \eqref{eq:bif-inv-twisted}) is different from zero.    
%is an $\mathfrak O$-complementing function and $\omega_G(\lambda_o)$ is given by 
%\eqref{eq:bif-inv-twisted}. If $\omega_G(\lambda_o)\not=0$,  
Then, there exists a branch $\mathcal C$ of nontrivial solutions bifurcating from $(\lambda_o,0)$. Moreover, if
%such that for every $(\mathcal H)\in \Phi_1^t(G)$ (see Appendix \ref{appendix:Twist} for details), 
\[
%\text{ if }\;\; 
\text{\rm coeff}^{\mathcal H}(\omega_G(\lambda_o))\not=0,\;\; 
%\text{ then }\;\; \mathcal C\cap \mathscr S^{\mathcal H}\not=\emptyset,
\]
for $(\mathcal H)\in \Phi_1^t(G)$, then there exists a branch $\mathcal C'$ of nontrivial solutions with symmetries at least $\mathcal H$ bifurcating from $(\lambda_o,0)$.
%where we denote by $\mathscr S^{\mathcal H}$ the set of the $H$-fixed points in $\mathscr S$. 
%In particular, there exists a nontrivial continuum $\bm C \subset \overline{\mathcal C}$ with 
%$(\lambda_o,0) \in \bm C$ such that for any $ (\lambda,v) \in \bm C$, $v \not= 0$, one has 
%$G_{(\lambda,v)}{:=\{g\in G: (\lambda,gv)=(\lambda, v)\}} \geq {\mathcal H}$.  
%{\color{red} Explain the notation $G_{(\lambda,v)}$}\\
\end{theorem}

%{\color{red} Shouldn't we define a symmetric branch as in the book and state its existence?}

%{\color{VioletRed} Dr. W. K. agreed  that we should state the same definition as the book and to write it.}

%{\color{red} So: Have you done it or not???}

\begin{proof} 
%We use the argument parallel to the one utilized in the proof of Theorem \ref{th:bif-cont-0}. 
Assume for contradiction, that there is no  branch $\mathcal C$  bifurcating from  $(\lambda_o,0)$ and such that   $\mathcal C^{\mathcal H} $ contains a non-trivial  continuum containing $(\lambda_o,0)$. By  \eqref{eq:H-twisted} and \eqref{eq:twisted-def}, one has
\begin{align*}
\text{\rm coeff}^{\mathcal H}(\omega_G(\lambda_o))&=d_{\mathcal H}(\mathscr F_\eta,\mathfrak O)\\
&=\frac{\displaystyle \deg_{S^1}(\mathscr F_\eta^{\mathcal H},\mathfrak O^{\mathcal H})-\sum_{(\mathcal K) > (\mathcal H)}d_{\mathcal K}(\mathscr F_\eta,\mathfrak O)\, n(\mathcal H,\mathcal K)\, 
|W(\mathcal K)/S^1|}{|W(\mathcal H)/S^1|}.
\end{align*}
Hence, by assumption $d_{\mathcal H}(\mathscr F_\eta,\mathfrak O)\not=0$, it follows that there exists a twisted orbit type $(\mathcal K)\ge (\mathcal H)$ such that $\deg_{S^1}(\mathscr F_\eta^{\mathcal K},\mathfrak O^{\mathcal K})\not=0$. 
%Notice that the cumulative $S^1$-degree satisfies existence and homotopy properties.
Put $K:=\mathscr F^{-1}(0)\cap  \overline{\mathfrak O}$,  
 $B_0:=(\mathbb C \times \{0\}) \cap \overline{\mathfrak O^{\mathcal K}}$ and $B_1:=\{(\lambda,v)\in \overline{\mathfrak O^{\mathcal K}}: \|v\|=\ve\}$,  and consider $K^{\mathcal K} = \mathscr F^{-1}(0)\cap  \overline{\mathfrak O^{\mathcal K}}$. By the equivariant Kuratowski's Lemma (see Theorem \ref{th:G-Kur}), there exist disjoint open $S^1$-invariant sets $\mathscr U_0$ and $\mathscr U_1$ in $\mathbb C \times \mathbb E$ such that 
\begin{gather*}
B_0\subset \mathscr U_0, \quad B_1\subset \mathscr U_1,\\
B_0\cup  B_1\cup K^{\mathcal K}\subset \mathscr U_0\cup \mathscr U_1.
\end{gather*} 
Put $K_0^{\mathcal K} := K^{\mathcal K} \cap \mathscr U_0$ and $K_1^{\mathcal K} := K^{\mathcal K} \cap \mathscr U_1$. Clearly, $\mathscr K_0^{\mathcal K}$ and $\mathscr K_1^{\mathcal K}$ are compact $S^1$-invariant sets. Take a continuous $S^1$-invariant Urysohn function $\mu : \mathbb C \times \mathbb E^{\mathcal K} \to [0,1]$ separating $K_0^{\mathcal K} \cup B_0$ and 
$K_1^{\mathcal K} \cup B_1$, i.e.
\begin{equation}\label{eq:twisted-Uryshon}
\mu(\lambda,v)=
\begin{cases}
1 &\text{ if } \; (\lambda,v)\in K^{\mathcal K}_0\cup B_0,\\
0 &\text{ if } \; (\lambda,v)\in K^{\mathcal K}_1 \cup B_1.
\end{cases}
\end{equation}
Put 
%and define the $\mathfrak O^K$-complementing function 
\[
\tau(\lambda,v):=\|v\|-\mu(\lambda,v)\ve.
\]
Clearly, $\tau$ is an auxiliary $S^1$-invariant function for $\mathscr F$ in 
$\mathfrak O^{\mathcal K}$. Since the local bifurcation invariant is independent of
the choice of an invariant auxiliary function, one has the following equality for the cumulative $S^1$-degrees:
 \begin{equation}\label{eq:equal-cum-deg}
 \deg_{S^1}(\mathscr F_\tau^{\mathcal K},\mathfrak O^{\mathcal K})=\deg_{S^1}(\mathscr F_\eta^{\mathcal  K},\mathfrak O^{\mathcal  K}).
 \end{equation}
By assumption and \eqref{eq:equal-cum-deg}, 
\[
\deg_{S^1}(\mathscr F_\tau^{\mathcal K},\mathfrak O^K)\not=0.
\]
Hence, by the existence property of the cumulative $S^1$-degree, there exists $(\lambda_*,v_*)\in \mathfrak O^{\mathcal K}$ such that
\begin{equation}\label{eq:twisted-K0K1-s1}
\mathscr F^{\mathcal K}_\tau (\lambda_*,v_*)=(0,0)\;\;\; \Leftrightarrow\;\;\; \begin{cases} 
\mathscr F^{\mathcal K} (\lambda_*,v_*)=0,\\
\|v_*\|-\mu(\lambda_*,v_*)\ve=0.
\end{cases}
\end{equation}
Clearly, $(\lambda_*,v_*)\in K^{\mathcal K}$ and since 
$K^{\mathcal K} = K^{\mathcal K}_0 \cup K^{\mathcal K}_1$, 
$K^{\mathcal K}_0 \cap K^{\mathcal K}_1 = \emptyset$, one has that either $(\lambda_*,v_*)\in K^{\mathcal K}_0$ or $(\lambda_*,v_*)\in K^{\mathcal K}_1$.
However,
if $(\lambda_*,v_*)\in K^{\mathcal K}_0$, then, by \eqref{eq:twisted-K0K1-s1}, $\|v_*\|=\ve$, and one obtains  that $(\lambda_*,v_*)\in K^{\mathcal K}_1$, which is a contradiction. Similarly, if $(\lambda_*,v_*)\in K^{\mathcal K}_1$, then, by \eqref{eq:twisted-K0K1-s1}, $\|v_*\|=0$, and one obtains  that $(\lambda_*,v_*) \in K^{\mathcal K}_0$, which is again a  contradiction.
\end{proof}

\vs
In order to effectively use Theorem \ref{th:bif-cont-1}, one needs a method for the computation of the local bifurcation invariant $\omega_G(\lambda_o)$. The following proposition (allowing a passage to the linearization on an appropriate annulus-like domain) is the first step in this direction.
\vs
\begin{proposition}\label{pro:lin-bif-invariant-twisted}
Under the assumptions of Theorem  \ref{th:bif-cont-1}, take the linearization 
$\mathscr A(\lambda)$ of $\mathscr F(\lambda)$  and define the map $\mathfrak f_{\mathscr A} : \mathbb C \times \mathbb E\to \mathbb E$ by
\begin{equation}\label{eq:fA-bif}
\mathfrak f_{\mathscr A}(\lambda,v):=\left(\tfrac{\delta}{2} - |\lambda-\lambda_o|, \mathscr A(\lambda)v\right), \quad (\lambda,v)\in \bc\times \mathbb E.
\end{equation}
Put
\begin{equation}\label{eq:frak-O-o}
\mathfrak O_o := \Big\{(\lambda,v) \in \mathfrak O \, : \, {\delta \over 4} < 
|\lambda - \lambda_o| < \delta\Big\}.  
\end{equation}
Then, $\mathfrak f_{\mathscr A}$ is a completely continuous $\mathfrak O_o$-admissible 
$G$-equivariant field and 
\begin{equation}\label{eq:lin-bif-invariant-twisted}
\omega_G(\lambda_o)=\gdeg(\mathfrak f_{\mathscr A},\mathfrak O_o).
\end{equation}
\end{proposition}

\begin{proof} Clearly, $\mathscr A \in \text{L}_c(\mathbb E)$ (see, for example, \cite{KrasZab}, Theorem 17.1), from which the complete continuity of  the field $\mathfrak f_{\mathscr A}$ follows immediately. Also, since $G_{(\lambda,0)}= G$, the $G$-equivariance of  $\mathscr A$ follows from the $G$-equivariance of $\mathscr F$. 

\vs
Define the $G$-equivariant homotopy
\[
\mathfrak h_\eta(t,\lambda,v):=\Big(\eta(\lambda,v), (1-t)\mathscr F(\lambda,v)+t\mathscr A(\lambda)v)\Big),
\]
where $(\lambda,v)\in \bc\times \mathbb E$. We claim that $\mathfrak h_{\eta}$ is an $\mathfrak O$-admissible homotopy. Assume, for contradiction, that there exists a sequence $(t_n,\lambda_n,v_n)\in [0,1]\times \bc\times \mathbb E$ such that 
\begin{equation}\label{eq:lambda-delta}
|\lambda_n-\lambda_o|=\delta
\end{equation}
 and $0<\|v_n\|\le \ve$ satisfies  
\[
v_n \to 0 \;\; {\rm as} \;\; n\to \infty\;\; \;\; {\rm and} \;\;\; \; \mathfrak h_\eta(t_n,\lambda_n,v_n)=0
\]
(recall that, by definition, the auxiliary function $\eta$ is positive on the set $\{(\lambda,v) \in 
\partial \mathfrak O \, : \, \|v\| = \varepsilon\}$). 
Without loss of generality, one can assume that $t_n\to t_o$ and $\lambda_n\to \lambda_*$ as $n\to \infty$. Then, put $r(\lambda_n,v_n):=\mathscr F(\lambda_n,v_n)-\mathscr A(\lambda_n)v_n$ and notice that for any $n \in \mathbb N$, one has
\begin{equation}\label{eq:YTM223}
0 = t_n\mathscr F(\lambda_n,v_n)+(1-t_n)\mathscr A(\lambda_n)v_n 
%= t_n\big(\mathscr A(\lambda_n)v_n+r(\lambda_n,v_n)\big)+
%(1-t_n)\mathscr A(\lambda_n)v_n\\
=\mathscr A(\lambda_n)v_n +t_nr(\lambda_n,v_n).
\end{equation}
Combining \eqref{eq:YTM223} with $\|v_n \| \not=0$, one obtains
\[
0=\mathscr A(\lambda_n)\tfrac{v_n}{\|v_n\|}+ t_n\frac{r(\lambda_n,v_n)}{\|v_n\|}.
\]
Since, by \ref{bif2}, 
\[ \lim_{n\to\infty} \frac{r(\lambda_n,v_n)}{\|v_n\|}=0,\]
it follows that 
\begin{equation}\label{eq:A-lamda-0}
\lim_{n\to \infty} \mathscr A(\lambda_n)\tfrac{v_n}{\|v_n\|}=0.
\end{equation}
Put $w_n:=\tfrac{v_n}{\|v_n\|}$ and $\mathscr K(\lambda_n) := \id - \mathscr A(\lambda_n)$. 
By \eqref{eq:A-lamda-0} combined with compactness of $\mathscr K(\lambda_n)$ and $\|w_n\| =1$, there exists a subsequence $\{w_{n_k}\}$ such that 
\[
\lim_{k\to \infty} \mathscr K(\lambda_{n_k})w_{n_k} = \lim_{k\to \infty}w_{n_k} = w_o, \quad 
\|w_o\| = 1.
\]
%Since $\mathscr A(\lambda_n)=:\id -\mathscr K(\lambda_n)$, where $\mathscr K(\lambda_n):\mathbb E\to \mathbb E$ is a compact operator, $\|w_n\|=1$, it follows that there exists a subsequence $\{w_{n_k}\}$ such that 
%\[
%\lim_{k\to \infty} \mathscr K(\lambda_{n_k})w_{n_k}=w_o,
%\]
%thus 
%\[
%\lim_{k\to \infty} w_{n_k}=w_o,
%\]
Therefore (cf. \eqref{eq:lambda-delta}), 
\[
0=\lim_{k\to \infty}  \mathscr A(\lambda_{n_k})w_{n_k}=\mathscr A(\lambda_*)w_{o}, \quad \|w_o\|=1, \quad |\lambda_* -  \lambda_o| = \delta,
\]
which contradicts the choice of $\delta$ and condition \ref{bif3}.
%fact that $(\lambda_*,0)\not\in \Lambda$. 

\vs

Finally, take the homotopy 
$\mathfrak k : [0,1]\times \bc\times \mathbb E\to \mathbb E$ given by
\begin{equation}\label{eq:eta-t}
\mathfrak k(t,\lambda,v):=\Big(\eta_t(\lambda,v), \mathscr A(\lambda)v\Big), 
\end{equation}
where 
\[
\eta_t(\lambda,v)=(1-t)\eta(\lambda,v)+t\left(\tfrac{\delta}{2} - |\lambda-\lambda_o|\right).
\]
Now, if $\mathfrak k(t,\lambda,v) = 0$, then either $\|v\| = \ve$ and $\lambda = \lambda_o$
(in which case, $\eta_t(\lambda,v) > 0$), or $|\lambda - \lambda_o| = \delta$ and $v = 0$ 
(in which case, $\eta_t(\lambda,0) < 0$). Hence, by the excision property ($\mathfrak T_5$) of the degree (see Appendix, section \ref{appendix:Twist}), 
\[
\omega_G(\lambda_o) = \gdeg(f_{\mathscr A}, \mathfrak O) = \gdeg(f_{\mathscr A}, \mathfrak O_o),
\]
and the  result follows.
%It follows immediately from \eqref{eq:eta-t} and \ref{bif3} that if $\mathfrak k(t,\lambda,v) = 0$
%for some $(t,\lambda,v)\in [0,1] \times \partial \mathfrak O$, then $v\in \Ker \mathscr A(\lambda)$. Consider two cases (cf. \eqref{eq:auxil}). If $\|v\|=\ve$ and $\lambda=\lambda_o$, then, clearly, $\eta_t(\lambda_o,v)>0$. On the other hand, if $|\lambda-\lambda_o|=\delta$ and $v=0$, then $\eta_t(\lambda,0)<0$. Consequently, the homotopy $\mathfrak k$ is $\mathfrak O$-admissible and the result follows.
\end{proof}

\vs
In order to compute the degree $\gdeg(\mathfrak f_{\mathscr A},\mathfrak O_o)$, one can apply the  finite-dimensional reduction. With Proposition \ref{pro:lin-bif-invariant-twisted} supplemented by Proposition \ref{prop:a-map-finite} (see section \ref{sec:main}), one can give an effective computational formula for $\omega_G(\lambda_o)$. Namely, take the map 
$\mathscr A: \bc\to \text{\rm L}_c^G(\mathbb E)$ (see condition \ref{bif2}) and put 
$K:=\{\lambda\in \bc \, : \, {\delta \over 4} \leq   | \lambda-\lambda_o| \leq \delta\}$ (cf. \eqref{eq:frak-O-o}). 
 %Then, in order to compute the local bifurcation invariant $\omega_G(\lambda_o)$,  consider the map $\mathscr A: \bc\to \text{\rm L}_c^G(\mathbb E)$, put
%$K:=\{\lambda\in \bc: |\lambda-\lambda_o|=\delta\}$. 
Then, by Proposition \ref{prop:a-map-finite}, 
there exists a $G$-invariant decomposition $\mathbb E:=W\oplus\mathbb E^o$ with $\text{\rm dim}(W) < \infty$ \, and a homotopy $h : [0,1] \times K \to  \text{\rm GL}_c^G(\mathbb E)$  between $\mathscr A : K \to  \text{\rm GL}_c^G(\mathbb E)$   and a continuous map $a: K \to  \text{\rm GL}_c^G(\mathbb E)$ \big($ a(\lambda)  :=  h(1,\lambda)$\big) such that 
$a(\lambda)(W)\subset W$, $a(\lambda)(\mathbb E^o)\subset \mathbb E^o$, and  $a|_{\mathbb E^o}=\id_{\mathbb E^o}$ for all 
 $\lambda\in K$. 
 %Then, identify (by shifting and scaling) the set $K$ with $S^1:=\{\lambda\in \bc: |\lambda|=1\}$. 
 Then, by homotopy and suspension properties \ref{T2} and \ref{T9} of the degree (see Appendix, section \ref{appendix:Twist}) combined with Proposition \ref{pro:lin-bif-invariant-twisted}, one obtains
 \begin{equation}\label{eq:twisted-omega-reduction}
 \omega_G(\lambda_o)=\gdeg(\mathscr F_\eta,\mathfrak O)=\gdeg(\mathfrak f_{\theta,a},\mathfrak O^{\prime}_o),
 \end{equation}
 where 
 %$\theta(\lambda,v):=\tfrac{\delta}{2} - |\lambda-\lambda_o|$ and
 \begin{align*}
 \theta(\lambda,v) &:=\tfrac{\delta}{2} - |\lambda-\lambda_o|,\\
 \mathfrak O^{\prime}_o&:=\mathfrak O_o \cap (\bc\times W),\\
 \mathfrak f_{\theta,a}(\lambda,v)&:=\left(\theta(\lambda,v),a\left(\frac \lambda{|\lambda|}\right)v \right).
 \end{align*}
 %and $\theta$ is an auxiliary function for $a$ on $\mathfrak O^{\prime}_o$ satisfying the condition 
%is called  $\mathfrak O$-complementing if 
%\begin{equation}\label{eq:auxil}
%\begin{cases} \eta(\lambda,0)<0 &\text{ if } \; |\lambda-\lambda_o|  = \delta,\\
%\eta(\lambda,v)>0 &\text{ if } \; |\lambda-\lambda_o|\le \delta \text{ and }\; \|v\|=\ve\\
%\end{cases}
%\end{equation}
 \vs
 \noindent
By combining this with  the standard rescaling argument and Theorem \ref{th:comp-twisted},  one obtains 
 \begin{equation}\label{eq:omega-twisted-final}
 \omega_G(\lambda_o)=\prod_l {\bm \delta_l}\cdot \mathop{\sum_{k,l}}\limits_{k>0} d_{k,l}\deg_{\cV_{k,l}},
 \end{equation}
where $d_{k,l}$ is given by \eqref{eq:dkl}
%:=\deg({\det}_\bc(\wt a_{k,l}))$, 
and ${\bm \delta_l}$ is given by \eqref{eq:factor-j-twisted}.
\vs

\subsection{Global bifurcation results}\label{sec:global-twisted} 

To formulate the global bifurcation result for equation \eqref{eq:operator-eq}, we make the following assumption (cf. \eqref{eq:critical-twisted}):

\begin{enumerate}[label=(${\mathscr B}_\arabic*$)]\setcounter{enumi}{4}\itemc
\item\label{bif4} the set $\Lambda\subset \br^2$ is discrete.
\end{enumerate}

For our further exposition,
assume, in addition, that $\Gamma:=\Gamma_o\times \bz_2$, where the action of $\bz_2$  on $\mathbb E$ is antipodal, and put 
%. Put $
\begin{equation}\label{eq:def-G-o-H}
\bz_2^d:=\{(1,1),(-1,-1)\}\le \bz_2\times S^1 \qquad \text{and} \qquad {\bm H} := \{{\bm e}\} \times \bz_2^d \le \Gamma_o \times \bz_2\times S^1
\end{equation} 
%$ and take the subgroup 
%${\bm H} := \{e\} \times \{(1,1),(-1,-1)\}\le \Gamma_o \times \bz_2\times S^1$ 
(here ${\bm e}$ stands for the unity in $\Gamma_o$). It is convenient to study the global behavior of branches of solutions to \eqref{eq:operator-eq} on which $S^1$-acts nontrivially, by taking the restriction to the $\bm H$-fixed subspace, i.e. consider the equation
\begin{equation}\label{eq:operator-bif-H}
\mathscr F^{\bm H}(\lambda,v)=0, \quad (\lambda,v)\in \bc\times \mathbb E^{\bm H}.
\end{equation}
\vs
%of non-constaSuppose our goal is to identify the branches of nontrivial solutions to \eqref{eq:operator-eq} on which $S^1$-acts nontrivially. Therefore, it is convenient to consider the problem \eqref{eq:operator-eq} restricted to the $\bm H$-fixed subspace, i.e.
%\begin{equation}\label{eq:operator-bif-H}
%\mathscr F^H(\lambda,v)=0, \quad (\lambda,v)\in \bc\times \mathbb E^{\bm H}.
%\end{equation}
\noindent
Since
\begin{equation}\label{eq:G-o-def}
{G \over {\bm H}} = {\Gamma_o \times \mathbb Z_2 \times S^1 \over {\bm H}} \simeq \Gamma_o \times {\mathbb Z_2 \times S^1 \over \mathbb Z_2^d} \simeq \Gamma_o \times S^1 =: G_o,
\end{equation}
%Notice that $\Gamma_o\times S^1\simeq G/H$ and put $G_o:=\Gamma_o\times S^1$. T
the space $ \bc\times \mathbb E^{\bm H}$ is an isometric Banach $G_o$-representation. { To describe $\Gamma_0$-isotypic  decomposition of $\mathbb E^{\bm H}$, consider first the $\Gamma$-isotypic decomposition  
%{\it for simplicity} 
%that $\mathbb E$ admits the $\Gamma_o$-isotypic decomposition
%\Gamma_o$ admits 
%only finitely many $\Gamma_o$-isotypic components,
% in $\mathbb E$, 
%i.e. 
\begin{equation}\label{eq:Gamma-isot-E}
\mathbb E = \overline{\bigoplus_{l=0}^{\infty} V_l},
%\oplus V_1\oplus \dots \oplus V_r,
\end{equation}
where the component $V_l$ is modeled on the irreducible $\Gamma$-representation  $\cV_l$ (which is the irreducible $\Gamma_0$-representation  $\cV_l$ with antipodal $\bz_2$ action)}. The irreducible $\Gamma\times S^1$-representations can be easily described. Assume that $\{\cV_l\}$ is a complete list of irreducible $\Gamma$-representations. Then, for each $\Gamma$-irreducible representation $\cV_l$ and $k=0,1,2,\dots$, we define the $\Gamma\times S^1$-representation $\cV_{k,l}$ as follows:
\begin{itemize} 
\item if $\cV_l$ is of real type, then $\cV_{k,l}$ is the complexification $\cV_l^c$ of $\cV_l$ with the $S^1$-action given by $zv:=z^k\cdot v$, where $z\in S^1$, $v\in \cV_l^c$ and `$\cdot$' denotes the complex multiplication;

\item if $\cV_l$ is of complex type, then $\cV_l$  admits a natural complex structure and one can define  $\cV_{k,l}$ to be  $\cV_l$ with the $S^1$-action given by $zv:=z^k\cdot v$, where $z\in S^1$, $v\in \cV_l$ and `$\cdot$' denotes the complex multiplication in $\cV_l$.
\end{itemize} 

\vs
Let $\{\mathcal V_{k,l}\}$, $k,l = 0,1,2,...$, be the collection of all irreducible real $G = \Gamma \times S^1$-representations. Then,  \eqref{eq:Gamma-isot-E} suggests the following
$G$-isotypic decomposition:
\begin{equation}\label{eq:Gamma-S1-isot-E}
\mathbb E = \overline{\bigoplus_{l=0}^{\infty}\bigoplus_{k=0}^\infty V_{k,l}},
\end{equation}
where $V_{k,l}$ denotes the $G$-isotypic component of $\mathbb E$ modeled on $\cV_{k,l}$ (some of $V_{k,l}$ may be trivial).
Recall that for $k>0$, the space $V_{k,l}$ admits the natural complex structure, and
%(cf. \eqref{eq:ired-O2-g-z2-irred-r}--\eqref{eq:ired-O2-g-z2-irred-0}),
 the element 
$({\bm e},-1,-1) \in G = {\Gamma_0} \times \mathbb Z_2 \times S^1$ acts on $V_{k,l}$ by
\[
({\bm e},-1,-1)v = - e^{i\pi}v = 
\begin{cases}
- e^{ik\pi} \cdot v, \quad \rm{if} \quad k > 0;\\
-v, \qquad\quad   \rm{if} \quad k = 0,
\end{cases}
\] 
where ``$\cdot$" stands for the complex multiplication. Thus, for $k > 0$, one has
\begin{equation}\label{eq:non-trivial-S1-action}
({\bm e},-1,-1)v = v \quad \Leftrightarrow \quad {\bm e}^{ik\pi} = -1 \quad \Leftrightarrow \quad \text{$k$ is odd}.
\end{equation}
By combining \eqref{eq:Gamma-S1-isot-E} with \eqref{eq:non-trivial-S1-action}, one obtains the following $G_o$-isotypic decomposition of $\mathbb E^{\bm H}$:
\begin{equation}\label{eq:Go-iso-Banach}
 \mathbb E^{\bm H}=\overline{\bigoplus_{l=0}^{\infty}\bigoplus_{k'=0}^\infty V_{2k'+1,l}}.
\end{equation}
%where $V_{k',l}$ denotes the $G$-isotypic component of $\mathbb E$ modeled on $\cV_{k',l}$.

\vs
To pass from equation \eqref{eq:operator-eq}   to equation \eqref{eq:operator-bif-H} and to formulate the corresponding global result, 
%In order to formulate our global bifurcation result, 
one needs to redefine the set of critical points for \eqref{eq:operator-bif-H}, i.e. put
\begin{equation}\label{eq:Lambda-o-def}
\Lambda_o:=\{(\lambda_o,0)\in \bc\times \mathbb E^{\bm H}: \mathscr A(\lambda_o)|_{\mathbb E^{\bm H}} : \mathbb E^{\bm H} \to \mathbb E^{\bm H} \;  \;\text{ is not an isomorphism}\}.
\end{equation}
Clearly, $\Lambda_o\subset \Lambda$, thus assumption \ref{bif4} implies that the set $\Lambda_o$ is also discrete. By following the same construction as in subsection 
\ref{sec:local-twisted}, 
%\ref{sec:equiv-bif-Twisted}, 
%to each critical point $(\lambda_o,0)\in \Lambda_o$, 
one can associate the local bifurcation invariant $\omega_{G_o}(\lambda_o)\in A_1^t(G_o)$ to each critical point $(\lambda_o,0)\in \Lambda_o$. Notice that in such a case,
\begin{equation}\label{eq:omega-G-o-def}
\omega_{G_o}(\lambda_o)=  \sum_{k',l}d_{2k'+1,l}\deg_{\cV_{2k'+1,l}}.
\end{equation}
In addition, put
\begin{equation}\label{eq:matscr-S-o-def}
\mathscr S_o:=\{(\lambda,v)\in \bc\times \mathbb E^{\bm H}: \mathscr F^{\bm H}(\lambda,v)=0 \; \text { and } \; v\not=0\}.
\end{equation}
\vs
The following result can be proved using the Kuratowski's Lemma and the same standard argument applied in the classical Rabinowitz result combined with cosmetic modifications related 
to $G$-equivariance and infinite dimensions (see the proof of Theorem \ref{th:Rab} in section \ref{sec:main}).
\vs
\begin{theorem} \label{th:global-bif-twisted}
Let  $G:=\Gamma_o\times \bz_2\times S^1$, where $\Gamma_o$ is a compact Lie group, and let  $\mathbb E$ be a Banach $G$-representation with the antipodal $\mathbb Z_2$-action. %action of $\bz_2$ on $\mathbb E$ is antipodal. 
Let $\mathscr F:\br^2\times \mathbb E\to \mathbb E$ satisfy conditions  
%be a completely continuous $G$-equivariant field   
%satisfying 
\ref{bif0}--\ref{bif2} and \ref{bif4}. Suppose that $\mathcal C\subset \overline{\mathscr S_o}$ is a bounded branch of solution to equation \eqref{eq:operator-bif-H} ({where $\bm H$ is given by} \eqref{eq:def-G-o-H} and {$\mathscr S_o$ is given by} \eqref{eq:matscr-S-o-def}). Then,
\[ \mathcal C\cap \Lambda_o=:\{(\lambda_1,0),(\lambda_2,0), \dots,(\lambda_N,0)\}\] {(see \eqref{eq:G-o-def}, \eqref{eq:Lambda-o-def})}, and 
\begin{equation}\label{eq:omega-sum-twisted}
\sum_{s=1}^N \omega_{G_o}(\lambda_s)=0
\end{equation}
(see \eqref{eq:omega-G-o-def}).
\end{theorem}
\vs

\subsection{Relaxing conditions \ref{bif3} and \ref{bif4}}\label{sec:cluster-bif} In this section, we adapt the  results 
presented in subsections \ref{sec:local-twisted}  and \ref{sec:global-twisted} to the settings with relaxed conditions  \ref{bif3} and \ref{bif4}. 
%These generalizations will be used in what follows to treat {\it families} of dynamical systems undergoing a simultaneous  Hopf bifurcation (important in the robust control, see section {\color{red} ???}).
\vs
%Under the notations of subsection  \ref{sec:local-twisted}, 
Consider a family $\mathscr F:\br^2\times \mathbb E\to \mathbb E$  
satisfying conditions \ref{bif0}--\ref{bif2}, and 
 %We consider a compact Lie group $\Gamma$, put $G:=\Gamma\times S^1$, and assume that $\mathbb E$ is an isotypic Banach $G$-representation. Take the $G$-representation $\br^2\times \mathbb E$ (where $G$ acts trivially on $\br^2$). Consider a completely continuous $G$-equivariant field
%$\mathscr F:\br^2\times \mathbb E\to \mathbb E$ satisfying conditions \ref{bif1}--\ref{bif2}. As we are interested in \eqref{eq:operator-eq}
%we consider the manifold $M$ (given by \eqref{eq:trivialM-twisted}), the set of nontrivial solutions $\mathscr S$ (given by \eqref{eq:nontr-twisted}) and the 
%critical set $\Lambda$ (given by \eqref{eq:critical-twisted}).
%In order to make possible studying branches of  nontrivial solutions \eqref{eq:operator-eq}  to %from the critical set in a more general setting, we 
introduce the following concept (cf. \eqref{eq:trivialM-twisted} and  \eqref{eq:critical-twisted}).
\vs

\begin{definition}\label{ded:cluster-twisted}\rm
A compact set $K\subset \Lambda$ is called a {\it cluster of critical points} (or simply {\it cluster}) for \eqref{eq:operator-eq}, if there exists a diffeomorphism $\xi: \bc\to M$ such that:
\begin{itemize}
\item[(i)] $\xi$ preserves the standard orientations of $\bc$ and $M$;
\item[(ii)] $K\subset \mathcal U:=\xi(B_1(0))$, where $B_1(0):=\{\lambda\in \bc:|\lambda|<1\}$; 
\item[(iii)] $\overline {\mathcal U} \cap \Lambda=K$.
\end{itemize}
In such a case, we also say that $(K,\mathcal U)$ is a {\it cluster pair} and call $\mathcal U$ a {\it cluster isolating} neighborhood. 
%{\color{red} Why is not enough to require that $\xi$ is a homeomorphism?}
\end{definition}

\vs
An illustration of a  cluster pair is shown on Figure \ref{fig:cluster-twisted}.
\vs

\begin{figure}[h]
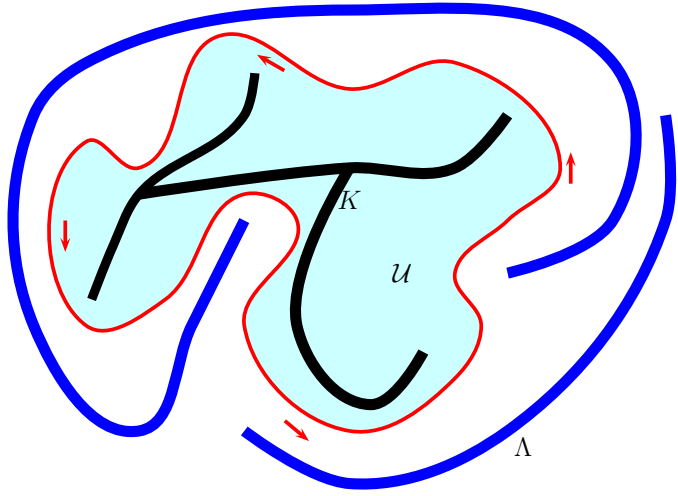

\vskip5cm\hskip8cm
\scalebox{.7}{\psccurve[fillcolor=lightblue,fillstyle=solid,linecolor=red,linewidth=2pt](0,-2)(2,-1)(2.5,0)(2,1)(3,2)(4,3)(2,5)(0,4.5)(-2.5,5.5)(-4,3)(-5,3.5)(-5,0)(-3.5,0.5)(-2,2.5)(-1,2)(-2,0)
\pscurve[linewidth=5pt](-4.9,3)(-4.9,.5)
\pscurve[linewidth=5pt](-4.9,.5)(-4.5,1.5)(-4,2.5)(-2,4)(-1.8,4.8)
\pscurve[linewidth=6pt](-4,2.5)(0,3)(2,3)(3,4)
\pscurve[linewidth=6pt](0,3)(-1,0)(0.4,-1.5)(1.4,-.5)
\rput{90}(4.2,3){\psline[arrows=->,linecolor=red,linewidth=2pt](-.3,0)(.3,0)}
\rput(0,2.4){\Large $K$}
\rput(1,1){\Large $\mathcal U$}
\pscurve[linecolor=blue,linewidth=6pt](3,1)(5,2)(5,5)(0,6)(-5,5)(-6,4)(-6,0)(-4,-2)(-3,0)(-2,2)
\pscurve[linecolor=blue,linewidth=6pt](-2,-2)(0,-3)(3,-2)(6,2)(6,4)
\rput(3.3,-2.3){\Large $\Lambda$}
\rput{150}(-1.5,5.){\psline[arrows=->,linecolor=red,linewidth=2pt](-.3,0)(.3,0)}
\rput{-90}(-5.4,1.7){\psline[arrows=->,linecolor=red,linewidth=2pt](-.3,0)(.3,0)}
\rput{-40}(-1,-2){\psline[arrows=->,linecolor=red,linewidth=2pt](-.3,0)(.3,0)}
}
\vskip3cm
\caption{\small Points of the critical set $\Lambda$ are indicated by black and dark blue colors; $K$
is indicated by black color; a cluster isolating neighborhood $\mathcal U$ for $K$ is marked in light blue color.}\label{fig:cluster-twisted}
\end{figure}

\vs
Take a cluster pair $(K,\mathcal U)$ with a diffeomorphism $\xi :\overline{B_1(0)}
\to \overline {\mathcal U}$. Since $\partial \mathcal U\cap \Lambda=\emptyset$, there exists a sufficiently small $\ve>0$ such that 
\begin{equation}\label{eq:epsilon-cluster}
\{(\lambda,v)\in \mathbb C \times \mathbb E : \mathscr F(\lambda,v)=0,\;   
 \lambda\in \partial \mathcal U, \; \|v\|\le \ve \}\cap \overline{\mathscr S}=\emptyset 
%\forall_{\lambda\in \partial \mathcal U} \;\forall_{v\in \mathbb E}\;\;\; \|v\|<\ve \;; \Rightarrow \;\;\; \mathscr F(\lambda,v)\not= 0
\end{equation}
(cf. \eqref{eq:bif-comp-admiss}).
Put
\begin{equation}\label{eq:cluster-Om}
\Om:=\{(\lambda,v)\in \bc\times \mathbb E: \lambda\in \mathcal U, \; \|v\|<\ve\}
\end{equation}
and take an auxiliary function  $\theta:\bc\times \mathbb E\to \br$ for $\mathscr F$ on $\Om$ %function $\theta:\bc\times \mathbb E\to \br$ 
satisfying
\begin{equation}\label{eq:cluster-theta}
\begin{cases}
\theta(\lambda,0)<0 & \text { for }\;\; \lambda\in \partial \mathcal U,\\
\theta(\lambda,v)>0 &\text{ for } \;\; \lambda\in \overline{\mathcal U}\; \text{ and }\; \|v\|=\ve.
\end{cases}
\end{equation}
Clearly, the map $\mathscr F_\theta (\lambda,v):=\big(\theta(\lambda,v),\mathscr F(\lambda,v)\big)$, $(\lambda,v)\in \bc\times \mathbb E$, is an $\Omega$-admissible $G$-equivariant completely continuous field, thus the twisted $G$-degree $\gdeg(\mathscr F_\theta,\Om)$ is well-defined and is independent of the choice of a cluster isolating neighborhood $\mathcal U$ and a number $\ve>0$. Therefore, one can associate to the cluster pair
 $(K,\mathcal U)$  the element 
 %of $A_1^t(G)$ given by
 \begin{equation}\label{eq:K-loc-bif}
 \omega_G(K):=\gdeg(\mathscr F_\theta,\Om) \in A_1^t(G),
 \end{equation}
and call it the {\it local bifurcation invariant} for the cluster set $K$ of critical points. 
\vs

\begin{remark}\label{rem:smooth-xi}\rm
(i) The construction of invariant \eqref{eq:K-loc-bif} can be similarly performed 
if one replaces $B_1(0)$ in Definition \ref{ded:cluster-twisted} with finitely many disjoined discs
in $\mathbb C$.

(ii) To define  \eqref{eq:K-loc-bif}, it is enough to require that $\xi$ in Definition \ref{ded:cluster-twisted} is just a homeomorphism. However, we require $\xi$ to be differentiable in order to use the linearization techniques relevant to the computation of the twisted degree.
\end{remark}

Following the same lines as in the proof of  Theorem \ref{th:bif-cont-1}, %(see also Theorem \ref{th:bif-cont-0}), 
one can establish the following statement.
%proved by following the same steps as in the proof of 
%Theorem \ref{th:bif-cont-0} (cf. Theorem \ref{th:bif-cont-1}):
%
\vs
\begin{theorem}\label{th:bif-cont-2} Let  $\mathscr F:\br^2\times \mathbb E\to \mathbb E$  
satisfy \ref{bif0}--\ref{bif2}, and let
%be  a completely continuous $G$-equivariant field   satisfying \ref{bif1}--\ref{bif2},  
$(K,\mathcal U)$ be a cluster pair for \eqref{eq:operator-eq} (see Definition \ref{ded:cluster-twisted}).
Assume that $\ve>0$ satisfies \eqref{eq:epsilon-cluster}, $\Om$ is given by \eqref{eq:cluster-Om} and   $  \theta$ is an auxiliary function  for $\mathscr F$ on $\Om$ satisfying   \eqref{eq:cluster-theta}.  Assume that for some $(\mathcal H)\in \Phi_1^t(G)$,  the local bifurcation invariant  $\omega_G(K)$ (given by \eqref{eq:K-loc-bif})  satisfies 
$\text{\rm coeff}^{\mathcal H}(\omega_G(K))\not=0$. Then, there exists a branch $\mathcal C\subset \mathscr S$ of nontrivial solutions such that  for some $(\lambda_o,0)\in K$, one has $(\lambda_o,0)\in 
\mathcal C^{\mathcal H}$.
 
In particular, there exists a non-trivial continuum $\bm C \subset \mathcal C$ with 
$(\lambda_o,0) \in \bm C$ such that for any $ (\lambda,v) \in \bm C$, $v \not= 0$, one has 
$G_{(\lambda,v)} \geq {\mathcal H}$. 
\end{theorem}
\vs
In order to compute   $\omega_G(K)$, one can combine  the linearization argument used in the proof of Proposition \ref{pro:lin-bif-invariant-twisted} with the finite-dimensional reduction provided by Proposition \ref{prop:a-map-finite} (see Appendix \ref{appendix:Ref}). To be more specific, 
take  $\mathscr A:\bc\to \text{L}_c^G(\mathbb E)$ provided by  condition \ref{bif2}, and define  the map $\mathfrak f_{\mathscr A}:\bc \times \mathbb E\to\br\times  \mathbb E$  by
\[
\mathfrak f_{\mathscr A}(\lambda,v):=\big(\theta(\lambda,v),\mathscr A(\lambda)v\big), \quad (\lambda,v)\in \bc \times \mathbb E.
\]
By definition, $\mathfrak f_{\mathscr A}$  is a completely continuous $G$-equivariant field. For a sufficiently small $\ve>0$, the map $\mathfrak f_{\mathscr A}$ is $\Om$-admissibly $G$-equivariantly homotopic (by a linear homotopy) to $\mathscr F_\theta$. Hence, by homotopy property \ref{T2} of the degree, one has
\[
\omega_G(K)=\gdeg(\mathscr F_\theta,\Om)=\gdeg(\mathfrak f_{\mathscr A},\Om).
\]
Define the map $\wt{\mathfrak f_{\mathscr A}}:\bc \times \mathbb E\to\br\times  \mathbb E$ by
\[
\wt{\mathfrak f_{\mathscr A}}(\lambda,v):=\big(\theta(\xi(\lambda),v),\mathscr A(\xi(\lambda))v\big), \quad (\lambda,v)\in \bc \times \mathbb E,
\]
and put
\[
\Om':=\{(\lambda,v)\in \bc \times \mathbb E: |\lambda|<1, \; \|v\|<\ve\}.
\]
Clearly, $\wt{\mathfrak f_{\mathscr A}}$ is an $\Om'$-admissible $G$-equivariant completely continuous field and, since $\xi$ is a diffeomorphism preserving the orientation, one has
\begin{equation}\label{eq:cluster-inv1}
\gdeg(\mathfrak f_{\mathscr A},\Om)=\gdeg(\wt{\mathfrak f_{\mathscr A}},\Om').
\end{equation}
Next, take $1>r_o>0$ such that  $\xi\one (K)\subset B_{r_o}(0):=
\{\lambda\in \bc:|\lambda|<r_o\}$, 
%and define the map $\wt{\mathfrak f_{\mathscr A}}:\bc \times \mathbb E\to\br\times  \mathbb E$, defined by
%\[
%\wt{\mathfrak f_{\mathscr A}}(\lambda,v):=\big(\theta(\xi(\lambda),v),\mathscr %A(\xi(\lambda))v\big), \quad (\lambda,v)\in \bc \times \mathbb E.
%\]
%Put
%\[
%\Om':=\{(\lambda,v)\in \bc \times \mathbb E: |\lambda|<1, \; \|v\|<\ve\}.
%\]
%Notice that $\wt{\mathfrak f_{\mathscr A}}$ is $\Om'$-admissible $G$-equivariant completely continuous field and by a standard argument we have 
%\begin{equation}\label{eq:cluster-inv1}
%\gdeg(\mathfrak f_{\mathscr A},\Om)=\gdeg(\wt{\mathfrak f_{\mathscr A}},\Om').
%\end{equation}
put 
\[
\Om'':=\left\{(\lambda,v)\in \bc \times \mathbb E:\frac{r_o}2< |\lambda|<1, \; \|v\|<\ve\right\}
\]
and define
${\mathfrak f_{a}}:\overline{\Om''}\to\br\times  \mathbb E$ by
\[
{\mathfrak f_{a}}(\lambda,v)=\left( r_o-|\lambda|, a(\lambda)v\right),
\]
where 
\[
a(\lambda) := \mathscr A\left(\xi\left(r_o\frac \lambda{|\lambda|}\right)\right)v.
\]
By applying the same linear homotopy/excision property argument as at the end of the proof of Proposition \ref{pro:lin-bif-invariant-twisted}, one obtains
\begin{equation}\label{eq:Om-Om}
\gdeg(\wt{\mathfrak f_{\mathscr A}},\Om')=\gdeg(\wt{\mathfrak f_{a}},\Om'').
\end{equation}
Finally, by combining \eqref{eq:Om-Om} with the finite-dimensional reduction provided by Proposition \ref{prop:a-map-finite}, one arrives %at the computational formula for the local bifurcation invariant  $\omega_G(K)$ given by the right-hand side of
%\eqref{eq:omega-twisted-final}. 
\begin{equation}\label{eq:omega-twisted-final}
 \omega_G(K)=\prod_l {\bm \delta_l}\cdot \mathop{\sum_{k,l}}\limits_{k>0} d_{k,l}\deg_{\cV_{k,l}},
 \end{equation}
where $d_{k,l}$ is given by \eqref{eq:dkl}
%:=\deg({\det}_\bc(\wt a_{k,l}))$, 
and ${\bm \delta_l}$ is given by \eqref{eq:factor-j-twisted}.
\vs

\section{Equivariant Hopf Bifurcation}
Consider a compact Lie group $\Gamma$, and 
let $H$ be an isometric Hilbert $\Gamma$-representation with the norm  $|\cdot|$ and the  inner product $\langle\cdot,\cdot\rangle$, in $H$. We denote by  $H^c$ the complexification of $H$. 
 \vs
 
Consider  a real unbounded linear operator  $A:{\mathfrak{D}}(A)\subset  H\to  H$ (here ${\mathfrak{D}}(A)$ is $\Gamma$-invariant) and assume $A$ satisfies the following condition: 

\vs
\begin{enumerate}[label=($A$)]\setcounter{enumi}{0}

\item\label{i}  $A:{\mathfrak{D}}(A)\subset  H\to  H$ is a self-adjoint  unbounded $\Gamma$-equivariant operator satisfying
\[ \exists_{\boldsymbol{\delta}>0}\;\forall_{x\in {\mathfrak{D}}(A)} \;\;\; \langle Ax,x\rangle \ge \boldsymbol{\delta}|x|^2.\]

\end{enumerate}
\vs
Notice that under the condition \ref{i}, the operator $A$ is accretive\footnote{According to the terminology used by Brezis \cite{Brezis}, an accretive operator  is called a {\it monotone} operator.}  and  for every $\lambda\in \bc$, $\re \lambda>  -\boldsymbol{\delta}$, one has
\begin{equation}\label{eq:resolv-delta}
\left\|(A+\lambda \id)^{-1}\right\|\le \frac{1}{\boldsymbol{\delta}+\re \lambda}.
\end{equation}

\vs
Under the assumption \ref{i}, the domain  ${\mathfrak{D}}(A)$  equipped with the graph norm, which will be denoted by $V$, is a Hilbert $\Gamma$-representation.  

\vs
 \subsection{Statement of the problem} 
Let  $W$ be a Banach $\Gamma$-representation and  $ \ii : V\to W$  a $\Gamma$-equivariant embedding.  Assume that
\begin{enumerate}[label=($F$)]
\item\label{c0}   $F:\br\times W\to H$  is a continuous $\Gamma$-equivariant map.
\end{enumerate}
The map $F$ can be  viewed as a  family of nonlinear operators $F(\alpha,\cdot ):W\to H$ parametrized by $\alpha\in \br$. 
%boundary valueproblems:
\vs
We consider the following problem
\begin{equation}\label{eq:Hopf}
\begin{cases}
\dot u(t)+Au(t)=F(\alpha,\mathfrak \ii(u(t))), \;\; \text{ for a.e. } t\in \br,\; u(t)\in V,\\
u(t)=u(t+p),
\end{cases} 
\end{equation}
for some (unknown)  $p>0$.
\vs
In addition, we make the following assumptions:
\vs
\begin{enumerate}[label=($E_\arabic*$)]%\setcounter{enumi}
\item\label{c1} There exists an interval 
%$\rho>0$ and a differentiable function $u:U_\rho:=(\alpha_o-\rho,\alpha_o+\rho)\to V^\Gamma$ such that  $u(\alpha_o)=u_o$ and, for all $\alpha \in U_{\rho}$, one has 
${\mathcal I}:=(\alpha_1,\alpha_2)$ such that 
\begin{equation}\label{eq:equilibrium}
%Au(\alpha)=F(\alpha,\ii(u(\alpha)))=0.
F(\alpha,0)=0, \qquad \alpha\in {\mathcal I}.
\end{equation}

\item\label{c2} For any $\alpha\in {\mathcal I}$, there exists $B(\alpha):=D_uF(\alpha,0): W \to  H$ which is a bounded operator, here $D_u$ stands for the Frechet derivative with respect to $u$. Moreover, the map 
$\alpha\mapsto B(\alpha)$ is continuous and 
\begin{equation}\label{eq:Edik}
\lim_{(\alpha',w)\to (\alpha,0)}\frac{\|F(\alpha',w)-B(\alpha)w\|}{\|w\|}=0, \qquad \alpha\in {\mathcal I}.
\end{equation}
\item\label{c3} For any $\alpha\in{\mathcal I}$, the equilibrium $(\alpha,0)$ is {\it non-degenerate}, i.e., 
\begin{equation}\label{eq:non-deg}
\ker\left(A-B(\alpha)\circ \ii \right)=\{0\}.
\end{equation}
%In such a case we call the  family of equilibria $\{(\alpha,u(\alpha))$, $\alpha\in U_\rho\}$, {\it non-degenerate}. 
\end{enumerate} 
\begin{itemize}
\item[$(C)$] The embedding $\mathfrak i:V\hookrightarrow W$  is a compact operator. 
\end{itemize}
%\item\label{c4}
\vs

Constant solutions to  \eqref{eq:Hopf} can be easily identified. More precisely, if a point  $u_o\in V$ satisfies 
\begin{equation}\label{eq:equilibrium-o}
Au_o=F(\alpha_o,\ii(u_o))
\end{equation}
for some $\alpha_o\in \br$, then  $(\alpha_o, u_o)$ is called an {\it equilibrium} of \eqref{eq:Hopf}. It is clear that the constant function $u(t)\equiv u_o$ satisfies    \eqref{eq:Hopf}  for $\alpha=\alpha_o$. 
In what follows, we will use the notation
\begin{equation}\label{eq:equi-set}
\Xi_o:=\{(\alpha_o,u_o)\in \br\times V: Au_o=F(\alpha_o,\ii(u_o))\}
\end{equation}
to denote the set of all equilibria for \eqref{eq:Hopf}. In particular, $(\alpha,0)$ is an equilibrium for every $\alpha\in {\mathcal I}$ according to $(E_1)$. 
\vs
Below, we consider the so-called {\it symmetric Hopf bifurcation} for \eqref{eq:Hopf}, namely  we are interested in the  existence and symmetric properties (local and global) of branches (see subsection \ref{sec:branch}) of {\it non-constant} solutions to \eqref{eq:Hopf} appearing near the equilibrium  
 $(\alpha_o,u_o)$  as the parameter $\alpha$ varies.  
\vs

%\end{enumerate}
\begin{remark}\label{rem:minimal-condit} 
\rm (i) Condition \ref{c2} is required for the existence of the linearization at the equilibrium points $(\alpha,0)$ in order to investigate a possible change of stability of the equilibrium. At the same time, condition  \ref{c3} prevents the bifurcation of steady-state solutions. 

\smallskip
\noi
(ii) Suppose $\Gamma$ contains an element $\gamma_o$ acting antipodally on $V$, 
i.e. $\gamma_ov=-v$ for all $v\in V$. Then, $(\alpha,0)$ is the  only equilibrium  in $V^\Gamma$ for each $\alpha\in \br$, and condition \ref{c1} is clearly satisfied.

\smallskip \noi
(iii) Let us stress that the topological nature of our argument does not require differentiability of $F$ in a deleted neighborhood of  $(\alpha_o,u_o)$. From this viewpoint, conditions \ref{c0} and \ref{c2} indeed seem to be minimal regularity conditions.  
\end{remark}
\vs

\begin{example}\label{ex:parab}\rm Let $\Gamma_1$ and $\Gamma_2$ be compact Lie groups acting orthogonally on $\mathbb R^N$ and $\mathbb R^m$, respectively.
Let $\Omega\subset \br^N$  be  an open  bounded $\Gamma_1$-invariant set with $C^1$-smooth boundary. Put $\Gamma:= \Gamma_1 \times \Gamma_2$. Then, the space
$H:=L^2(\Omega;\br^m)$ admits the structure of an  isometric  Hilbert $\Gamma$-representation given by the formula
\begin{equation}\label{eq:Hilbert-represent}
(\gamma u)(x):= (\gamma_2 u)(\gamma_1^{-1}x) ,\quad x \in \Omega, \; u \in H, \; \gamma := (\gamma_1,\gamma_2)  \in \Gamma. 
\end{equation} 
In general, $\Gamma_1$ (resp. $\Gamma_2$) can act trivially on $\mathbb R^N$ (resp. $\mathbb R^m$). Also, given any reasonable Hilbert (resp. Banach) space $X$ of $\mathbb R^k$-valued  functions on $\Omega \subset \mathbb R^N$, one can use a similar to \eqref{eq:Hilbert-represent} 
formula to define on $X$ a structure of isometric Hilbert (resp. Banach) $\Gamma$-representation. 
 
%We consider the space $H:=L^2(\Omega;\br^m)$, where $\Omega\subset \br^N$  is an open  bounded set of class $C^1$.  
Put ${\mathfrak{D}}(A):=H^1_o(\Omega;\br^m)\cap H^2(\Omega;\br^m)\subset H$ and $Au=-\bm{ \triangle} u$ for $u\in {\mathfrak{D}}(A)$ (here $\bm{\triangle}$ stands for the classical vector Laplace operator)\footnote{In fact, one could consider an arbitrary self-adjoint strongly elliptic second order differential operator on $\Omega$ with Dirichlet or other type of boundary conditions.}. Then, $A$ satisfies condition \ref{i}. 

\vs
Next, define on $\Omega \times \mathbb R^m \times (\mathbb R^m)^N$ the $\Gamma$-action by 
\begin{equation}\label{eq:action-diagonal}
\gamma(x, v_0,v_1,...,v_N) : = (\gamma_1 x, \gamma_2 v_0,\gamma_2 v_1...,\gamma_2 v_N), \qquad x \in \Omega,\; v_i  \in \mathbb R^m, 
\gamma = (\gamma_1,\gamma_2) \in \Gamma,
\end{equation}
and consider a number ${\mathfrak p}>1$ satisfying the following condition (provided that $N>2$):  ${\mathfrak p}\in [1,{\mathfrak p}^*)$, where
${\mathfrak p}^*:=\frac{2N}{N-2}$. If $N = 1,2$, there is no additional condition needed for $\mathfrak p$. Then, one has the 
compact embedding  $\ii : {\mathfrak{D}}(A)\to L^{\mathfrak p}\big(\Omega; \mathbb R^m \times (\br^{m})^N\big)$ given by 
\begin{equation}\label{eq:cemb}
\ii(u)(x)=(u(x),Du(x)), \quad x\in \Omega
\end{equation}
(cf. \cite{Brezis}, Theorem 9.16); here $Du(x)$ stands for the Jacobian matrix of $u$ at $x$.
Put $W:= L^{\mathfrak p}\big(\Omega;\mathbb R^m \times (\br^m)^N\big)$ 
and consider $V:={\mathfrak{D}}(A)$ equipped with the Sobolev norm. Clearly, $\ii:V\to W$ 
satisfies  condition (C).

\vs
Furthermore, consider a (continuously) parameterized family of Carath\'eodory  functions $f_\alpha: \Omega \times \mathbb R^m  
\times (\br^m)^N \to \br^m$,
$\alpha \in \mathbb R$, satisfying the condition
\begin{equation}\label{eq:f-equiv}
f_\alpha (\gamma (x, v)) = \gamma_2 f_\alpha(x, v), \qquad x \in \Omega, \; v := (v_0,v_1,...,v_N) \in \mathbb R^m  \times (\mathbb R^m)^N.
\end{equation}
Assume that for any $\alpha \in \mathbb R$, 
there exist real $a_\alpha$, $b_\alpha>0$  such that   for some $\mathfrak p>1$, one has:
\begin{equation}\label{eq:growth-condition}
\forall_{(x,v)\in \Omega \times \br^m \times (\br^m)^N} \;\; |f_\alpha(x,v)|\le a_\alpha+ b_\alpha |v|^{\frac {\mathfrak p}2}.
\end{equation}
Then, the formula  
\[
F(\alpha,v)(x):= f_\alpha(x,v(x)), \quad  x\in \Omega,
\]
defines a continuous $\Gamma$-equivariant map $F:\br \times L^{\mathfrak p}(\Omega; \mathbb R^m \times (\br^m)^N)\to L^2(\Omega;\br^m)$ (see \cite{Kra}). Thus, $F:\br\times W\to H$ satisfies condition \ref{c0}.

\vs
Assume, in addition,  that for any $\alpha \in \mathbb R$, one has $f_\alpha(x,-v) = -f_\alpha(x,v)$ for any $x \in \Omega$ and 
$v \in  \mathbb R^m \times (\mathbb R^m)^N$. Since, in this case,
$f_{\alpha}(x,0) = 0$ for all $x \in \Omega$, condition \ref{c1} is satisfied with $u_o = 0$. 

\medskip

Finally, suppose that under the above assumptions, $D_vf_\alpha(x,0)$ exists and depends continuously on $\alpha$. Then, condition
\ref{c2} is satisfied.
\end{example}
\vs

\subsection{Characteristic equation}\label{sec:char-eq}
Consider the linearized equation of \eqref{eq:Hopf} at $(\alpha,0)\in \br\times  V^\Gamma$ in the space\footnote{For simplicity, we denote the complexified operator $A$ by the same symbol.} $H^c$, which is
\begin{equation}\label{eq:Linearized}
\dot u(t) +Au(t)=B(\alpha)\mathfrak \ii(u(t)), \quad \text{ a.e. } t\in \br.
\end{equation}
Substituting $u(t):=e^{\lambda t} v$, $\lambda \in \bc$, $v\in D(A)^c$, into \eqref{eq:Linearized}, one obtains the equation
\begin{equation}\label{eq:char1}
\Big(\lambda \id + A  \Big)v=B(\alpha) \ii(v), \quad v\in  D(A)^c.
\end{equation}
Put $\mathbb C_+:=\{z\in \mathbb C: \text{Re\,}z>0\}$ and  $\overline{\mathbb C}_+:=\{z\in \mathbb C: \text{Re\,}z\ge0\}$. Then, for $\lambda\in \overline{\mathbb C}_+$ (see Remark \ref{rem:analyt} below), 
equation \eqref{eq:char1}  can be written as 
\begin{equation}\label{eq:Characteristic}
\triangle_\alpha(\lambda)v:= v-(A+\lambda\id)^{-1} B(\alpha)\ii(v)=0, \qquad v \in V^c.
\end{equation}
Equation \eqref{eq:Characteristic} is called the {\it characteristic equation} for equation \eqref{eq:Hopf} at the equilibrium point $(\alpha,0)$, and $\triangle_\alpha(\lambda)\in L(V^c)$ is  the {\it characteristic operator} at this point. Also, $\lambda_o\in \bc$ is called a {\it characteristic value} for \eqref{eq:Linearized}  at $\alpha=\alpha_o$ if $\ker \triangle_{\alpha_o}(\lambda_o)\not=\{0\}$. 
Clearly, 
$\triangle_\alpha(\lambda)$ is a $\Gamma$-equivariant Fredholm operator of index zero, and for every characteristic value $\lambda_o$ the    {\it characteristic space} 
\[
E(\lambda_o):= \bigcup_{k=1}^\infty \ker (\triangle_{\alpha_o}(\lambda_o))^k
\]
is a finite-dimensional $\Gamma$-invariant subspace of $V^c$. 
Put
\[
m(\lambda_o):=\text{\rm dim\,} (E(\lambda_o))
\]
and call  it the (algebraic)   {\it multiplicity}  of the characteristic value $\lambda_o$ at $\alpha=\alpha_o$. 
\vs 
Consider the $\Gamma$-isotypic decomposition of the Hilbert $\Gamma$-representation $V$, i.e.
\begin{equation}\label{eq:H-isotypic1}
V=V_0\oplus V_1\oplus \dots \oplus V_r,
\end{equation}
where the isotypic component $V_j$ is modeled on the irreducible $\Gamma$-representation $\cV_j$ (here we assume that $\{\cV_j\}_{j=0}^r$ is a  list of distinct real irreducible $\Gamma$-representations with $\cV_0$ being the trivial representation).  
 Since, for every $j=0,1,\dots, r$, the subspace $E_j(\lambda_o):=E(\lambda_o)\cap V_j$ is a finite-dimensional $\Gamma$-subrepresentation of $V$, the integer
 \begin{equation}\label{eq:isot-multi}
 m_j(\lambda_o):=\dim_\bc E_j(\lambda)/\dim_\bc \cV^c_j 
 \end{equation}
 is well defined. We call $m_j(\lambda_o)$ the {\it $\cV_j$-isotypic multiplicity} of the characteristic value $\lambda_o$. 

\vs
Since the map $\triangle_{\alpha_o} :\overline {\mathbb C}_+\to L(V^c)$ is  analytic for every $j=0,1,\dots, r$, the map
\begin{equation}\label{eq:i-char}
\triangle^j_{\alpha_o}(\lambda):=\triangle_{\alpha_o}(\lambda)|_{V_j^c} :\overline {\mathbb C}_+\to L(V_j^c), \quad \lambda\in \overline {\mathbb C}_+,
\end{equation}
is also analytic.

\vs
\begin{definition}\label{def:center}\rm 
(i) Under the assumptions \ref{i},  \ref{c1}--\ref{c3}, \ref{c0} and {\it ($C$)}, an equilibrium $(\alpha_o, 0)\in \br\times V$ is called a {\it center} for \eqref{eq:Hopf} if equation \eqref{eq:Linearized} admits a purely imaginary characteristic value $\lambda_o=i\beta_o$,
$\beta_o>0$, i.e. $ \ker \triangle_{\alpha_o}(i\beta_o)\not=\{0\}$.
\vskip.1cm
\noi(ii)  $(\alpha_o, 0)$ is called an  {\it isolated  center} for \eqref{eq:Hopf}  if   $(\alpha_o, u_o)$ is a center and there exists $\delta>0$  such that for any $\alpha\in [\alpha_o-\delta,\alpha_o+\delta]$, $\alpha\not=\alpha_o$, the equilibrium $(\alpha,0)$ is not a center. 
\end{definition}

\vs
Assume that
\vs
\begin{enumerate}[label=($E_\arabic*$)]\setcounter{enumi}{3}
\item\label{c4} $(\alpha_o, 0)$ is an isolated center for \eqref{eq:Hopf}.
\end{enumerate}

\vs
\begin{lemma}\label{lem:iso-cent}
(i) Suppose that conditions  \ref{i}, \ref{c1}--\ref{c2}, \ref{c0}, {\it ($C$)} are satisfied and the operator $B(\alpha_o)\ii $ has a continuous extension to $H$  (denoted by $\mathscr B(\alpha_o)$). 
Then, for $\alpha_o\in{\mathcal I}$, the set 
\begin{equation}\label{eq:set-beta-alpha-o}
\bm b(\alpha_o):=\big\{\beta>0: 
 \ker \triangle_{\alpha_o}(i\beta)\not=\{0\}\big\}
\end{equation}
is  bounded.  

\medskip\noi
(ii) If, in addition,  $(\alpha_o,0)$  is non-degenerate, then 
 the set $\bm b(\alpha_o)$ is  compact in $\br$.
\end{lemma}
\begin{proof} 
(i) By Remark \ref{rem:analyt} (see Appendix, section \ref{sec:isya}) and condition (C), the operator $ A^{-1}\in L(H)$ is bounded and self-adjoint, which implies that 
$\sigma(A)\subset \br.$ 
Then, for every $\lambda\in \bc$ such that 
\begin{equation}\label{eq:spectrum-far}
\text{dist\,}(-\lambda,\sigma(A))>\|\mathscr B(\alpha_o)\|,
\end{equation}
one has that $(A+\lambda \id-\mathscr B(\alpha_o) )^{-1}:H^c\to H^c$ is well-defined (see Kato \cite{Kato}, Chapter V, section 4, Problem 4.8), which implies 
that $A+\lambda \id-\mathscr B(\alpha_o) $ is injective. On the other hand, 
\[
\beta\in \bm b(\alpha_o) \;\;\Leftrightarrow\;\;  \ker \triangle_{\alpha_o}(i\beta)\not=\{0\}\;\; \Leftrightarrow\;\; \ker(A+\lambda \id-\mathscr B(\alpha_o))\not=\{0\},
\]
hence, if  $\beta\in \bm  b(\alpha_o) $,  then  $A+\lambda \id-\mathscr B(\alpha_o)$ is not invertible. Consequently (cf.\eqref{eq:spectrum-far}), if  
$\beta\in \bm b(\alpha_o) $, then
$\text{dist\,}(-i\beta,\sigma(A))\le \|\mathscr B(\alpha_o)\|$, which implies
  $\beta\le \|\mathscr B(\alpha_o)\|$, so the set $\bm b(\alpha_o)$ is bounded.

\vs
(ii)%It follows from (i) that the set $\bm b  (\alpha_o)$ is bounded.
 Let us show that $\bm b (\alpha_o)$ is also closed. Indeed, assume that $\{\beta_n\}\subset \bm b (\alpha_o)$ and $\beta_n\to \beta_o$ as $n\to\infty$.  Suppose, without loss of generality, that $v_n\in D(A)$, $|v_n|=1$, is such that 
\[
i\beta_nv_n -Av_n=B(\alpha_o)\ii(v_n).
\]
Since $\ii$ is a compact operator, we can assume without loss of generality that $\ii(v_n)\to w_o$ as $n\to \infty$, therefore 
\[
v_n=(i\beta_n  \id-A)^{-1} B(\alpha_o)\ii(v_n)\to v_o:=(i\beta_o \id - A)^{-1}B(\alpha_o) w_o,
\]
which implies that 
\[
i\beta_ov_o-Av_o=B(\alpha_o)\ii(v_o), \quad |v_o|=1.
\]
Since $(\alpha_o,0)$ is non-degenerate, one has $0 \not=\beta_o \in \bm b (\alpha)$
and the statement follows.
\end{proof}

\vs
Motivated by Lemma \ref{lem:iso-cent}, we introduce the following condition:
\begin{itemize}
\item[($\mathscr B$)] The set $ \bm b(\alpha_o)$, given by \eqref{eq:set-beta-alpha-o}, is bounded.
\end{itemize}

\vs
\begin{definition}\label{def:crit-freq} \rm 
Under the assumptions  \ref{i}, \ref{c1}--\ref{c2}, \ref{c0}  and ($C$), elements of
$\bm b(\alpha_o)$ are called {\it critical frequencies}. Also, we say that two different critical frequencies  in $\bm b(\alpha_o)$ are {\it resonant} if one is an integer multiple of the another. 
\end{definition}

\vs
\begin{lemma}\label{lem:co-finite}
Let the assumptions   \ref{i}, \ref{c1}--\ref{c2}, \ref{c0}, ($C$) hold.  
If for $\alpha_o \in {\mathcal I}$ condition   $(\mathscr B)$ is satisfied, then  the set $\boldsymbol{\bm b }(\alpha_o)$ (see \eqref{eq:set-beta-alpha-o})
is finite.
\end{lemma}

\begin{proof} First, notice that  $\triangle_{\alpha_o}(\lambda)$ is analytic with respect to $\lambda$ with $\im(\lambda) > -\boldsymbol{\delta}$
(see Remark \ref{rem:analyt}). 
Take $\beta\in \boldsymbol{\bm b }(\alpha_o)$. Since  $\triangle_{\alpha_o}(i\beta)$ is a Fredholm operator, $\ker \triangle_{\alpha_o}(i\beta)$ is finite-dimensional. 
Put $W_o:=\ker \triangle_{\alpha_o}(i\beta)$ and   $V_o:= \big(W_o)^\perp$  and take the orthogonal projection $P:V\to V$ onto $V_o$. Thus $P\circ  \triangle_{\alpha_o}(i\beta):V_o\to V_o$ is an isomorphism. Then, by continuity of   $\triangle_{\alpha_o}(\cdot)$,   there exists a neighborhood $\mathcal O\subset \overline{\bc}_+$ of $i\beta$ such that for all $\lambda\in \mathcal O$,  $P \circ  \triangle_{\alpha_o}(\lambda):V_o\to V_o$ is also an isomorphism.    Since for $\lambda\in \mathcal O$, the equation 
$\triangle_{\alpha_o}(\lambda)v=0$, $v\in V$, can be written (by taking $v=v_o+w_o\in V_o\oplus W_o$) as
\[
\begin{cases}
P\circ \triangle_{\alpha_o}(\lambda)v_o=-P\circ \triangle_{\alpha_o}(\lambda)w_o\;\;\Leftrightarrow\;\; v_o:=-(P\circ \triangle_{\alpha_o}(\lambda))^{-1}\circ P\circ \triangle_{\alpha_o}(\lambda)w_o,\\
(\id-P)\circ \triangle_{\alpha_o}(\lambda)(v_o+w_o)=0,
\end{cases}
\]
one obtains that $\lambda\in \mathcal O$ is a characteristic value if and only if there exists a non-zero vector $w_o\in W_o$ such that 
\[
(\id-P)\circ \triangle_{\alpha_o}(\lambda)\big(w_o-(P\circ \triangle_{\alpha_o}(\lambda))^{-1}\circ P\circ \triangle_{\alpha_o}(\lambda)w_o\big)=0.
\]
Choose a unitary basis in $W_o$ and  identify $W_o$ with a complex space $\bc^N$ (for some $N\in \bn$), thus  the operator  $ (\id-P)\circ \triangle_{\alpha_o}(\lambda)\big(\id-(P\circ \triangle_{\alpha_o}(\lambda))^{-1}\circ P\circ \triangle_{\alpha_o}(\lambda))\big)$ can be represented as an $N\times N$-matrix. Consequently, one has that $\lambda\in \mathcal O$ is a characteristic value if and only if 
\[ \phi(\lambda):={\det}_\bc \left[  (\id-P)\circ \triangle_{\alpha_o}(\lambda)\big(\id-(P\circ \triangle_{\alpha_o}(\lambda))^{-1}\circ P\circ \triangle_{\alpha_o}(\lambda))\big) \right]=0.
\]
Since $\triangle_{\alpha_o}(\cdot)$ is analytic  and $\phi(i\beta)=0$, it follows that $i\beta$ is an isolated zero of $\phi$. Therefore, the conclusion follows from Lemma \ref{lem:iso-cent}(ii).
\end{proof} 
\vs

\subsection{Setting in functional spaces} Put $G := \Gamma \times S^1$, where $S^1:= \{e^{i\theta} \in \mathbb C \, : \, \theta \in \mathbb R\}$, and
%Based on the notation introduced in \eqref{eq:per-spaces}, put
\begin{equation}\label{eq:spaces2}
\mathscr E := H^1_{2\pi}(\mathbb R; H)\cap  L^2_{2\pi}(\mathbb R; V),  \quad \mathscr H := L^2_{2\pi}(\mathbb R; H), \quad \mathscr W :=C_{2\pi}(\br;W)
\end{equation}
(here the lower index $2\pi$ is omitted for simplicity).
Each of the spaces $\mathscr E$, $\mathscr H$ and $\mathscr W$ is an isometric $G$-representation with the actions given by
\begin{equation}\label{eq:Gamma-times-S1}
\big((\gamma, e^{i\theta})v\big)(t) := \gamma v(t + \theta), \quad \gamma \in \Gamma, \quad e^{i\theta} \in S^1. 
\end{equation}  
By the Arzela-Ascoli theorem (see also condition (C)), the natural embedding 
\begin{equation}\label{eq:embedding-j}
	\bm{j} : \mathscr E \to \mathscr H
\end{equation}
is a compact operator.

\vs
By using the standard time rescaling and the substitution $\beta:= {2\pi \over p}$, one arrives at the two-parameter problem
\begin{equation}\label{eq:Hopf1}
\begin{cases}
\dot{u} + {1 \over \beta} Au(t) = {1 \over \beta}F(\alpha,\ii(u(t))), \;\; \text{ for a.e. } t \in \br,\\
u(t) = u(t + 2\pi),
\end{cases}
\end{equation}
equivalent to \eqref{eq:Hopf}. Then, using the Isomorphism Theorem  (see Appendix, Theorem \ref{th:iso-thm})  and the functional spaces $\mathscr E$, $\mathscr W$  and 
$\mathscr H$, one can reformulate problem \eqref{eq:Hopf1} as an operator equation in the functional space $\mathscr{E}$. First, define the operator
$\L_{\beta}:  \mathscr{E}   \to \mathscr H$, $\beta>0$,   by 
\begin{equation}\label{eq:L-beta}
\L_{\beta}u := \frac{d}{dt} u + {1 \over \beta} Au 
\end{equation}
and the operator $N_F:\br\times \mathscr W\to \mathscr H$ by
\begin{equation}\label{eq:NF}
N_F(\alpha,v)(t):= F(\alpha,v(t)), \quad t\in \br,\; v\in \mathscr W.
\end{equation}
Then, problem \eqref{eq:Hopf1} can be written as
\begin{equation}\label{eq:operat1}
\L_{\beta}u = {1 \over \beta} N_F(\alpha,\bm{j}(u)),  \quad u\in \mathscr{E},
\end{equation}
where both operators $\L_\beta$ and $N_F$ are $G$-equivariant.
Since, by the Isomorphism Theorem, the operator $\L_{\beta}$ 
is an isomorphism, one obtains that \eqref{eq:operat1} is equivalent to
\begin{equation}\label{eq:operat2}
u-\frac1 \beta  \L_{\beta}^{-1}N_F(\alpha,\bm{j}(u))=0, \quad u\in \mathscr{E}.
\end{equation}

\vs
We put $\br^2_+:=\br\times \br_+$, where $\br_+:=\{t\in \br: t>0\}$, and consider it as a metric space equipped with the metric
\begin{equation}\label{eq:metric-R2+} d((\alpha_1\beta_1),(\alpha_2,\beta_2)):= \sqrt{(\alpha_2-\alpha_1)^2+\ln^2\frac{\beta_1}{\beta_2}}.\end{equation}
The map 
 $\mathscr F: \br_+^2\times \mathscr E\to \mathscr E$, given by
\begin{equation}\label{eq:F}
\mathscr F(\alpha,\beta, u):= u-\frac1 \beta \L_{\beta}^{-1} N_F(\alpha,\bm{j}(u)), \quad u\in \mathscr E, 
\end{equation}
is a $G$-equivariant completely continuous field (recall 
$\bm{j}$ is compact).
Consequently, problem \eqref{eq:Hopf} is equivalent to the $G$-equivariant operator equation
\begin{equation}\label{eq:Hopf2}
\mathscr F(\alpha,\beta,u)=0, \quad (\alpha,\beta,u)\in \br_+^2\times \mathscr E.
\end{equation}
 
 \vs
 Under the condition \ref{c1}, for all $(\alpha,\beta)\in {\mathcal I} \times \br_+$,  one has 
\[
\mathscr F(\alpha,\beta,0)=0,
\]
i.e. the set 
\[
\mathscr M:=\{(\alpha,\beta,0):  \;\alpha\in {\mathcal I}, \;  \; \beta>0\}
\]
 is contained in the solution set of equation \eqref{eq:Hopf2}.  We call  points from $\mathscr M$ {\it trivial solutions} to \eqref{eq:Hopf2}. 

\vs
Consider a point $(\alpha,\beta,0)\in \mathscr M$. Then,  $\mathscr A(\alpha,\beta):=D_u\mathscr  F(\alpha,\beta, 0):\mathscr E\to \mathscr E$ exists and 
\begin{equation}\label{eq:DerF}
D_u\mathscr  F(\alpha,\beta, 0)v= v-\frac1 \beta \L_{\beta}^{-1} N_{ B(\alpha)}(\bm{j}(v)), \quad v\in \mathscr E,
\end{equation}  
where $N_{B(\alpha)}: \mathscr W\to \mathscr H$ is defined by 
\[
N_{B(\alpha)}(v)(t)= B(\alpha)v(t), \quad v \in \mathscr W.
\]
Since % $0\in V^\Gamma$, one has 
$G_{0}=G$, the linear operator $\mathscr A(\alpha, \beta)$ is $G$-equivariant.  
Consider the $\Gamma$-isotopic decomposition \eqref{eq:H-isotypic1}.
Using the standard Fourier modes decomposition of the space  $\mathscr E$, one has the following $G$-isotypic decomposition:
\begin{equation}\label{eq:iso-E} 
\mathscr E=\overline{\bigoplus_{k=0}^\infty \bigoplus_{j=0}^r\mathscr E_{k,j}}, 
\end{equation}
where $\mathscr E_{0,j}:=V_j$ stands for the subspace of  constant $V_j$-valued functions and 
\begin{equation}\label{eq:iso-E-kj}
\mathscr E_{k,j}:=\left\{\vp:\br\to V_j\;: \;\vp(t)= \cos(kt)u+\sin(kt)v ,\; u,\, v\in V_j, \; t\in \br  \right\}.
\end{equation}
Notice that  for $k>0$, 
the  component  $\mathscr E_{k,j}$ 
is modeled on the (real) irreducible $G$-representation $\cV_{j,k}$ which is the complexification $\cV^c_j$ of $\cV_j$ with the $S^1$-action given by $k$-folding, i.e. 
\begin{equation}\label{eq:Vjk}
e^{i\theta}w:=e^{ik\theta}\cdot w, \quad e^{i\theta}\in S^1, \; w\in \cV_j^c
\end{equation}
(here ``$\cdot$" stands for the complex multiplication).
Moreover, 
the subspace $\mathscr E_{k,j}$  ($k>0$) is $G$-equivariantly isomorphic to the space 
\[\mathfrak V_j:=\{\psi:\br\to V_j^c\; :\; \psi(t):=e^{ikt}w,\;\; w\in V_j^c, \; t\in \br\},\]
 with the $G$-action given by
\[
(\gamma,e^{i\theta})(e^{ikt}w)=e^{ikt}\gamma (e^{-ik\theta}\cdot w), \quad (\gamma,e^{i\theta})\in \Gamma\times S^1.
\]
Indeed, define the $\Gamma$-equivariant operator $\vp:\mathfrak V_j\to\mathscr E_{k,j}$ by 
\begin{equation}\label{eq:complex-str}
\vp(e^{ikt}(u+iv))=\cos(kt)u+\sin(kt)v.
\end{equation}
From the relations 
\begin{align*}
\vp(e^{i\theta}e^{ikt}(u+iv))&=\vp(e^{ikt}(\cos(k\theta)-i\sin(k\theta))(u+iv))\\
&=\vp(e^{ikt}(\cos(k\theta)u+\sin(k\theta)v)+i(\cos(k\theta)v-\sin(k\theta)u))\\
&=\cos(kt)(\cos(k\theta)u+\sin(k\theta)v)+\sin(kt)(\cos(k\theta)v-\sin(k\theta)u)\\
&=(\cos(kt)\cos(k\theta)-\sin(kt)\sin(k\theta))u+(\cos(kt)\sin(k\theta)+\sin(kt)\cos(k\theta))v\\
&= \cos(k(t+\theta))u+\sin(k(t+\theta))v\\
&=e^{i\theta}(\cos(kt)u+\sin(kt)v)\\
&=e^{i\theta}\vp(e^{ikt}(u+iv)),
\end{align*}
it follows that $\vp$ is a $G$-equivariant isomorphism. One can easily notice that $\mathfrak V_j$ is equivalent to the $G$-representation $V_j^c$ equipped with the $k$-folded $S^1$-action, i.e. $e^{i\theta}w:=e^{-ik\theta}\cdot w$, $w\in V^c_j$. 

\vs
Since $\mathscr A(\alpha,\beta)$ is $G$-equivariant, one has $\mathscr A(\alpha, \beta)(\mathscr E_{k,j})\subset \mathscr E_{k,j}$, $k=0,1,2,\dots$, $j=0,1,\dots,r$,  so one can introduce the following notation:
\begin{equation}\label{eq:Ak}
\mathscr A_{k,j}(\alpha,\beta):=\mathscr A(\alpha, \beta)|_{\mathscr E_{k,j}}:\mathscr E_{k,j}\to \mathscr E_{k,j}.
\end{equation}
\vs

\begin{lemma}\label{lem:A_k} Under the assumptions \ref{i}, ($C$), \ref{c0}, \ref{c1}---\ref{c3}, one has for all $k=1,2,\dots$\ and $j=0,1,\dots, r$ (see \eqref{eq:i-char}):
\begin{equation}\label{eq:Ch-k}
\mathscr A_{k,j}(\alpha,\beta)=\triangle^j_{\alpha}(ik\beta):V_j^c\to V_j^c, \quad \alpha\in \br, \; \beta\in \br_+.
\end{equation}
Moreover, for $k=0$, 
\[
\mathscr A_{0,j}(\alpha,\beta)= \id-A^{-1}B(\alpha)\ii|_{V_j}:V_j\to V_j, \quad j=0,1,2,\dots,r.
\]
\end{lemma}
\vs

\begin{proof} Notice that for $\psi \in \mathscr E_{k,j}$, i.e. $\psi(t)=e^{ikt}w$, $w\in V_j^c$, 
\begin{align*}
\mathscr A_{k,j}(\alpha,\beta)\left( \psi \right)(t)&= \psi(t)-\frac1 \beta \L_{\beta}^{-1}  B(\alpha)\ii(\psi)(t)\\
&= e^{ikt}w-\frac1 \beta \L_{\beta}^{-1} B(\alpha)\ii e^{ikt}w\\
&=e^{ikt}w-\frac1 \beta \left(\frac{d}{dt}+\frac 1\beta A\right)^{-1} B(\alpha)\ii e^{ikt}w\\
&=e^{ikt}\left(w-\frac1 \beta \left(ik\id  +\frac 1\beta A\right)^{-1} B(\alpha)\ii w\right)\\
&=e^{ikt}\left(w-\left(ik\beta\id +A\right)^{-1} B(\alpha)\ii w\right)\\
&=e^{ikt} \triangle_{\alpha}^j(ik\beta)w =\triangle^j_{\alpha}(ik\beta)(\psi)(t) .
\end{align*} 
Therefore, under the identification $\mathscr E_{k,j}=V_j^c$, one obtains that $\mathscr  A_{k,j}(\alpha,\beta)=\triangle^j_{\alpha}(ik\beta)$. The case $k=0$ is obvious.
\end{proof}
\vs

If $(\alpha_o,0)$  is non-degenerate, then  Lemma \ref{lem:A_k} implies that $\mathscr A(\alpha,\beta)$ is not an isomorphism if and only if $(\alpha,0)$ is a center with a critical frequency $\beta$ (see Definition \ref{def:crit-freq}).
\vs

\subsection{Branches of non-constant periodic solutions}\label{sec:branch}
Below we %provide a  definition of a {\it branch} (cf. \cite{Kra})  that will  be used for describing properties of  non-constant periodic solutions bifurcating from the set of trivial solutions.
adapt Definition \ref{def:branch-1} of a branch of non-trivial solutions and a bifurcation point.

\vs
Given problem \eqref{eq:Hopf} (see also \eqref{eq:Hopf2}), define the set 
\begin{equation}\label{eq:setS}
\mathscr P:=\{ (\alpha,\beta,u)\in \br_+^2\times \mathscr E: \mathscr F(\alpha,\beta,u)=0, \; u\notin V\}.
\end{equation}
The set $\mathscr P$ consists of non-constant solutions to \eqref{eq:Hopf2}. 
\vs
%In order to avoid misunderstanding, we 
We denote by $\overline A$ the closure  of a set $A\subset  \br_+^2 \times \mathscr E$, where $\br_+^2$ is considered to be equipped with the metric $d$ given by \eqref{eq:metric-R2+}.
\vs

%Notice that under the assumptions  (C), \ref{c0},  \ref{i} and   \ref{c1}---\ref{c2}, every branch $\mathscr C$ of  $\overline{\mathscr P}$ contains a non-trivial continuum.

\vs
\begin{definition}\label{def:Hopf}\rm Under the assumptions  \ref{i}, ($C$),   \ref{c0}, \ref{c1}---\ref{c3},
	a set $\mathscr C\subset \mathscr P$ (cf. \eqref{eq:setS}) is called a {\it branch} of non-constant solutions to  \eqref{eq:operator-eq}  if $\overline{\mathscr C}$ is a connected component of $\overline{\mathscr P}$ containing a nontrivial continuum. Moreover, if $(\alpha_o,\beta_o,0)\in \overline{\mathscr C}\cap \mathscr M$,
	 we say that system  \eqref{eq:Hopf} undergoes a {\it Hopf bifurcation}  
at the equilibrium  $(\alpha_o,0)$ with the {\it limit frequency}  $\beta_o>0$ and that  $(\alpha_o,\beta_o,0) $ is  a {\it Hopf bifurcation point} for \eqref{eq:Hopf2}.
\end{definition}
\vs
The following statement  is sometimes called the {\it necessary condition for Hopf bifurcation}.
\vs

\begin{proposition}\label{pro:necessary}
Suppose conditions  \ref{c0}, ($C$), \ref{i}, \ref{c1}---\ref{c3} are satisfied and $(\alpha_o,\beta_o,0)$ ($\beta_o>0$) is a  Hopf bifurcation point for \eqref{eq:Hopf2}.
 Then, $\mathscr A(\alpha_o,\beta_o):\mathscr E\to \mathscr E$ is not an isomorphism, and consequently, $(\alpha_o,0)$ is a center for  \eqref{eq:Hopf} with a characteristic value $i\beta_o k$ for some $k\in \bn$.
\end{proposition}
\begin{proof} From Definition \ref{def:Hopf}, one can deduce that if  $(\alpha_o,\beta_o,0)$ is a Hopf bifurcation point, then  the operator $\mathscr A(\alpha_o,\beta_o)=D_u\mathscr F(\alpha_o,\beta_o,0):\mathscr E\to \mathscr E$ cannot be an isomorphism. Thus,  for some $k\in \bn$  and $j=0,1,\dots, r$, the operator $\mathscr A_{k,j}(\alpha_o,\beta_o)$ is not an isomorphism, i.e. 
\[\ker  \mathscr A_{k,j}(\alpha_o,\beta_o)=\ker \triangle^j_{\alpha_o}(ik\beta_o)\not=\{0\},
\]
which implies that $(\alpha_o, 0)$ is a center with the characteristic value $ik\beta_o$.
\end{proof} 
\vs

Let $k\in\bn$ and assume that $(\alpha,\beta,u)\in \mathscr P$. Define $u^k(t):=u(kt)$. Then,  $u^k\in \mathscr E$ ($u^k$ could be called a {\it shadow}  of $u$) and $(\alpha,\frac \beta k, u^k)\in \mathscr P$. Therefore, one can observe that if $(\alpha_o,\beta_o,0)$ is a Hopf bifurcation point, then for every $k\in \bn$, the point $(\alpha_o,\frac{\beta_o}k,0)$ is also a Hopf bifurcation point. 
\vs

For $k\in \bn$, define the map
$\Bcyr_k:\br_+^2\times \mathscr E\to \br_+^2\times \mathscr E$ by
\begin{equation}\label{eq:Bcyr-map}
\Bcyr_k(\alpha,\beta,u):= \left(\alpha,\tfrac \beta k,u^k\right), \quad u^k(\cdot):=u(k\cdot).
\end{equation}
We also put $ \bcyr _k(u):=u^k$. Clearly, $\Bcyr_k$ is a continuous map.
\vs

\begin{lemma}\label{lem:fold} Suppose that $u\in \mathscr E$ is a non-constant $2\pi$-periodic function. Then,  
\[
\lim_{k\to \infty} \|\bcyr _k(u)\|=\infty.
\]
\end{lemma} 
\begin{proof}
One has
\begin{align*}
 \|\bcyr _k(u)\|^2_{L^2} & = \int_0^{2\pi} |u(kt)|^2dt= \frac 1k \int_0^{k2\pi}|u(s)|^2ds=\int_0^{2\pi}|u(s)|^2ds,\\
  \left\|\frac d{dt} \bcyr _k(u)\right\|^2_{L^2} & =\int_0^{2\pi} k^2|\dot u(kt)|^2dt= k \int_0^{k2\pi}|\dot u(s)|^2ds=k^2\int_0^{2\pi}|\dot u(s)|^2ds,
\end{align*}
which implies
\[
 \|\bcyr _k(u)\|=\sqrt{\|u\|^2_{L^2}+k^2\|\dot u\|^2_{L^2}}.
\]
Since $u$ is non-constant, $\|\dot u\|_{L^2}\not=0$, and $k^2\|\dot u\|^2_{L^2}\to \infty$ as $k\to \infty$, thus the conclusion follows.

\end{proof} 
\vs

\begin{proposition} Suppose that conditions  \ref{i}, ($C$), \ref{c0},  \ref{c1}---\ref{c3} are satisfied and assume that $(\alpha_o,\beta_o,0)\in \mathscr C$, where $\mathscr C$ is a branch. If 
\begin{equation}\label{eq:reson-branch} 
\Bcyr_k(\overline{\mathscr C})\cap \overline{\mathscr C}\not=\emptyset\quad \text{ for some $k>1$}
\end{equation}
then:\\
(i) $\Bcyr_{k^m}(\overline{\mathscr C})\subset \overline{\mathscr C} $ for all $m\in \bn$, \\
(ii) $(\alpha_o,0,0)$ belongs to the closure  of $\mathscr C$ in $\br\times \br\times \mathscr E$ (see Figure \ref{fig:fold}), and\\
(iii) the branch $\mathscr C $ is unbounded in $\br^2_+\times \mathscr E$.
\end{proposition}
\begin{proof}
(i) Assume that 
$\Bcyr_k(\overline{\mathscr C})\cap \overline{\mathscr C}\not=\emptyset$. Then, iterating
\eqref{eq:Bcyr-map}, one obtains for any $m \geq 1$:
\begin{equation}\label{eq:filtration}
\Bcyr_{k^m}(\overline{\mathscr C}) \subset \Bcyr_{k^{m-1}}(\overline{\mathscr C}) \subset \dots \subset \Bcyr_k(\overline{\mathscr C})\subset \overline{\mathscr C}.
\end{equation}  

(ii) Follows from \eqref{eq:filtration}.

(iii) Follows from Lemma \ref{lem:fold}.
\end{proof}
\vs

\begin{remark}\rm Notice that condition \eqref{eq:reson-branch} implies that  for two critical resonant frequencies $\beta$, $\beta'$ for a fixed  $\alpha_o$, the  points $(\alpha_o,\beta,0)$ and $(\alpha_o, \beta',0)$ belong to the branch $\mathscr C$.
\end{remark}
\vs

\begin{figure}[H]
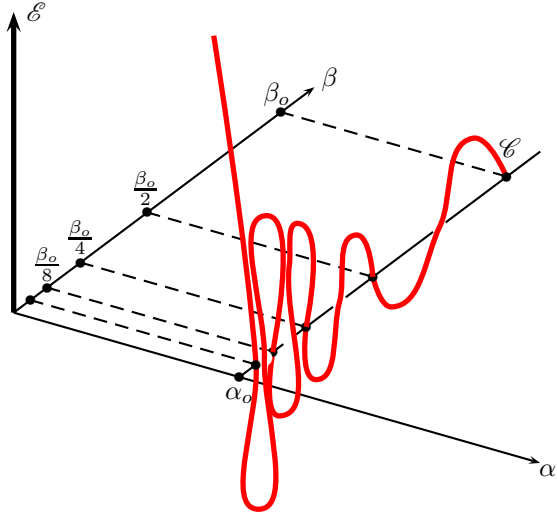

\vskip4cm\hskip4cm
\psline{->}(0,0)(7,-2)
\psline{->}(0,0)(4,3)
\psline[linewidth=2pt]{->}(0,0)(0,4)
\rput{180}(6.550,1.8114){\psline[linestyle=dashed](0,0)(3,-.85714)\psdot(3,-.85714)}
\rput{180}(4.775,0.47619){\psline[linestyle=dashed](0,0)(3,-.85714)\psdot(3,-.85714)}
\rput{180}(3.8875,-.190475){\psline[linestyle=dashed](0,0)(3,-.85714)\psdot(3,-.85714)}
\rput{180}(3.44375,-0.5238075){\psline[linestyle=dashed](0,0)(3,-.85714)\psdot(3,-.85714)}
\rput{180}(3.22187,-0.690505){\psline[linestyle=dashed](0,0)(3,-.85714)\psdot(3,-.85714)}
\rput(3,-.85714){\psline(0,0)(4,3)}
\rput(4.775,0.47619){\psdots[dotsize=6pt,linecolor=white](0.888,0.6666)\pscurve[linecolor=red,linewidth=2pt](0,0)(0.244,-0.4)(0.888,0.6666)(1.332,1.84465)(1.775,1.3333)}
\psdots(6.55,1.80952)(4.775,0.47619)
\rput(3.8875,-.190475){\scalebox{.5}{\psdots[dotsize=12pt,linecolor=white](0.888,0.6666)\pscurve[linecolor=red,linewidth=4pt](0,0)(0.244,-1.4)(0.888,0.6666)(1.332,2.4465)(1.775,1.3333)}}
\psdot(3.8875,-.190475)
\rput(3.44375,-0.5238075){\scalebox{.25}{\psdots[dotsize=24pt,linecolor=white](0.888,0.6666)\pscurve[linecolor=red,linewidth=8pt](0,0)(0.344,-3.4)(0.888,4.6666)(1.332,6.84465)(1.775,1.3333)}}
\psdot(3.44375,-0.5238075)
\rput(3.22187,-0.690505){\scalebox{.125}{\psdots[dotsize=48pt,linecolor=white](0.888,0.6666)\pscurve[linecolor=red,linewidth=16pt](-4.5,35)(0,0)(0.244,-15.4)(0.888,0.6666)(1.132,15.84465)(1.775,1.3333)}}
\psdots(3.22187,-0.690505)(3,-.85714)
\rput(0.3,4){$\mathscr E$}
\rput(7.1,-2.1){$\alpha$}
\rput(4.2,3.1){$\beta$}
\rput(3.0,-1.1){$\alpha_o$}
\rput(3.5,2.9){$\beta_o$}
\rput(1.75,1.65){$\frac{\beta_o}2$}
\rput(.88,1){$\frac{\beta_o}4$}
\rput(.44,.65){$\frac{\beta_o}8$}
\rput(6.6,2.2){$\mathscr C$}
\vskip2.3cm
\caption{Development of a ``folding'' branch $\mathscr C$.}\label{fig:fold}
\end{figure}

In what follows, we are interested in studying properties of  branches $\mathscr C\subset {\mathscr P}$ bifurcating from a point $(\alpha_o,\beta_o,0)$, where $\beta_o\in \boldsymbol{\bm b }(\alpha_o)$ (cf. Lemma \ref{lem:co-finite}). To this end, 
let us introduce the following additional condition on the map $F$:
\vs

\begin{enumerate}[label=($\mathscr L$)]%\setcounter{enumi}{3}
\item\label{lip} The map $F:\br\times W\to H$ is {\it locally Lipschitz}  with respect to $v$, i.e for every $(\alpha_o,v_o)\in \br\times W$, there exist $\ve>0$ and $L>0$ such that for all $\alpha$ and $v$, $v'\in W$ with $ |\alpha-\alpha_o|<\ve$, $ \|v-v_o\|<\ve $, $\|v'-v_o\|<\ve$, one has: 
\begin{equation}\label{eq:Vid1}
 \|F(\alpha,v)-F(\alpha,v')\|\le L\|v-v'\|.
\end{equation}
\end{enumerate}
\vs 

Lemmas \ref{lem:Vidossich}--\ref{lem:Vid2} following below play an important role in our considerations. In particular, they provide a link between condition 
\ref{c4} and estimates for the period of a periodic solution to \eqref{eq:Hopf}. 
\vs
\begin{lemma}\label{lem:Vidossich}{\rm (G. Vidossich \cite{Vidossich})} Let $\mathbb E$ be a Banach space and $v:\br\to \mathbb E$ a $p$-periodic function ($p>0$) such that:
\begin{itemize}
\item[(a)] the function $v$ is integrable and $\int_0^p v(t)dt=0$;
\item[(b)] there is $U\in L^1([0,\tfrac p2];\br)$ such that
\[
\|v(t)-v(s)\|\le U(t-s) \quad \text{ for  a.e. }\;\; t- s\in [0,\tfrac p2],\; 0\le s\le t\le p.
\]
\end{itemize}
Then,
\begin{equation}\label{eq:Vid}
p\sup_{t\in \br} \|v(t)\|\le 2\int_0^{\frac p2} U(t)dt.
\end{equation}
\end{lemma}
\vs

\begin{lemma} \label{lem:Vid1}  Suppose that conditions  \ref{i}, ($C$), \ref{c0}, \ref{c1}---\ref{c3} and \ref{lip} are satisfied. Let   $u:\br\to V$ be a $p$-periodic solution ($p>0$) to \eqref{eq:Hopf} for some $\alpha\in \br$.  Put $r:=\max_{t\in \br} |u(t)|$ and assume that the following condition is satisfied:
\[
\exists_{L>0} \; \forall_{v,v'\in V} \; \|v\|,\, \|v'\|\le r \; \Rightarrow\;\;  \|F(\alpha,\ii(v))-F(\alpha,\ii(v'))\|\le L\|\ii\|\|v-v'\|.
\]
Then,
\begin{equation}\label{eq:Vid2}
p\ge \frac{4}{\|A\|+L\|\ii\|}.
\end{equation}
\end{lemma} 
\begin{proof} Put $v(t):=\dot u(t)$. Since $u$ satisfies \eqref{eq:Hopf}, by condition \ref{c0},  $v$ is a continuous $p$-periodic function. Moreover, $\int_0^pv(t)dt=u(p)-u(0)=0$, so the condition (a) of Lemma \ref{lem:Vidossich} is satisfied.
Fix two values $0\le s\le t\le p$ such that $t-s\in [0,\frac p2]$. Then, 
\begin{align*}
\|v(t)-v(s)\|&=\|-A\big(u(t)-u(s)\big)+ F(\alpha,\ii(u(t)))-F(\alpha,\ii(u(s)))\|\\
&\le (\|A\|+L\|\ii\|) \|u(t)-u(s)\|\\
&\le (\|A\|+L\|\ii\|)\sup_{\tau\in \br} \|v(\tau)\|(t-s).
\end{align*}
Therefore, the function $U:[0,\tfrac p2]\to \br$ defined by $U(t):= (\|A\|+L\|\ii\|)\sup\limits_{\tau\in \br} \|v(\tau)\|t$ satisfies condition (b) of Lemma \ref{lem:Vidossich}.
Hence, from Lemma \ref{lem:Vidossich}, we obtain 
\begin{align*}
p\sup_{\tau\in \br}  \|v(\tau)\|&\le 2\int_0^{\frac p2} U(\tau)d\tau\\
&=2  (\|A\|+L\|\ii\|)\sup_{\tau\in \br} \|v(\tau)\|\int_0^{\frac p2}\tau d\tau\\
&=2  (\|A\|+L\|\ii\|)\sup_{\tau\in \br}\|v(\tau)\| \cdot \frac{p^2}{8},
\end{align*}
which implies \eqref{eq:Vid2}.
\end{proof}
\vs

\begin{lemma}\label{lem:Vid2}  Under the assumptions  \ref{i}, ($C$), \ref{c0}, \ref{c1}---\ref{c3} and \ref{lip}, assume that $X\subset \br\times V$ is a bounded set. Then, 
\begin{equation}\label{eq:Lip1}
\exists_{L>0} \forall_{(\alpha,u),(\alpha,u')\in X}\;\; \|F(\alpha,\ii(u))-F(\alpha,\ii(u'))\|
\le L\|\ii\| \, \|u-u'\|, 
\end{equation}
i.e. the map $(\alpha,u)\mapsto F(\alpha, \ii(u))$ is globally  Lipschitz continuous on $X$ with respect to $u$ (with the Lipschitz constant $L\|\ii\|$).
\end{lemma}
\begin{proof} Since the operator $\ii :V\to W$ is compact, $K:=\overline{\text{conv\,}\{ \ii(u):\exists_{\alpha}\;(\alpha,u)\in X\}}$ is compact and convex in $ H$. Put $a:=\inf\{\alpha: \exists_{u\in V} \; (\alpha,u)\in X\}$ and $b:=\sup \{\alpha: \exists_{u\in V} \; (\alpha,u)\in X\}$.
By \ref{c3} and compactness of the set $[a,b]\times K$, there exists a finite family of subsets 
$\mathscr U_l=(\alpha_l,\alpha_l')\times B_{\ve_l}(v_l)$, $v_l\in W$, on which $F$ is Lipschitz continuous with constant $L_l$, $l=1,2,\dots,N$, and
\[
[a,b]\times K\subset \bigcup_{l=1}^N \mathscr U_l.
\]
Put
\[
L:=\max\{L_l: l=1,2,\dots,N\}.
\]
Then, for a fixed $\alpha$, consider two points $(\alpha,v)$, $(\alpha,v')\in [a,b]\times K$. Since $K$ is convex, the segment $[v,v']:=\{v'+\lambda(v-v'):\lambda\in [0,1]\}$ is contained in $\{\alpha\}\times K$. Hence, one can find numbers 
\[0=\lambda_0<\lambda_1< \dots \lambda_{m-1}<\lambda_m=1,\]
such that the subsegment  
\[[v_{\nu-1},v_{\nu}]:=\{v'+\lambda(v-v'):\lambda\in [\lambda_{\nu-1},\lambda_{\nu}]\}
\]
satisfies  $\{\alpha\}\times [v_{\nu-1},v_{\nu}]\subset \mathscr U_{l_{\nu}}$ for some $l_{\nu}\in \{1,2,\dots,N\}$.  Consequently, 
\begin{align*}
\|F(\alpha,v')-F(\alpha,v)\|&\le \sum_{\nu=1}^m \|F(\alpha,v_{\nu-1})-F(\alpha,v_{\nu})\|\\
&\le \sum_{\nu=1}^m L_{l_\nu}\|v_{\nu-1}-v_{\nu}\|\\
&\le L\sum_{\nu=1}^m \|v_{\nu-1}-v_{\nu}\|\\
&=L\|v'-v\|
\end{align*}
and, for any $(\alpha,u')$, $(\alpha,u)\in X$, 
 \[
\|F(\alpha, \ii (u'))-F(\alpha,\ii(u))\|
\le L\|\ii(u'-u)\|\le L\|\ii\|\|u'-u\|.
\]
\end{proof}
\vs

\begin{remark}\label{rem:proj}{\rm 
Since the frequency $\beta$ does not appear in the original problem \eqref{eq:Hopf}, it is natural to study the ``geometric" unboundedness of branches of periodic solutions 
$(\alpha,u) \in  \br\times \mathscr E$. For this purpose,  it is useful to consider the natural projection  
\begin{equation}\label{eq:natur-proj}
\mathfrak p:\br_+^2\times \mathscr E\to  \br\times \mathscr E, \qquad
\mathfrak p(\alpha,\beta,u):=(\alpha,u).
\end{equation}
}
\end{remark}
\vs

\begin{proposition} \label{pro:Vid} Under the assumptions \ref{i}, ($C$), \ref{c0}, \ref{c1}---\ref{c3} and \ref{lip},
suppose  $\mathscr C\subset {\mathscr P }$  is a branch  such that $\mathfrak p(\mathscr C)$ is bounded in $\br\times \mathscr E$. 
Then, $\mathscr C$ is bounded in $ \br_+^2\times \mathscr E$, where $ \br\times \br_+$ is equipped with the usual Euclidean metric. 
\end{proposition}
\begin{proof} Assume that $\mathfrak p(\mathscr C)\subset [a,b]\times \overline{B_r(0)}$ for some $a<b$ and $r>0$. Then, since $X:= [a,b]\times \overline{B_r(0)}$ is bounded,  it follows  from Lemma \ref{lem:Vid2} that there exists a constant $L>0$ such that $F(\alpha, \ii(\cdot))$ is Lipschitz  continuous (with constant $L\|i\|)$ on the set 
$ [a,b]\times \overline{B_r(0)}$. Hence, for every $(\alpha,\beta,u)\in \mathscr C$, the function $u$ satisfies \eqref{eq:Hopf1}. Put $p:= \frac{2\pi}\beta$. Then, by rescaling the argument,  the function $\wt u(t)= u(\frac {\beta}{2\pi}t)$ is a $p$-periodic solution to \eqref{eq:Hopf}. Therefore, by Lemma \ref{lem:Vid1}, 
\[
p:=\frac{2\pi}\beta\ge \frac{4}{\|A\|+L\|\ii\|} \;\;\; \Leftrightarrow \;\;\; \beta\le \tfrac 12 \pi(\|A\|+L\|\ii\|),
\]
which implies that $\mathscr C$ is bounded.
\end{proof}
\vs

\subsection{Crossing numbers}
In this and subsequent subsections, we consider problem \eqref{eq:Hopf} under the assumptions \ref{i}, {\rm (C)}, \ref{c0}, \ref{c1}--\ref{c4}.  Consider $\beta_o\in {\bm b}(\alpha_o)$ and denote 
\begin{equation}\label{eq:resonant-beta}
{\bm b}(\alpha_o,\beta_o):=\{\beta\in  {\bm b}(\alpha_o):\exists_{k\in \bn}\;\; \beta=ik\beta_o\}.
\end{equation}
We call ${\bm b}(\alpha_o,\beta_o)$  the set of   {\it resonant critical frequencies} at the  point $(\alpha_o,\beta_o)$.
\vs
Since $\mathscr A(\alpha_o,\beta_o)$ is a Fredholm operator of index zero, the set ${\bm b}(\alpha_o,\beta_o)$ is finite, hence
\begin{equation}\label{eq:beta-alpha-o}
{\bm b }(\alpha_o,\beta_o)=\{\beta_1,\beta_2,\dots,\beta_n\}.
\end{equation}
\vs
%We will need  the following lemma.

\begin{lemma}\label{lem:cross-dom}  Under the assumptions \ref{i}, ($C$), \ref{c0}, \ref{c1}--\ref{c4}, given the set ${\bm b }(\alpha_o,\beta_o)$ defined by \eqref{eq:beta-alpha-o},
 there exist $\ve>0$ and $\delta>0$ such that the set $\mathscr U\subset \bc$ given by
\begin{equation}\label{eq:setU}
\mathscr U:=\bigcup_{l=1}^n \mathscr U_l, \quad \mathscr U_l:=\{\mu+i(\beta_l+\nu): 0<\mu<\ve,\; |\nu|<\ve\},
\end{equation}
satisfies  
\begin{equation}\label{eq:lamb}
\forall_{\alpha\in [\alpha_o-\delta,\alpha_o+\delta]}\; \forall_{\lambda\in \partial \mathscr U}\;\; \ker\triangle _\alpha(\lambda)\not=\{0\}\;\;\Rightarrow\;\;\; \alpha=\alpha_o\; \text{ and } \; \exists_{l\in\{1,2,\dots,n\}} \;\; \lambda=i\beta_l.
\end{equation}
Moreover, for any fixed $\alpha$ with  $0<|\alpha-\alpha_o|\le \delta$, the set
\begin{equation}\label{eq:finite}
\Big\{\lambda\in \mathscr U  :   \ker\triangle_{\alpha}(\lambda)\not=\{0\} \Big\} 
\end{equation} 
is finite.
\end{lemma}
\begin{proof} Since the operator $\triangle_{\alpha}(\lambda)$ depends continuously on $\alpha$ and is analytic with respect to $\lambda$,  for every $\beta_l\in \boldsymbol{\bm b }(\alpha_o,\beta_o)$, there exist $\ve>0$ and $\delta>0$ such that the sets $ {\mathscr U}_l:=\{\mu+i(\beta_l+\nu): 0<\mu<\ve,\; |\nu|<\ve\}$ satisfy $ \overline{\mathscr U_l}\cap  \overline{\mathscr U_{l'}}=\emptyset$ for $l\not=l'$ and 
\[\overline{\mathscr U_l}\cap \Big\{ \lambda: \ker \triangle_{\alpha_o}(\lambda)\not=\{0\}\Big\}=\{i\beta_l\}.\] 
Put $\ve:=\min\{\ve_1,\dots,\ve_n\}$ and find (by continuity of $\triangle_\alpha(\lambda)$) a sufficiently small $\delta>0$ such that condition \eqref{eq:lamb} is satisfied. Condition \eqref{eq:finite} follows from analyticity of the function $\lambda\mapsto \triangle_\alpha(\lambda)$.

\end{proof}
\vs

%\begin{definition}\rm \label{def:crossN}
 Assume that conditions  \ref{i}, {\rm (C)}, \ref{c0}, \ref{c1}--\ref{c4} hold and consider $\mathscr U$ defined by \eqref{eq:setU}. Put $\alpha_-:=\alpha_o-\delta$, $\alpha_+:=\alpha _o+\delta$. Then, for each $\beta_l \in \boldsymbol{\bm b }(\alpha_o)$ and $j=0,1,\dots,r$, define the numbers 
\[\mathfrak t_{j}^\pm(\alpha_o,\beta_l,0):= \sum_{\lambda\in \Sigma_{l,j}^\pm} m_j(\lambda),\]
where 
\[
\Sigma_{l,j}^\pm :=\Big\{\lambda\in \mathscr U_l: \ker\triangle^j_{\alpha_\pm}(\lambda)\not=\{0\}\Big\}
\]
and $m_j(\lambda)$ stands for the $\cV_j$-isotypic multiplicity of $\lambda$ 
(cf. \eqref{eq:H-isotypic1}--\eqref{eq:isot-multi}).
Then, put
\begin{equation}\label{eq:cross-j}
\mathfrak t_j(\alpha_o,\beta_l,0):=\mathfrak t_j^-(\alpha_o,\beta_l,0)-\mathfrak t_j^+(\alpha_o,\beta_l,0),
\end{equation}
and call the number $\mathfrak t_j(\alpha_o,\beta_l,0)$  the {\it $\cV_j$-crossing number} at $\alpha=\alpha_o$ for the critical frequency $\beta_l$. Notice that the elements from $\boldsymbol{\bm b }^k(\alpha_o)$  are resonant critical frequencies (see Definition \ref{def:crit-freq}).  
\vs

\subsection{Maximal twisted orbit types}\label{sec:max-orb-types}
\begin{definition} \rm
Consider the isometric $\Gamma$-representation $V$. Suppose $u:\br\to V$ is a non-constant periodic function with the minimal period $p>0$. Then,  the twisted subgroup 
$K\le \Gamma\times S^1$ given by
\begin{equation}\label{eq:symm}
K:=\{(\gamma,e^{i\theta})\in \Gamma\times S^1:\forall_{t\in \br}\;\; \gamma u(t+\tfrac{\theta p}{2\pi})=u(t)\}
\end{equation}
is called  the {\it spatio-temporal symmetry group} (in short, {\it symmetry group}) of the periodic function $u$ (cf. \cite{Fied,Gol}).  
\end{definition}

\vs

Let $\cV_j$ (here $j\in \{0,1,2,\dots, r\}$)  be an irreducible $\Gamma$-representation such that the isotypic component $V_j$ in \eqref{eq:H-isotypic1}  is non-zero. Consider the irreducible $G$-representation $\cV_{k,j}$, $k>0$, where $S^1$ acts by $k$-folding (cf. \eqref{eq:Vjk}) and observe that $\Phi_{k,j}$ consists of  the maximal twisted orbit types in $\mathscr E_{k,j}$.
Denote by $\Phi_{k,j}$ the set of all maximal twisted orbit types in $\cV_{k,j}$ and 
put
\[
\Phi_k:=\bigcup_{j=0}^r \Phi_{k,j}.
\] 
 
\vs 

\begin{remark}\label{rem:maximal-types}\rm 
(a) Consider the set $\Phi_{0,j}$ of the orbit types in $\cV_j\setminus \{0\}$ (with  $V_j\not=\{0\}$). These orbit types correspond to 
constant functions $u$ in $\mathscr E$. Notice that for $(K)\in \Phi_{0,j}$, one has dim\,$W(K)=0$, and as such $K$ is not twisted. At the same time,  for $(K)\in \Phi_k$, $k>0$, $K$ is a (twisted) finite group and  dim\,$W(K)=1$.\vs\noi
(b) For $k>0$, consider  the $k$-folding homomorphism $\vp_k:\Gamma\times S^1\to \Gamma\times S^1$ given by $\vp_k(\gamma,e^{i\theta})=(\gamma,e^{ik\theta})$, $\gamma\in \Gamma$, $e^{i\theta}\in S^1$. Then, for $(K)\in   \Phi_{1,j}$, define $\vp_k^*(K)$ as the conjugacy class of the subgroup $\vp_k^{-1}(K)$. Notice that $\vp_k^*(K)\in \Phi_{k,j}$ and 
$\vp_k^*:\Phi_{1,j}\to \Phi_{k,j}$ is a bijection.\vs\noi
(c)  Consider a function $u\in \mathscr E$ such that $G_u= K$ for  $(K)\in \Phi_{k,j}$,  $k>0$. Then, $u$ is a non-constant $2\pi$-periodic function with the minimal period $p:=\frac{2\pi}k$ and the  symmetry group $G_u = \vp_k(K)$. \vs\noi
(d) Take a periodic function $u\in \mathscr E$ such that $(G_u)\ge (K)$, where  $(K)\in \Phi_k$ for some $k\in \bn$. Then, even if we do not know the minimal period $p$ of $u$, still we know the exact symmetry group of $u$, which is $\vp_k(K)$. 
\end{remark}

\vs 
Take $(K)\in \Phi_k$ and define 
\begin{equation}\label{eq:zeta}
\zeta_{k,j}(K):=\begin{cases}
1&\text{ if } \; (K)\in \Phi_{k,j},\\
0&\text{ otherwise}.
\end{cases}
\end{equation}
Further, for  $\beta_o\in \boldsymbol{\bm b }(\alpha_o)$ and $(K)\in \Phi_k$,
 define the integers
\begin{equation}\label{eq:inv-b}
\mathfrak t^k_K(\alpha_o,\beta_o,0) :=\sum_{j=0}^r\zeta_{k,j}(K)\cdot \mathfrak t_j(\alpha_o,k\beta_o,0),
\end{equation}
which play an important role in the following results.
\vs

\subsection{Main results} 
We are now in a position to present the results. We start with a local bifurcation theorem. 
\vs 
\begin{theorem}[Local Hopf Bifurcation Theorem]\label{th:main1}
Under the assumptions \ref{i},  ($C$),  \ref{c0},  \ref{c1}--\ref{c4}, suppose that 
 there exists $\beta_o\in \boldsymbol{\bm b }(\alpha_o)$ such that for some $k\ge 1$ and $(K)\in \Phi_k$, one has $\mathfrak t_K^k(\alpha_o,\beta_o,0)\not=0$ (cf. \eqref{eq:inv-b}). Then:  
 \begin{enumerate}[label=($i_\arabic*$)]
 \item\label{i1} system \eqref{eq:Hopf} undergoes a Hopf bifurcation at the center $(\alpha_o,0)$ with the limit frequency $\beta_o$ (cf. Definition \ref{def:Hopf});
 %In particular,
 %there exists a branch $\mathscr C$  of $\overline{\mathscr P }^+$ such that $(\alpha_o,\beta_o,u_o)\in \mathscr C$. Moreover, 

\item\label{i2} the fixed-point set ${\mathscr C}^K$ contains a non-trivial continuum and 
  $G_u\ge K$ for every $(\alpha,\beta,u)\in {\mathscr C}^K$. 
  %$G_u\ge K$.
  \end{enumerate}
  
%\noi Moreover, for  any   bounded isolating neighborhood  $\Omega\subset \br\times \mathscr E$ of the center $(\alpha_o,u_o)$ (i.e. $\Omega$  contains no other centers), one has:
   
 %  \begin{enumerate}[label=($i_\arabic*$)]
%   \setcounter{enumi}{2}
%\item\label{i3} if $\mathscr C$ is compact and $\mathfrak p(\mathscr C)\subset \Omega$ (see \eqref{eq:natur-proj}), then 
 %for all $m\ge 1$ 
%\[
%\sum_{\beta_k\in B_{\mathscr C}} \mathfrak t_K^k(\alpha_o,\beta_k,u_o)=0,
%\]
%where $B_{\mathscr C}:= \{\beta_k\in \bm b(\alpha_o): (\alpha_o,\beta_k,u_o)\in \mathscr C\}$; 

%\item\label{i4} if $\mathscr C=\overline{\mathscr C}$ but $\mathscr C$ is not compact (i.e. it is unbounded) then $\partial\mathscr U\cap \mathfrak p(\mathscr C)\not=\emptyset$.
%\end{enumerate}
\end{theorem} 
\vs

\begin{remark}\rm (i) It was observed by J. Ize \cite{Ize} that the occurrence of the Hopf bifurcation with prescribed symmetry is related to the non-triviality of the equivariant 
$J$-homomorphism associated with the equivariant operator equation. This observation gives rise to the following two questions: (a) Under which conditions on $A$ and $F$ in 
\eqref{eq:Hopf} is the $J$-homomorphism correctly defined? (b) Under which conditions is this homomorphism non-trivial? From this viewpoint, conditions  \ref{i},   ($C$),  \ref{c0},  \ref{c1}--\ref{c4} are related to (a), while the condition  $\mathfrak t_K^k(\alpha_o,\beta_o,0)\not=0$  is related to (b).

\medskip

(ii) Observe that the hypotheses for conclusions \ref{i1} and \ref{i2} of Theorem \ref{th:main1} are essentially based on the behavior of the {\it linearization} of $F$ along the branch of equilibria, and as such, can be easily verified.
\end{remark}
\vs

To formulate a global result, we assume in this section the following condition:

\vs
\begin{enumerate}[label=($G_\arabic*$)]\setcounter{enumi}{-1}
 \item\label{g0} $\Gamma:=\Gamma_o\times \bz_2$, where $\Gamma_o$ is a finite group, and $\bz_2$ acts antipodally on $V$, i.e. by multiplication by $\pm 1$.
 \end{enumerate}
 \vs 
  Consider the two-element subgroup 
$\mathfrak H$ of $\Gamma_o\times \bz_2\times S^1$ given by
\begin{equation}\label{eq:h-subset}
\mathfrak H:=\{(e,1,1),(e,-1,-1)\}, \quad e\in \Gamma \;\text{ is the neutral element of $\Gamma$}.
\end{equation}
Then,
\[
\mathscr E^{\mathfrak H}=\overline{\bigoplus_{k=0}^\infty \bigoplus_{j=0}^r\mathscr E^{\mathfrak H}_{k,j}}, 
\]
where 
\[
\mathscr E^{\mathfrak H}_{k,j}=\begin{cases}\left\{\vp:\br\to V_j\;: \;\vp(t)= \cos(kt)u+\sin(kt)v ,\; u,\, v\in V_j, \; t\in \br  \right\}, & \text{ if } k=2m+1,\\
\{0\}, & \text{ otherwise}
\end{cases}
\]
(here  $m=0,1,2,\dots$).
\vs
Instead of the operator equation \eqref{eq:Hopf2}, let us consider the following $\mathfrak H$-reduced operator equation
\begin{equation}\label{eq:Hopf3}
\mathscr F^{\mathfrak H}(\alpha,\beta,u)=0, \quad (\alpha,\beta, u)\in \br\times \br_+\times \mathscr E^{\mathfrak H}.
\end{equation}
\vs

\begin{remark}\label{rem:center-0}{\rm Notice that:
\begin{itemize}
\item[(a)] Every solution to \eqref{eq:Hopf3} is a solution to \eqref{eq:Hopf2}, therefore every branch of solutions to  \eqref{eq:Hopf3} is also a branch of solutions to  \eqref{eq:Hopf2}.
\item[(b)] Since $W(\mathfrak H)=G$, equation  \eqref{eq:Hopf3} is $G$-equivariant.
\item[(c)] The points $(\alpha,0)$, $\alpha\in \br$, are the only equilibria for \eqref{eq:Hopf} satisfying equation \eqref{eq:Hopf3}, i.e. the only constant function satisfying  \eqref{eq:Hopf3}  is the zero function. Therefore, 
%put $u(\alpha) \equiv u_o =0$ in condition \ref{c2}, in which case 
condition \ref{c1} is not needed anymore.
\end{itemize}}
\end{remark}
\vs

Being motivated by Remark \ref{rem:center-0}(c), we introduce the following condition:
\vs
\begin{enumerate}[label=($G_\arabic*$)]%\setcounter{enumi}{3}
\item\label{g1} For every $\alpha\in \br$, the Fr\'echet derivative $B(\alpha):=D_uF(\alpha,0):W\to H$ exists and depends continuously on $\alpha$.
\end{enumerate}
\vs

\begin{definition}\label{def:center-0}\rm 
Assume that conditions  \ref{i}, ($C$), \ref{c0}, \ref{g0} and \ref{g1} are satisfied. A point  $(\alpha_o, \beta_o,0)\in \br^2_+\times {\mathscr E}^{\mathfrak H}$ is called an {\it $\mathfrak H$-critical point} of $\mathscr F$ if $D_u\mathscr F^{\mathfrak H}(\alpha_o,\beta_o,0):{\mathscr E}^{\mathfrak H}\to {\mathscr E}^{\mathfrak H}$ is not an isomorphism. We denote by ${\Lambda}^{\mathfrak H}$ the set of all $\mathfrak H$-critical points of $\mathscr F$.
\end{definition}
\vs

\begin{remark}\rm (i) Notice that $(\alpha_o,\beta_o,0)$ is an $\mathfrak H$-critical point  of $\mathscr F$
if  $(\alpha_o,0)$ is a center for which  there exists a  characteristic value $\lambda_o=ik\beta_o$ with  an odd integer $k\in \bn$.  

(ii) For simplicity, we keep the same notation 
$\triangle_{\alpha_o}:V^c\to V^c$ for the 
%following our previously introduced notation, 
%we continue to denote  the 
characteristic operator at the center $(\alpha_o,0)$  
as in  \eqref{eq:Characteristic} (with the operator $B(\alpha)$  given in condition \ref{g1}).
%(i.e. we replace $u_o$ by $0$) by  $\triangle_{\alpha_o}:V^c\to V^c$  
%(see \eqref{eq:Characteristic})  the operator $B(\alpha)$  given in condition \ref{g1}.
\end{remark}

%\noi(ii)  A center $(\alpha_o, 0)$ is called  {\it $\mathfrak H$-isolated}   if   

\vs
We  introduce the following condition:
\vs
\begin{enumerate}[label=($G_\arabic*$)]\setcounter{enumi}{1}
\item\label{g2} The set $\Lambda^{\mathfrak H}\subset \br^2_+$ is discrete.
\end{enumerate}
\vs 
Notice that under the assumptions \ref{i}, ($C$), \ref{c0}, \ref{g0}, \ref{g1}--\ref{g2}, in order for the set $\Lambda$ to be  discrete, it suffices  to assume that any center $(\alpha_o,0)$ for \eqref{eq:Hopf}  satisfies the condition
\[
\exists_{\delta>0}\; \forall_{\alpha\in \br}\;\; 0<|\alpha-\alpha_o|<\delta\;\; \Rightarrow\;\; \text{$(\alpha,0)$ is not a center.}
\]

\vs

\begin{theorem}[Global Hopf Bifurcation Theorem]\label{th:main2}
Assume  that conditions \ref{i}, ($C$), \ref{c0}, \ref{g0}, \ref{g1}--\ref{g2} %  and  \ref{lip}  
are satisfied and let $\mathscr W\subset \br^2_+\times \mathscr E^{\mathfrak H}$ be an open bounded {(where $\br^{2}_+$ is equipped with metric \eqref{eq:metric-R2+})} $G$-invariant set such that 
(i) 
$\overline{\mathscr W}\subset  \br^2_+\times \mathscr E^{\mathfrak H}$;
(ii)  for all $(\alpha,\beta, 0)\in \partial \mathscr W$, $(\alpha, \beta,0)$ is not an $\mathfrak H$-critical point, {i.e. $\Lambda^{\mathfrak H} \cap \partial \mathscr W=\emptyset$}; 
(iii)  there exists an $\mathfrak H$-critical point $(\alpha_o,\beta_o,0)\in \mathscr W$.
Suppose that ${\mathscr C}\subset \br^2_+\times \mathscr E^{\mathfrak H}$ is a branch of non-constant periodic solutions  such that $(\alpha_o,\beta_o,0)\in \overline{{\mathscr C}}$. Then, either
\begin{itemize}
\item[(a)] { $\mathscr{C} \cap \partial\mathscr W =\overline{\mathscr C} \cap \partial\mathscr W\not=\emptyset$},  
\item[(b)] or there is a  finite set
\[\overline{\mathscr C}\cap\br^2_+ \times \{0\}=\{(\alpha_1,\beta_1,0),(\alpha_2,\beta_2,0), \dots, (\alpha_l,\beta_l,0),\dots,(\alpha_n,\beta_n,0)\}
\]
(where none of two different  critical frequencies $\beta_l$ are resonant) 
such that for every odd $m$ and $K\in \Phi_m$, one has
\begin{equation}\label{eq:global}
\sum_{l=1}^n \mathfrak t_K^m(\alpha_l,\beta_l,0)=0.
\end{equation}
\end{itemize}
\end{theorem}
\vs
As an immediate consequence of Theorem \ref{th:main2} and Proposition 
\ref{pro:Vid},
one obtains the following statement.
\vs

\begin{theorem}\label{th:main3}
Assume  that conditions \ref{i}, ($C$), \ref{c0}, \ref{g0}, \ref{g1}--\ref{g2} and the Lipschitz condition \ref{lip}  
are satisfied. Assume, in addition, that  for some odd $m\in \bn$, there exists $(K) \in \Phi_m$ 
%for some odd $m\in \bn$ 
such that:
\begin{itemize}
 \item[(i)] $\mathfrak t^j_K(\alpha,\beta,0) \cdot \mathfrak t^j_K(\alpha',\beta',0) \geq 0$  for all 
 $(\alpha,\beta,0), (\alpha',\beta',0) \in \Xi^{\mathfrak H}_0$, $j=0,1,2,\dots,r$;
 \item[(ii)] for some $(\alpha_o,\beta_o,0)\in \Xi^{\mathfrak H}_0$, one has $\mathfrak t^j_K(\alpha_o,\beta_o,0) \not=0$ 
 \end{itemize}
 {(where $\Xi_0$ is given by \eqref{eq:equi-set}, $\mathfrak H$ by \eqref{eq:h-subset}, and $\Xi^{\mathfrak H}_0 $ denotes the $\mathfrak H$-fixed points set in $\Xi_0$).}
 Then, there exists an unbounded  branch $\mathscr C$ with  $(\alpha_o,\beta_o,0)\in \overline{\mathscr C}$.
\end{theorem}
\vs

\section{Example} As a follow-up to Example \ref{ex:parab}, assume that  $\Om:=\{z\in \bc: |z|<1\}\subset \br^2$ and take $\Gamma_1:=O(2)$. Clearly, $\Gamma_1$ represents the  symmetry group of $\Om$. We  consider the parabolic system 
\begin{equation}\label{eq:prototype}\begin{cases}
\dot u(t)-\bm\triangle u(t)=\mathfrak f(\alpha,u(t)), \quad u(t):=u(t,\cdot):\Om\to\br^2,\\
u(t)|_{\partial \Om}\equiv 0\;\; \text{for all }\; t\ge 0,
\end{cases}
\end{equation}
where $\mathfrak f:\br\times \br^2\to \br^2$ is a function satisfying the following conditions: 
\begin{itemize}
\item[(e0)] $\mathfrak f$ is continuous and $\mathfrak f(\alpha,-u)=-\mathfrak f(\alpha,u)$ for all $\alpha$ and $u\in \br^2$;
\item[(e1)] under the identification of  $\br^2$ with $\bc$,  one has  for some $a$, $b$ $\in \br\setminus \{0\}$:
\[\mathfrak f(\alpha,u):= \alpha \bm \eta u+r(u),\quad \text{ where  }\;  \bm \eta:=a+ib,\quad
\lim_{\|u\|\to 0} \frac {|r(u)|}{|u|} = 0;\]
\item[(e2)] there exists $\mathfrak p>1$ such that 
\[\exists _{\bm c,\bm d>0}\; \;\forall_{u\in \br^2}\; \;  \;|r(u)|\le \bm c+\bm d|u|^{\frac{\mathfrak p}2}.\]
\end{itemize}
\vs

We have the following   functional spaces related to \eqref{eq:prototype}:
\begin{align*}
H&:=L^2(\Om;\br^2),\\
{\mathfrak{D}}(A)&:= H_o^1(\Om;\br^2)\cap H^2(\Om; \br^2),\\
W&:=L^p(\Om;\br^2).
\end{align*}
We also take $V$ to be the space ${\mathfrak{D}}(A)$ equipped with the Sobolev norm.
All the above spaces are natural Banach $O(2)\times \bz_2$-representations (with the antipodal  $\bz_2$-action on $\Om$).
Then, conditions \ref{i}, \ref{c0}, \ref{c1} and \ref{c2} are satisfied (see Example 
 \ref{ex:parab}). Moreover, $B(\alpha)v:= \alpha \bm\eta v$, $v\in W$, and the characteristic equation for \eqref{eq:prototype} at zero is given by 
 \begin{equation}\label{eq:ex-char-op}
 \triangle _\alpha(\lambda)v:=v-(A+\lambda\id)^{-1}(\alpha \bm\eta i(v)) =0, \quad v\in V:={\mathfrak{D}}(A).
 \end{equation}
 Notice that  the characteristic operator $\triangle _\alpha(\lambda)$ is $O(2)\times \bz_2$-invariant.
The characteristic  equation can be written as the following boundaru-value problem: 
\begin{equation}\label{eq:Lapace-eigen}
\begin{cases}
-\bm \triangle v+(\lambda-\alpha \bm\eta)v=0, \quad v\in V,\\
v|_{\partial \Om}\equiv 0.
\end{cases}
\end{equation}
Observe that $\lambda_o\in \bc$ is a characteristic  value for \eqref{eq:prototype} if and only if $\alpha \bm\eta -\lambda_o \in \sigma(-\bm \triangle)$. The spectrum $\sigma(-\bm \triangle)$ can be effectively described (by using the standard polar coordinates $(r,\theta)$). Namely, denote by $\bm s_{n,j}$ the $n$-th positive zero of the Bessel function $\bm J_{j}$ of the first kind.
Define 
\[
\sigma(-\bm \triangle)=\{\bm s_{n,j}: n\in \bn, \; j=0,1,2,\dots\},
\]
where corresponding to  each eigenvalue $\bm s_{n,j}$, the eigenspace  can be described as follows: 
\[
\mathbb E_{n,j}:=\left\{\bm J_{j}(\bm s_{n,j}r)\Big(\cos(j\theta)\bm a+\sin(j\theta)\bm b\Big): \bm a,\, \bm b\in \br^2\right\}.
\]
Clearly, $\mathbb E_{n,j}$ are $O(2)\times \bz_2$-invariant. 
Notice that  the space $\mathbb E_{n,0}$ is equivalent to $\br^2$ and for  $j>0$,  $\mathbb E_{n,j}$ is equivalent to $(\br^2)^c=\br^2\oplus i\br^2:=\{\bm a+i\bm b: \bm a,\, \bm b\in \br^2\}$ (cf. \eqref{eq:complex-str}). The $O(2)$-action on $\mathbb E_{n,j}=\br^2\oplus i\br^2$ is given by 
\[e^{i\psi}(\bm a+i\bm b)=(\cos (j\psi)+i\sin(j\psi))\cdot(\bm a+i\bm b),\quad \kappa(\bm a+i\bm b)=\bm a-i\bm b,
\]
where $\bm a$, $\bm b\in \br^2$ and ``$\cdot$'' denotes the complex multiplication in $\mathbb E_{n,j}$. 
\vs
 Put   $\Gamma:=O(2)\times \bz_2$ and $G:=O(2)\times \bz_2\times S^1$. The $\Gamma$-isotypic decomposition of $V$ is given by
 \[
 V=\overline{\bigoplus_{j=1}^\infty V_j}, \quad V_j:=\overline{\bigoplus_{n=1}^\infty \mathbb E_{n,j}},
 \]
 where the component $V_j$ is modeled on the irreducible $O(2)\times \bz_2$-representation $\cW_j^-$ (here if $j>0$ then $\cW_j^-\simeq \bc$ and $(e^{i\vp},\pm1)z:=\pm e^{ij\vp}\cdot z$, 
 $\kappa z=\overline z$, while $\cW_0\simeq \br$   is a trivial $O(2)$-representation with  the antipodal $\bz_2$-action). 

%Then, in the $G$-isotypic decomposition of $\mathscr E$ given by \eqref{eq:iso-E}, one has
%\[
%\mathscr E_{k,j}=\bigoplus
%\]

\vs

Notice that,  the characteristic operator $\triangle _\alpha(\lambda)$ (given by \eqref{eq:ex-char-op})  is $O(2)\times \bz_2$-equivariant. We put 
\[
\triangle^j_\alpha(\lambda):=\triangle _\alpha(\lambda)|_{V_j^c}:V_j^c\to V_j^c
\]
(see subsection \ref{sec:char-eq}).
\vs

A number $i\beta$, $\beta>0$ is a characteristic root for   \eqref{eq:prototype}  if  $\alpha=\frac{\bm s_{n,j}}a$ and 
$\beta =\frac{\bm s_{n,j}b}a$ for some $n\in \bn$ and $ j=0,1,2,\dots$\
Consequently, we have the critical set 
\[
\bm \Lambda_o:=\left\{ (\alpha,\beta)\in \br^2_+:\alpha=\frac{\bm s_{n,j}}a,\;  \beta =\frac{\bm s_{n,j}b}{ka},\;\;n, k\in \bn, \; j=0,1,2,\dots\right\}.
\]
Since the monotone sequence $\bm s_{n,j}$, $n\in \bn$, increases to infinity for every fixed $j$ (as a matter of fact, $\bm s_{n+1,j}-\bm s_{n,j}>\pi$ for all $n\ge 0$ and $j\ge 1$) and $\bm s_{1,j}>2\sqrt{j+1}$
for all $j\ge 0$
%(also, $\bm s_{1,j}>\sqrt{j(j+2)}$ for all $j\ge0$; and, $\bm s_{1,j}> \bm s_{1,0}+j$ for $j>0$).
(see \cite{ismail, ismail1}), the set $\bm \Lambda_o$ is discrete,
i.e. every critical value $(\alpha_o,\beta_o)\in \bm \Lambda_o$ is isolated. Hence, one can associate with it the $\cW_j^-$-isotypic crossing numbers  $\mathfrak t_j(\alpha_o,\beta_o)$ (see \eqref{eq:cross-j}).
Observe that for the critical points $(\alpha_o,\beta_o)$ where $\alpha_0=\frac{\bm s_{n,j}}a$ and $\beta_o=\frac{\bm s_{n,j}b}{ka}$, by inspection, we have the crossing numbers 
\[\mathfrak t_{j}(\alpha_o,\beta_o)=-1.\] 
\vs
Let us describe the orbit types in $V_j$ for $j=0,1,2,\dots$. There is only one orbit type  $(O(2))$ in $V_0\setminus \{0\} $ and also one orbit type $(D_{2j}^d)$  in $V_j\setminus \{0\} $ for $j>0$. 
\vs
Now, problem \eqref{eq:Lapace-eigen} can be reformulated in the space $\mathscr E$ (see \eqref{eq:spaces2}) as the two-parameter $G$-equivariant  equation (see \eqref{eq:Hopf2}), with 
$N_{\mathfrak f} (\alpha,v)(t):=\mathfrak f(\alpha,v(t))$. One can verify that the $G$-isotypic decomposition of $\mathscr E$ %(here $\mathscr E$ is given by \eqref{eq:spaces2}) 
is
\[
\mathscr E=\overline{\bigoplus _{k=0}^\infty\bigoplus_{j=0}^\infty \mathscr E_{k,j}},
\]
where the $G$-isotypic component $\mathscr E_{k,j}$, provided by \eqref{eq:iso-E-kj} is modeled on the $G$-irreducible representation $\cV_{k,j}$. The $G$-representation $\cV_{k,j}$, as a real vector space, is   the complexification of $\cW_j^-$ on which   the $S^1$-action is defined by $k$-folding (see \eqref{eq:Vjk}). Now, we can identify the maximal twisted orbit types in $\mathscr E_{k,j}$ (see subsection \ref{sec:max-orb-types}). For this purpose we use the following amalgamated notation:
\begin{align*}
\bullet\;\text{ for $j=0$ and $k>0$  the maximal twisted orbit type is } \;\; &H_k:= \big(O(2)\times \bz_2\big)^{O(2)\times \bz_1}\times _{\bz_2}^{\bz_k}\bz_{2k};\\
\bullet\; \text{ for $j>0$ and $k>0$  the maximal twisted orbit type is } \;\; & K_k:=\big(D_{2j}\times \bz_2\big)^{D_{2j}^d}\times_{\bz_2}^{\bz_k}\bz_{2k}.
\end{align*}
Therefore, for $k>0$, 
\[
\Phi_k=\{(H_k),(K_k)\}.
\]
\vs
   Now, as an immediate consequence of  Theorem \ref{th:main1}, we obtain the following result.
   \vs
   \begin{theorem}\label{th1:example1}
   Assume that the function  $\mathfrak f:\br\times \br^2\to \br^2$  satisfies conditions {\rm (e0)--(e2)}.
   Then for every  $\alpha_o=\frac{\bm s_{nj}}a$ and $\beta_o=\frac{\bm s_{nj}b}{a}$  for some  $n\ge 1$ there exists a branch of non-constant periodic solutions  to \eqref{eq:prototype}  bifurcating from $(\alpha_o,\beta_o,0)$ with the twisted orbit type $(\mathcal H)$ which is 
   \begin{itemize}
   \item[(i)] $(H_k)$ for some $k\in \bn$, if $j=0$, and
   \item[(ii)] $(K_k) $ for some  $k\in \bn$, if $j>0$.
   \end{itemize}
   \end{theorem}
   \vs
   Notice that 
one can apply  Theorem \ref{th:main3} to  obtain a global Hopf bifurcation result for \eqref{eq:prototype}. To this end, observe that 
\[
\bm \Lambda_o^{\mathfrak H}:=\left\{ (\alpha,\beta)\in \br^2_+:\alpha=\frac{\bm s_{n,j}}a,\;  \beta =\frac{\bm s_{n,j}b}{(2k-1)a},\;\;n, k\in \bn, \; j=0,1,2,\dots\right\}
\]
   and for every $(\alpha_o,\beta_o)\in \bm \Lambda_o^{\mathfrak H}$, one has $\mathfrak t_{k,j}(\alpha_o,\beta_o)=-1$, thus we have the property 
   \[
   \mathfrak t_j(\alpha,\beta) \cdot \mathfrak t_j(\alpha',\beta') \geq 0, \; \; j=0,1,2,\dots,
   \]
   is satisfied for all $(\alpha,\beta)$, $(\alpha',\beta')\in \bm \Lambda_o^{\mathfrak H}$. Since one gets the sum of non-positive numbers in \eqref{eq:inv-b} for any $K$, the conditions (i) and (ii) in Theorem \ref{th:main3} are satisfied, which leads us to the following result: 
   \vs
   
  \begin{theorem}\label{th:example2}
Assume that the function  $\mathfrak f:\br\times \br^2\to \br^2$  satisfies conditions {\rm (e0)--(e2)}.
    Assume, in addition, that  for some odd $m\in \bn$, there exists $(K) \in \Phi_m$. 
%for some odd $m\in \bn$ 
 Then, there exists an unbounded  branch $\mathscr C$ with  $(\alpha_o,\beta_o,0)\in \overline{\mathscr C}$ for every $(\alpha_o,\beta_o)\in \bm \Lambda^{\mathfrak H}$.
\end{theorem}
 \vs 
%\vskip2cm 
    
\begin{figure}[H]
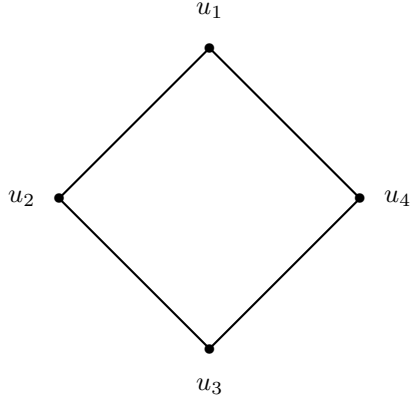

\vskip2.5cm\hskip5.6cm
\rput(-0.5,0){$u_2$}\psline{}(0,0)(2,2)\psdot(0,0)
\rput(2,2.5){$u_1$}\psline{}(2,2)(4,0)\psdot(2,2)
\rput(4.5,0){$u_4$}\psline{}(4,0)(2,-2)\psdot(4,0)
\rput(2,-2.5){$u_3$}\psline{}(2,-2)(0,0)\psdot(2,-2)
\vskip2.3cm
\caption{System with $D_4$ symmetries.}\label{fig:d4}
\end{figure}
\vskip0.1cm

\noindent
Let us consider four identical systems of type \eqref{eq:prototype} coupled into a dihedrally symmetric configuration as shown in Figure \ref{fig:d4}. The corresponding differential system reads
\begin{equation} \label{eq:exD4}
  \begin{cases} 
  & \frac{\partial u_1}{\partial t}-\bm\triangle u_1=f(\alpha, u_1(t))-\alpha(cg(u_1)+dg(u_2)+dg(u_4)),\\
  & \frac{\partial u_2}{\partial t}-\bm\triangle u_2=f(\alpha, u_2(t))-\alpha(dg(u_1)+cg(u_2)+dg(u_3)),\\
  & \frac{\partial u_3}{\partial t}-\bm\triangle u_3=f(\alpha, u_3(t))-\alpha(dg(u_2)+cg(u_3)+dg(u_4)),\\
  & \frac{\partial u_4}{\partial t}-\bm\triangle u_4=f(\alpha, u_4(t))-\alpha(dg(u_1)+dg(u_3)+cg(u_4)),\\
  &u_1(t)|_{\partial \Om}=u_2(t)|_{\partial \Om}=u_3(t)|_{\partial \Om}=u_4(t)|_{\partial \Om}=0 \;\; \text{for all }\; t\ge 0,
  \end{cases}
\end{equation}
where $u:=(u_1, u_2, u_3, u_4)^T \in \bc^4=(\br^2)^4$, $c,d \in \br$ (values of $c$ and $d$ will be specified later), $g: \bc \to \bc$ is an analytic odd function such that $g(0)=0$ and $g'(0)=1$. Notice that 
 \eqref{eq:exD4} has  symmetry group $\Gamma_2:=D_4\times \bz_2$. 
 \vs
 Consider the $\Gamma_2$-isotypic decomposition of $\bc^4=(\br^2)^4$:
 \[
 \bc^4=W_0 \oplus W_1 \oplus W_3,
 \]
 where the $\Gamma_2$-isotypic component $W_j (j = 0, 1, 3)$ is modeled on the irreducible $\Gamma_2$-representation, namely $\mathcal{U}^-_j (j = 0, 1, 3)$ (with the antipodal $\bz_2$-action).\footnote{The irreducible $\Gamma_2$-representations can be easily identified using the GAP package {\tt EquiDeg} (see Appendix \ref{app:D}  and Example \ref{pro:twisted-orbit-types}), where they are indexed by the order of their characters provided by GAP, i.e.  $\chi_4:=\chi_{\cU_0^-}$, $\chi_{10}=\chi_{\cU_1^-}$ and $\chi_2:=\chi_{\cU_3}^-$.}
The linearized system of \eqref {eq:exD4} at $(\alpha, \bm 0),\; \bm 0=(0, 0, 0, 0)$ is given by
\begin{equation} \label{exlinear}
\frac{\partial u}{\partial t}-\Delta u = D_{u} F(\alpha, \bm 0)u - \alpha C u
\end{equation}
where 
\begin{equation} F(\alpha, u)=\left[ \begin{array}{c} 
f(\alpha, u_1)\\
f(\alpha, u_2)\\
f(\alpha, u_3)\\
f(\alpha, u_4)
\end{array} \right], \quad C=\left[ \begin{array}{cccc}
c & d & 0 & d \\
d & c & d & 0 \\
0 & d & c & d \\
d & 0 & d & c 
\end{array} \right].
\end{equation}
\vs
Substituting $u(t)=e^{\lambda t}v(t)$ in equation \eqref{exlinear}, where $v(t) \in \bc^4, \lambda \in \bc$, one has for $A:=-\Delta$:
\begin{align*}
(\lambda \id + A) e^{\lambda t}v(t) &= D_{u} F(\alpha, \bm 0) e^{\lambda t}v(t) - \alpha C e^{\lambda t}v(t) \nonumber \\
\Rightarrow (\lambda \id + A) v(t) &= D_{u} F(\alpha, \bm 0) v(t) - \alpha C v(t). \label{eq:excharacter}
\end{align*}
\vs
Now, one can compute the spectrum of $C$ as 
\[
\sigma(C)=\{\mu_{0}=c+2d, \mu_{1}=c, \mu_{2}=c-2d\}.
\]
In particular, consider
\begin{equation} F(\alpha, u)=\left[ \begin{array}{c} 
-\alpha u_1 i+r(u_1)\\
 -\alpha u_2 i+r(u_2)\\
-\alpha u_3 i+r(u_3)\\
-\alpha u_4 i+r(u_4)
\end{array} \right].
\end{equation}
Put $\Gamma:=O(2)\times D_4\times \bz_2$ and notice that system \eqref{eq:exD4} is $\Gamma$-symmetric. Put
\[
G:=O(2)\times D_4\times \bz_2\times S^1
\]
and introduce the spaces
\begin{align*}
H&:=L^2(\Om;\bc^4),\\
D(A)&:= H_o^1(\Om;\bc^4)\cap H^2(\Om; \bc^4),\\
W&:=L^p(\Om;\bc^4).
\end{align*}
The operator $A: {\mathfrak{D}}(A)\subset H\to H$ is the negative Laplacian  $-\triangle$, and we define 
 the space $\mathscr E$ by \eqref{eq:spaces2}. 
 The space $V$  (i.e. the domain ${\mathfrak{D}}(A)$ equipped with the Sobolev norm) is a Hilbert $\Gamma$-representation with the $\Gamma$-isotypic decomposition
 \[
 V=\overline{ \bigoplus_{j=0}^\infty V_{j,0}}\oplus \overline{ \bigoplus_{j=0}^\infty V_{j,1}}\oplus \overline{ \bigoplus_{j=0}^\infty V_{j,3}}, 
 \]
 where, for $l=0,1,3$, 
 \begin{align*}
 V_{j,l}&:=\overline{\bigoplus_{n=1}^\infty \mathbb E_{n,j,l}},\\
 \mathbb E_{n,j,l}&:=\left\{\bm J_{j}(\bm s_{n,j}r)\Big(\cos(j\theta)\bm a+\sin(j\theta)\bm b\Big): \bm a,\, \bm b\in W_l\right\}.
 \end{align*}
 Clearly, $\mathscr E$ is a Hilbert $G$-representation 
with $G$-isotypic decomposition
 \begin{equation}\label{eq:G-iso-D4}
 \mathscr E=\overline{\bigoplus _{k=0}^\infty\bigoplus_{j=0}^\infty \mathscr E_{k,j,0}}\oplus  \overline{\bigoplus _{k=0}^\infty\bigoplus_{j=0}^\infty \mathscr E_{k,j,1}}\oplus \overline{\bigoplus _{k=0}^\infty\bigoplus_{j=0}^\infty \mathscr E_{k,j,3}},
 \end{equation}
where the $G$-isotypic component $\mathscr E_{k,j,l}$ given by
\[
\mathscr E_{k,j,l}:=\{\cos (kt)\bold a +\sin(kt)\bold b: \bold a,\, \bold b\in V_{j,l}\}
\]
is modeled on the irreducible $G$-representation $\cV_{k,j,l}$ being (for $k>0$) the complexification of
$\cW_j\otimes \cU_l^-$ with the $k$-folded $S^1$-action. By taking advantage of the $\Gamma$-isotypic decomposition  of $V$, one gets the following critical set for \eqref{eq:exD4}:
\[
\Lambda=\left\{(\alpha, \beta):  \alpha=\frac{\bm s_{n,j}}{a-\mu_l},\; \beta=\frac{\bm s_{n,j}b}{k(a-\mu_l)}, \; j=0,1,\dots,\, \text{and} \; k\in \bn,  l=0,1,3 \right\}.
\]
Moreover, one can verify that $\mathfrak t_{k,j,l}(\alpha_o,\beta_o)=-1$ for $k>0$ and   $ \alpha_o=\frac{\bm s_{n,j}}{a-\mu_l}$, $\beta_o=\frac{\bm s_{n,j}b}{k(a-\mu_l)}$, $l=0,1,3$ (here we assume that $c$, $d>0$ are such that $a-\mu_k\not=0$, $l=0,1,3$). 
\vs
The following local bifurcation result for system \eqref{eq:exD4} is a direct consequence of  Theorem \ref{th:main1}.
\vs
\begin{theorem}\label{th:main-local-D4}
   Assume that the function  $\mathfrak f:\br\times \br^2\to \br^2$  satisfies conditions {\rm (e0)--(e2)}, $g:\bc\to \bc$ is an analytic odd function, and $c$, $d$ are positive constants such that $a-\mu_l\not=0$, $l=0,1,3$. Then, for every $(\alpha_o,\beta_o)\in  \Lambda$, there exists a branch of non-constant periodic solutions to \eqref{eq:exD4} bifurcating from $(\alpha_o,\beta_o,0)$. Moreover, if \
   \[
  \alpha_o=\frac{\bm s_{n,j}}{a-\mu_l},\;\; \beta_o=\frac{\bm s_{n,j}b}{k(a-\mu_l)}
   \]
   for some $k>0$, $j>0$, then (see Example \ref{ex:o2xd4xz2}):
   \begin{itemize}
   \item[(a)] If $l=0$, then there exists  a branch of non-constant periodic solutions to \eqref{eq:exD4} bifurcating from $(\alpha_o,\beta_o,0)$ with symmetries at least
   \[
   \mathscr H_o:=(D_{2j}\times D_4^p)^{(D^{D_j}_{2j}\times _{\bz_2}^{D^d_4}D_4^p)}\times ^{\bz_k}\bz_{2k}.
   \]
   
   \item[(b)] If $l=1$, then there exists  a branch of non-constant periodic solutions to \eqref{eq:exD4} bifurcating from $(\alpha_o,\beta_o,0)$ with symmetries at least
   \begin{align*}
 \mathscr H_1:=(D_{2j}\times D_2^p)^{(D_{2j}^{D_j}\times _{\bz_2}^{D_2^d}D_2^p)}\times ^{\bz_k}\bz_{2k};\\
\mathscr H_2:=(D_{2j}\times D_2^p)^{(D_{2j}^{D_j}\times _{\bz_2}^{\wt D_2^d}\wt D_2^p)}\times ^{\bz_k}\bz_{2k};\\
\mathscr H_3:=(D_{2j}\times D_2^p)^{(D_{4j}^{\bz_j}\times _{D_4}^{\bz_2^-}D_4^p)}\times ^{\bz_k}\bz_{2k}.
   \end{align*}
   \item[(c)] If $l=3$, then there exists  a branch of non-constant periodic solutions to \eqref{eq:exD4} bifurcating from $(\alpha_o,\beta_o,0)$ with symmetries at least
   \[
   \mathscr K_o:= (D_{2j}\times D_2^p)^{(D_{2j}^{D_j}\times_{\bz_2}^{D_4^{ z}}D_4^p)}\times ^{\bz_k}\bz_{2k}.
   \]
   \end{itemize}
\end{theorem}
\vs
For the %specified in Theorem \ref{th:main-local-D4} 
orbit types $\mathscr K_o$, 
$\mathscr H_o$, $\mathscr H_1$, $\mathscr H_2$ and $\mathscr H_3$,    one can formulate, based on Theorem \ref{th:main3}, a global result showing that there exist unbounded branches of non-constant solutions to \eqref{eq:exD4} (where  $k$ is assume to be odd) with these minimal orbit types. 
\vs

\section{Proofs of Main Results} \label{sec:main}
\subsection{Preparatory results}
In this subsection we assume that conditions \ref{i}, ($C$), \ref{c1}--\ref{c2} are satisfied. Suppose that $\beta_o\in \bm b(\alpha_o)$. Then by Lemma \ref{lem:cross-dom}, there exists $\ve>0$ such that the set $\mathscr U$ defined by \eqref{eq:setU} satisfies  \eqref{eq:lamb} and \eqref{eq:finite}. We can assume that $\ve<\delta$. Then, from Lemma \ref{lem:A_k}, it follows that for all $(\alpha,\beta)\in \br\times \br_+$ such that $0<\max\{|\alpha-\alpha_o|,|\beta-\beta_o)|\}<\ve$ one has
$\ker \mathscr A(\alpha,\beta)=\{0\}$. 
\vs
We introduce the following notation:
\begin{align}
\mathscr P_o&:=\{(\alpha,\beta)\in  \br\times \br_+: \max\{|\alpha-\alpha_o|,|\beta-\beta_o)|\}<\ve\},\label{eq:mathscr So}\\
\mathscr D_o&:=\{(\alpha,\beta,0)\in \br\times \br_+\times V: \max\{|\alpha-\alpha_o|,|\beta-\beta_o)|\}<\ve\},\label{eq:Do}
\end{align}
and for a fixed $\eta>0$ define
\begin{equation}\label{eq:setD}
\mathscr D_\eta:=\{(\alpha,\beta,v)\in  \br\times \br_+\times \mathscr E: (\alpha,\beta,0)\in \mathscr D_o,\; \|v\|<\eta\}.
\end{equation}
The set $\mathscr D_\eta$ is illustrated on Figure \ref{fig:setD}.

\begin{figure}[H]
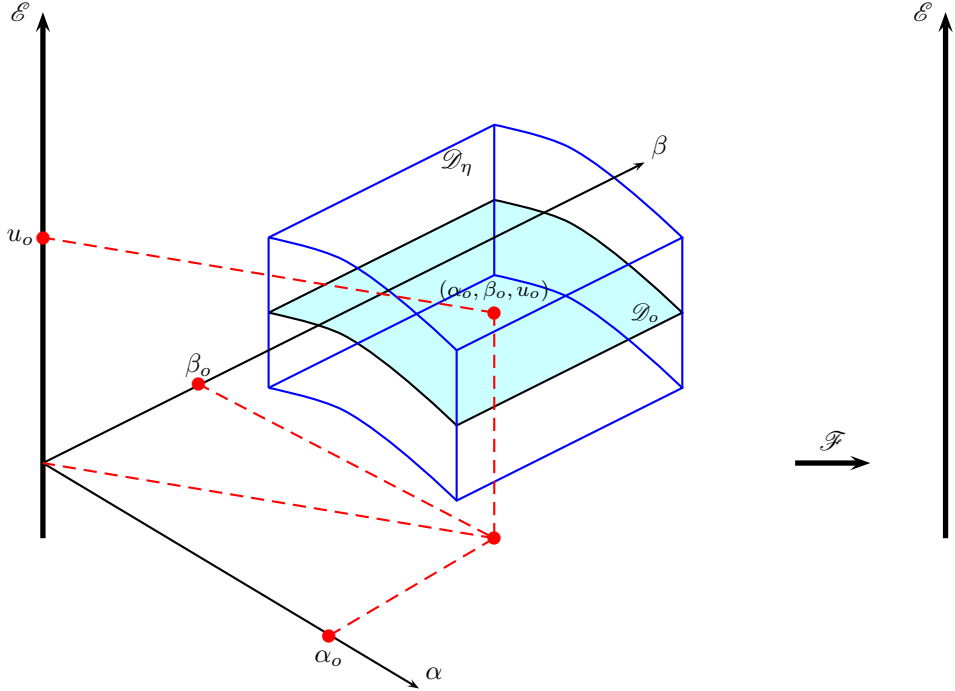

\vskip 5cm\hskip4.5cm
\pscurve[fillstyle=solid,fillcolor=lightblue](0,0)(1,-.3)(2.5,-1.5)(2.5,-1.5)(5.5,0.)(5.5,0.)(4,1.2)(3,1.5)(3,1.5)(0,0)
\rput(0,1){\pscurve[linecolor=blue](0,0)(1,-.3)(2.5,-1.5)(2.5,-1.5)(5.5,0.)(5.5,0.)(4,1.2)(3,1.5)(3,1.5)(0,0)
\psline[linecolor=blue](0,-2)(0,0)
\psline[linecolor=blue](2.5,-1.5)(2.5,-3.5)
\psline[linecolor=blue](3,-0.5)(3,1.5)
\psline[linecolor=blue](5.5,-2)(5.5,0)
}
\rput(0,-1){\pscurve[linecolor=blue](0,0)(1,-.3)(2.5,-1.5)(2.5,-1.5)(5.5,0.)(5.5,0.)(4,1.2)(3,1.5)(3,1.5)(0,0)}
\psline[linewidth=2pt]{->}(-3,-3)(-3,4)
\psline{->}(-3,-2)(5,2)
\psline{->}(-3,-2)(2,-5)
\psdots[dotsize=5pt,linecolor=red](3,0)(3,-3)(-.93,-.95)(.8,-4.3)(-3,1)
\psline[linecolor=red,linestyle=dashed](3,0)(3,-3)(-.93,-.95)
\psline[linecolor=red,linestyle=dashed](3,-3)(.8,-4.3)
\psline[linecolor=red,linestyle=dashed](3,-3)(-3,-2)
\rput(0,3){\psline[linecolor=red,linestyle=dashed](3,-3)(-3,-2)}
\rput(12,0){\psline[linewidth=2pt]{->}(-3,-3)(-3,4)\rput(-3.3,4){$\mathscr E$}}
\psline[linewidth=2pt]{->}(7,-2)(8,-2)
\rput(-3.3,4){$\mathscr E$}
\rput(7.5,-1.7){$\mathscr F$}
\rput(2.2,-4.8){$\alpha$}
\rput(5.2,2.2){$\beta$}
\rput(3.,0.3){\footnotesize$(\alpha_o,\beta_o,u_o)$}
\rput(-3.3,1){$u_o$}
\rput(-.93,-.7){$\beta_o$}
\rput(.8,-4.6){$\alpha_o$}
\rput(2.5,2.){$\mathscr D_\eta$}
\rput(5,0){\small$\mathscr D_o$}
\vskip5cm
\caption{The set $\mathscr D_\eta$.} \label{fig:setD}
\end{figure}

\vs
Put
\begin{equation}\label{eq:whD}
\wh{\mathscr D_\eta}:=\{(\alpha,\beta,v)\in  \br\times \br_+\times \mathscr E: (\alpha,\beta,0)\in \overline{\mathscr D_o},\; \|v\|=\eta\}.
\end{equation}
By an {\it auxiliary function} (with respect to $\mathscr D_\eta$) we mean a continuous $G$-invariant function $\theta:  \br\times \br_+\times \mathscr E\to \br$ such that
\begin{equation}\label{eq:auxiliary}
\begin{cases}
\theta(\alpha,\beta,u) >0 & \text{ for }\;\; (\alpha,\beta,u)\in \wh{\mathscr D_\eta},\\
\theta(\alpha,\beta,u) <0  & \text{ for }\;\; (\alpha,\beta,u)\in \overline{\mathscr D_o}\cap \partial {\mathscr D_\eta}.
\end{cases}
\end{equation}
Then for a given map $\mathfrak F: \br\times \br_+\times \mathscr E\to \mathscr E$ we will write $\mathfrak F^\theta$ to denote the map $\mathfrak F^\theta: \br\times \br_+\times \mathscr E\to \br\times \mathscr E$ given by 
\[
\mathfrak F^\theta(\alpha,\beta,u):=(\theta(\alpha,\beta,u),\mathfrak F(\alpha,\beta,u)).
\]
Notice that if  $\mathfrak F$ is a $G$-equivariant completely continuous field, then $\mathfrak F^\theta$ is also a $G$-equivariant completely continuous field.
\vs
Define by $\mathscr F_o:\overline{\mathscr D_\eta}\to \mathscr E$ the map
\begin{equation}\label{eq:F_o}
\mathscr F_o(\alpha,\beta,v):=\mathscr A(\alpha,\beta)v, \quad (\alpha,\beta,v)\in \overline{\mathscr D _\eta},
\end{equation}
where $\mathscr F$ is given by \eqref{eq:F}. Clearly, $\mathscr F_o$ is a $G$-equivariant compact field of $\overline{\mathscr D_\eta}$.

\vs
\begin{lemma} \label{lem:admiss} Under the assumptions  \ref{i},  ($C$), \ref{c0}--\ref{c2}, there exists  a sufficiently small $\eta>0$ such that for every auxiliary function $\theta$ with respect to  $\mathscr D_\eta$,  the homotopy $\mathfrak H^\theta:[0,1]\times \overline{\mathscr D_\eta} \to \br\times \mathscr E$ given by 
\[
\mathfrak H^\theta(\lambda,\alpha,\beta,u):=\Big(\theta(\alpha,\beta,u),\lambda \mathscr F(\alpha,\beta,u)+(1-\lambda) \mathscr F_o (\alpha,\beta,u)\Big),
\]
where  $(\lambda,\alpha,\beta,u)\in [0,1]\times \overline{\mathscr D_\eta}$, is a $\mathscr D_\eta$-admissible homotopy of $G$-equivariant compact fields.
\end{lemma}
\begin{proof}
Notice that 
\[
\partial \mathscr D_\eta:= \left\{ (\alpha,\beta,u)\in \overline{\mathscr D_\eta}: (\alpha,\beta)\in \partial \mathscr D_o     \right\}\cup \wh{\mathscr D_\eta}.
\]
By definition of the auxiliary function, if $\mathfrak H^\theta (\lambda,\alpha,\beta,u)=0$, then 
$(\alpha,\beta)\in \partial \mathscr D_o$ and $u=v$, $0<\|v\|<\eta$. Assume for contradiction that the homotopy $\mathfrak H^\theta$ is not $\mathscr D_\eta$-admissible for all $\eta>0$. Then, there exists a sequence $(\lambda_n,\alpha_n,\beta_n,v_n) \in [0,1]\times \partial \mathscr D_o\times B_\eta(0)$ such that $\lambda_n\to \lambda'$, $(\alpha_n,\beta_n)\to (\alpha',\beta')$ and $v_n\to 0$ as $n\to\infty$, and 
\[
\mathscr A(\alpha_n,\beta_n)v_n+\lambda_n \left( \mathscr F(\alpha_n,\beta_n,v_n)-\mathscr A(\alpha_m,\beta_n)v_n\right)=0.\] 
Since
\[\lim_{n\to\infty} \frac{\mathscr F(\alpha_n,\beta_n,v_n)-\mathscr A(\alpha_m,\beta_n)v_n}{\|v_n\|}=0,
\]
we obtain
\begin{equation}\label{eq:lim}
\lim_{n\to\infty} \mathscr A(\alpha_n,\beta_n)\left(\frac{v_n}{\|v_n\|}\right)=0.
\end{equation}
Then, by a standard argument, there exists a subsequence  $\frac{v_{n_k}}{\|v_{n_k}\|}$ convergent to $v'\in \mathscr E$. Since $\|v'\|=1$, \eqref{eq:lim} implies that
\[
\mathscr A(\alpha',\beta')v'=0\;\;\; \Leftrightarrow\;\;\; v'\in \ker \mathscr A(\alpha',\beta'),
\]
which contradicts the fact that $ \ker \mathscr A(\alpha',\beta')=\{0\}$.
\end{proof}
\vs
Suppose 
\[
\bm b(\alpha_o)=\{\beta_1,\beta_2,\dots, \beta_m\},
\]
then take $\beta_o\in \bm b(\alpha_o)$ and find $\eta>0$ such that the homotopy $\mathfrak H^\theta$ (given in Lemma \ref{lem:admiss}) is admissible. Put $\mathscr D(\alpha_o,\beta_o):=\mathscr D_\eta$ and define 
\begin{equation}\label{eq:inv}
\omega(\alpha_o,\beta_o,0):=\gdeg(\mathscr F^\theta,\mathscr D(\alpha_o,\beta_o)).
\end{equation}
The element $ \omega(\alpha_o,\beta_o,0)\in A_1^t(G)$ is called the {\it local bifurcation invariant}. One can easily notice that by standard properties of the twisted equivariant degree, $\omega(\alpha_o,\beta_o,0)$ does not depend on the choice of $\ve>0$, $\eta>0$ or the auxiliary function $\theta$. 
\vs
\begin{proposition}\label{pro:cont} Under the assumptions   \ref{i}, ($C$), \ref{c0}--\ref{c2}, if $(\alpha_o,0)$ is an isolated center, ${\beta_0}\in \bm b(\alpha_o)$ and   
\[
\omega(\alpha_o,\beta_o,0)=n_1(H_1)+n_2(H_2)+\dots + n_s(H_s)
\]
is non-zero, i.e. $n_k\not=0$ for some $k\in \{1,2,\dots,s\}$, then \eqref{eq:Hopf} undergoes a Hopf bifurcation 
at  $(\alpha_o,0)$ and there exists a branch $\mathscr C\subset\overline{\mathscr P }$ (cf. \eqref{eq:setS}) such that $(\alpha_o,\beta_o,0) \in \overline{\mathscr C}$ and {$\partial \mathscr D_\eta \cap \mathscr C^{H_k}\not=\emptyset$ ($\mathscr D_\eta$ is given by \eqref{eq:setD}, where $\eta >0$ is sufficiently small)}. 
\end{proposition}
\begin{proof}
We put $\beta_o:=\beta_k$ and consider the sets $\mathscr D_\eta=\mathscr D(\alpha_o,\beta_o)$ (see \eqref{eq:setD}), $\mathscr D_o$ given by \eqref{eq:Do} and $\wh {\mathscr D_\eta}$ given by \eqref{eq:whD}. We put $A_0:=\wh {\mathscr D_\eta}$, $A_1:=\overline{\mathscr D_o}$ and $K:=\mathscr F^{-1}(0)\cap\overline{\mathscr D_\eta}$, see Figure \ref{fig:setDn}.

\begin{figure}[H]
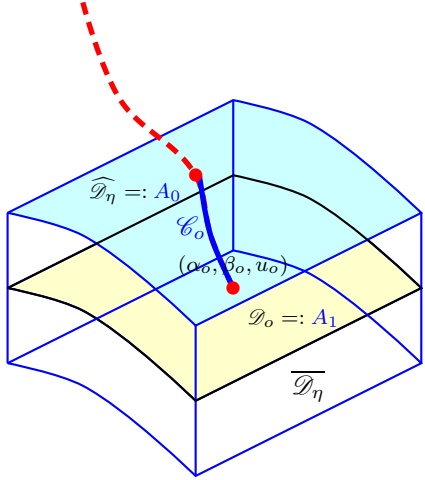

\vskip 3.7cm\hskip4.5cm
\pscurve[linecolor=blue,fillstyle=solid,fillcolor=lightyellow](0,0)(1,-.3)(2.5,-1.5)(2.5,-1.5)(5.5,0.)(5.5,0.)(4,1.2)(3,1.5)(3,1.5)(0,0)
\rput(0,1){\pscurve[linecolor=blue,fillstyle=solid,fillcolor=lightblue](0,0)(1,-.3)(2.5,-1.5)(2.5,-1.5)(5.5,0.)(5.5,0.)(4,1.2)(3,1.5)(3,1.5)(0,0)
\psline[linecolor=blue](0,-2)(0,0)
\psline[linecolor=blue](2.5,-1.5)(2.5,-3.5)
\psline[linecolor=blue](3,-0.5)(3,1.5)
\psline[linecolor=blue](5.5,-2)(5.5,0)
}
\pscurve(0,0)(1,-.3)(2.5,-1.5)(2.5,-1.5)(5.5,0.)(5.5,0.)(4,1.2)(3,1.5)(3,1.5)(0,0)
\rput(0,-1){\pscurve[linecolor=blue](0,0)(1,-.3)(2.5,-1.5)(2.5,-1.5)(5.5,0.)(5.5,0.)(4,1.2)(3,1.5)(3,1.5)(0,0)}
\pscurve[linewidth=2pt,linecolor=red,linestyle=dashed](3,0)(2.7,0.7)(2.5,1.5)(1.5,2.5)(1,3.8)
\pscurve[linewidth=2pt,linecolor=blue](3,0)(2.7,0.7)(2.55,1.5)
\psdots[dotsize=5pt,linecolor=red](3,0)(2.5,1.5)
\rput(3.,0.3){\footnotesize$(\alpha_o,\beta_o,u_o)$}
\rput(2.44,.8){\blue$\mathscr C_o$}
\rput(1.7,1.3){\footnotesize $\wh{\mathscr D_\eta}=:{\blue A_0}$}
\rput(4,-1.3){$\overline{\mathscr D_\eta}$}
\rput(3.8,-.4){\footnotesize $\mathscr D_o=:{\blue A_1}$}
\vskip2.8cm
\label{fig:setDn}
\caption{{The exit point of the branch $\mathscr C$.}}
\label{fig:setDn}
\end{figure}

\vs
Notice that $K\cap A_1\not=\emptyset$ and $\gdeg(\mathscr F^\theta, \mathfrak D_\eta)\not=0$, hence $K\cap A_0\not=\emptyset$. Assume for contradiction that {there is no compact connected set} $\mathscr C_o\subset K$ such that $\mathscr C_o\cap A_0\not=\emptyset\not=\mathscr C_o\cap A_1$ does not exist. Then, by Theorem \ref{th:G-Kur} (see Appendix \ref{appendix:Ref}), there exist two disjoint open sets $\mathscr U_0$ and $\mathscr U_1$ such that 
\[
A_0\subset \mathscr U_0, \quad A_1\subset \mathscr U_1,\quad  A_0\cup A_1\cup K\subset \mathscr U_0\cup\mathscr U_1.
\]
Put $K_0:=K\cap \mathscr U_0$ and $K_1:=K\cap \mathscr U_1$ and assume that $\mu:\br\times \br_+\times\mathscr E\to [0,1]$ is a continuous $G$-invariant function such that
\[
\mu(\alpha,\beta,u)=\begin{cases}
1 &\text{ if }\;\; (\alpha,\beta,u)\in K_0\cup A_0,\\
0 &\text{ if }\;\; (\alpha,\beta,u)\in K_1\cup A_1.\
\end{cases}
\]
Define the auxiliary function $\theta':\overline{\mathscr D_\eta}\to \br$ by 
\[
\theta'(\alpha,\beta,v):= \|v\|-\mu(\alpha,\beta,v)\eta.
\]
Since 
\[
\gdeg(\mathscr F^\theta,\mathscr D_\eta)=\gdeg(\mathscr F^{\theta'},\mathscr D_\eta)\not=0,
\]
the system 
\[
\begin{cases}
\mathscr F(\alpha,\beta,u)=0,\\
\theta'(\alpha,\beta, u)=0,
\end{cases} 
\]
has a solution $(\alpha',\beta', u')\in \mathscr D_\eta$. Therefore, either $(\alpha',\beta', u')\in K_0$ or $(\alpha',\beta', u')\in K_1$, which is impossible because in both cases $\theta'(\alpha',\beta', u')\not=0$. The contradiction proves the proposition.
\end{proof}
\vs

The local bifurcation invariant $\omega(\alpha_o, \beta_o,0)$ is closely related to the concept of crossing numbers $\mathfrak t_j(\alpha_o,\beta_k,0)$ (see \eqref{eq:cross-j} and \eqref{eq:inv-b}).
\vs
\begin{proposition}\label{pro-cross-inv} Under the assumptions  of Proposition \ref{pro:cont},  one has
\begin{equation}\label{eq:invar-cross}
\omega(\alpha_o,\beta_o,0) =\prod_{j=0}^r\text{\rm$\Gamma$-deg\,}(\mathscr A_{0,j}(\alpha_o,\beta_o),B^j_1(0))\cdot \sum_{k=1}^\infty \sum_{j=0}^r \mathfrak t_j(\alpha_o,k\beta_o,0)\deg_{\cV_{k,j}},
\end{equation}
where $B^j_1(0)$ stands for the unit ball in the $\Gamma$-isotypic component $V_j$, $\text{\rm$\Gamma$-deg\,}$ is the Leray-Schauder $\Gamma$-equivariant degree, $\deg_{\cV_{k,j}}$ is the basic twisted $G$-equivariant degree for the irreducible $G$-representation $\mathcal V_{k,j}$, and the (finite) summation is taken in the $A(\Gamma)$ module $A_1^t(G)$ (see Appendix for more details, definition and additional properties of the equivariant degree).
\end{proposition}
\begin{proof} The result presented in this proposition was proved for finite-dimensional $V$ in \cite{AED}, Proposition 9.34. In the setting relevant to our discussion, one needs to slightly modify the original arguments.   For the sake of clarity and completeness, we describe below the main ideas of the proof. \\

\noi {\it Step 1: Finite-dimensional reduction.} We start the computation of $\omega(\alpha_o,\beta_o,0)$ with constructing a finite-dimensional reduction of the map $\mathscr F^\theta$ (cf. \eqref{eq:inv}).
By Lemma \ref{lem:admiss}, one has 
\[
\omega(\alpha_o,\beta_o,0)=\gdeg(\mathscr F_1^\theta,\mathscr D(\alpha_o,\beta_o)),
\]
where
\[
\mathscr F_1^\theta(\alpha,\beta,u):=\Big((\theta(\alpha,\beta,u),\mathscr F_o (\alpha,\beta,u)\Big)
\]
(see \eqref{eq:F_o} for the definition of  $\mathscr F_o $). By using the $G$-isotypic decomposition of 
$\mathscr E$ (see \eqref{eq:iso-E}-\eqref{eq:iso-E-kj}), one has
\[
\mathscr F_o(\alpha,\beta,v_{k,j})= \mathscr A_{k,j}(\alpha,\beta)v_{k,j},
\]
where $v_{k,j}\in \mathscr E_{k,j}$. Since $(\alpha_o,0)$ is a non-degenerate isolated center,  $ \mathscr A_{0,j}(\alpha,\beta):V\to V$ is an isomorphism for all $(\alpha,\beta)\in \mathscr S_o$ (see \eqref{eq:mathscr So}). For $v_{k,j}\in \mathscr E_{k,j}$, put 
\[
\mathscr F'_o(\alpha,\beta,v_{k,j})=\begin{cases} \mathscr A_{0,j}(\alpha_o,\beta_o)v_{0,j} &\text{ if } j=0,1,\dots,r,\\
 \mathscr A_{k,j}(\alpha,\beta)v_{k,j} &\text{ if } k\ge 1, \; j=0,1,\dots,r,
 \end{cases}
\] 
and extend this map by continuity to the map 
\[
\mathscr F'_o(\alpha,\beta,v)=\sum_{j,k} \mathscr F'_o(\alpha,\beta, v_{k,j}), \qquad v=\sum_{k,j} v_{k,j}
\]
(we use the same letter for the extension).
Therefore, by homotopy property of the equivariant degree, one obtains 
\[
\omega(\alpha_o,\beta_o,0)=\gdeg(\mathscr F_2^\theta,\mathscr D(\alpha_o,\beta_o)),
\]
where 
\[
\mathscr F_2^\theta(\alpha,\beta,u):=\Big((\theta(\alpha,\beta,u),\mathscr F'_o (\alpha,\beta,u)\Big).
\]
%By applying a simple homotopy, one can assume without loss of generality that $u(\alpha)=0$ for $\alpha\in [\alpha_o-\rho,\alpha_o+\rho]$. 
Since the operator  $\mathscr F'_o(\alpha_o,\beta_o,\cdot ):\mathscr E\to \mathscr E$ is of the form $\id +\mathscr K$, where $\mathscr K$ is a compact operator, there exists  $N>0$ such that for all $k>N$, the operators  $\mathscr A_{k,j}(\alpha_o,\beta_o)$, $j=0,1,\dots, r$, are isomorphisms. Therefore, by using a deformation of $\mathscr S_o$ to the point $(\alpha_o,\beta_o)$, one obtain that the map $\mathscr F_2$ is $G$-equivariantly and admissibly homotopic to the map
\[
\mathscr F_3^\theta(\alpha,\beta,u):=\Big((\theta(\alpha,\beta,u),\mathscr F''_o (\alpha,\beta,u)\Big),
\] 
where 
\[
\mathscr F''_o(\alpha,\beta,v_{k,j})=\begin{cases}   \mathscr A_{0,j}(\alpha_o,\beta_o)v_{0,j} &\text{ if } j=0,1,\dots,r,\\
 \mathscr A_{k,j}(\alpha_o,\beta_o)v_{k,j}& \text{ if } k>N, \; j=0,1,\dots, r ,\\
 \mathscr A_{k,j}(\alpha,\beta)v_{k,j}& \text{  if } k\le N, \; j=0,1,\dots,r,
\end{cases}
\]
where $v_{k,j}\in \mathscr E_{k,j}$. We extend $\mathscr F''_o$ by continuity as above. 
Finally, since for $k>N$, the operators  $\mathscr A_{k,j}(\alpha_o,\beta_o)$ can be connected by a path in $GL^\Gamma(\cV_{k,j},\bc)$ with $\id_{|\cV_{k,j}} $, one obtains 
\[
\omega(\alpha_o,\beta_o,0)=\gdeg(\mathscr F_4^\theta,\mathscr D(\alpha_o,\beta_o)),\]
where 
\[\mathscr F_4^{\theta}(\alpha,\beta,\cdot)=\prod_{j=0}^r \mathscr A_{0,j}(\alpha_o,\beta_o)\times \Big(\theta(\alpha,\beta,\cdot),\mathscr F'''_o (\alpha,\beta,\cdot)\Big)
\]
and the map
 $\mathscr F'''_o(\alpha,\beta,\cdot ): \mathscr E\to \mathscr E$ is defined by 
\[
\mathscr F'''_o(\alpha,\beta,v_{k,j})= 
\begin{cases}
v_{k,j}&\text{ if } k\ge N ,\; j=0,1,\dots,r,\\
 \mathscr A_{k,j}(\alpha,\beta)v_{k,j}& \text{  if }\;1\le k\le N, \; j=0,1,\dots,r, 
\end{cases}
\]
($v_{k,j}\in \mathscr E_{k,j}$). Put 
\begin{equation}\label{eq:E-o}
\mathscr E_o:=\bigoplus_{k=1}^N\bigoplus_{j=0}^r \mathscr E_{k,j}.
\end{equation}
Then, for every $v\in \mathscr E_o$, one has  $\mathscr F'''_o(\alpha,\beta,v)\in \mathscr E_o$.
By the product property of the equivariant degree, one has 
\begin{align*}
\omega(\alpha_o,\beta_o,0)&=
\gdeg(\mathscr F_4^\theta, D(\alpha_o,\beta_o))\\
&=\prod_{j=0}^r\text{\rm$\Gamma$-deg\,}(\mathscr A_{0,j}(\alpha_o,\beta_o),B^j_1(0))\cdot \gdeg(\mathscr F_*^\theta, D_*(\alpha_o,\beta_o)),
\end{align*}
where 
\begin{equation}\label{eq:D-ast}
D_*(\alpha_o,\beta_o):=\{(\alpha,\beta,v)\in \mathscr S_o\times \mathscr E_o: \|v\|<\eta\}
\end{equation}
and 
\begin{equation}\label{eq:F-star}
\mathscr F_*^\theta(\alpha,\beta,v)= \left(\theta(\alpha,\beta,v),\sum_{j=0}^r\sum_{k=1}^N
 \mathscr A_{k,j}(\alpha,\beta)v_{k,j}\right)= \left(\theta(\alpha,\beta,v),\sum_{j=0}^r\sum_{k=1}^N
\triangle_\alpha^j(ik\beta)v_{k,j}\right)
\end{equation}
for $v\in \mathscr E_o$, $\|v\|\le \eta$,  $(\alpha,\beta)\in \mathscr S_o$. 
\vs
Notice that the characteristic operator $\triangle_{\alpha_o}(\lambda):V^c\to V^c$ is a compact perturbation of identity depending analytically on $\lambda$. Take an arbitrary $\zeta>0$ and consider the characteristic  decompositions of $V^c$ for $\triangle_{\alpha_o}$. Then, there exists a finite-dimensional $\Gamma$-invariant subspace $V_*$ of $V^c$ such that
\[\forall_{\beta_k\in \bm b(\alpha_o)}\;\;\; \ker \triangle_{\alpha_o}(i\beta_k)\subset V_*\]
and 
\[
\forall_{(\alpha,\beta)\in\mathscr S_o}\;\;\; \|\triangle_{\alpha}(i\beta)-P_o\triangle_{\alpha}(i\beta)\|<\zeta,
\]
where $P_o$ is the orthogonal projection of $V^c$ onto $V_*$. 
Define the  finite-dimensional space
\[
\wt{\mathscr E}_o:=\bigoplus_{k=1}^N\bigoplus_{j=0}^r \wt{\mathscr E}_{k,j}, \quad 
\wt{\mathscr E}_{k,j}:={{\mathscr E}_{k,j}}\cap V_*.
\]
Take a sufficiently small $\zeta> 0$  and put (cf.  \eqref{eq:E-o}, \eqref{eq:D-ast} and \eqref{eq:F-star})
\begin{align*}
\widetilde{D}_*(\alpha_o,\beta_o)&:=\{(\alpha,\beta,v)\in \mathscr S_o\times \wt{\mathscr E}_o \ : \, \|v\|<\eta\}, 
\quad 
\widetilde{\theta}:= \theta |_{\widetilde{D}_*(\alpha_o,\beta_o)},\\
\widetilde{\mathscr A}_{k,j}(\alpha_o,\beta_o)&:= P_o \mathscr A_{k,j}(\alpha_o,\beta_o)|_{\wt{\mathscr E}_{k,j}}, \quad k = 1,...,N, \quad j = 0,1,...,r,\\
\widetilde{\mathscr F}_*^\theta(\alpha,\beta,v) &:= \left(\widetilde{\theta}(\alpha,\beta,v),\sum_{j=0}^r\sum_{k=1}^N
\widetilde{\mathscr {A}}_{k,j}(\alpha,\beta)v_{k,j}\right), \quad v_{k,j}\in {\wt{\mathscr E}_{k,j}},
\end{align*}
for $v\in \wt{\mathscr E}_o$, $\|v\|\le \eta$,  $(\alpha,\beta)\in \mathscr S_o$.  
Then,
\begin{equation}\label{eq:fin-approks}
\omega(\alpha_o,\beta_o,0) = \prod_{j=0}^r\text{\rm$\Gamma$-deg\,}(\mathscr A_{0,j}(\alpha_o,\beta_o),B^j_1(0))\cdot \gdeg(\widetilde{\mathscr F}_*^\theta,\widetilde{D}_*(\alpha_o,\beta_o)).
\end{equation}

\vs
\noi{\it Step 2: Computation of $\gdeg(\widetilde{\mathscr F}_*^\theta,\widetilde{D}_*(\alpha_o,\beta_o))$. } By Splitting Lemma (cf. \cite{AED}, Lemma 4.21), one obtains
\[
 \gdeg(\widetilde{\mathscr F}_*^\theta,\widetilde{D}_*(\alpha_o,\beta_o))= \sum_{j=0}^r\sum_{k=1}^N\gdeg(\widetilde{\mathscr F}_{k,j}^\theta,\widetilde{D}_{k,j}(\alpha_o,\beta_o)),
\]
where 
\begin{align*}
\widetilde{\mathscr F}_{k,j}^\theta (\alpha,\beta,v) &:= \left(\widetilde{\theta}(\alpha,\beta,v),
\triangle_\alpha^j(ik\beta)v_{k,j}\right), \quad\quad  (\alpha,\beta)\in \mathscr S_o, \, v_{k,j}\in {\wt{\mathscr E}_{k,j}},\\
\widetilde{D}_{k,j}(\alpha_o,\beta_o)&:= \mathscr S_o\times {\wt{\mathscr E}_{k,j}}\cap 
\widetilde{D}_*(\alpha_o,\beta_o)).
\end{align*}
Since 
\[
\gdeg(\widetilde{\mathscr F}_{k,j}^\theta,\widetilde{D}_{k,j}(\alpha_o,\beta_o)) = 
\mathfrak k_{k,j} \cdot \deg_{\cV_{k,j}},
\] 
where 
\[
\mathfrak k_{k,j} = \deg\big({\det}_{\mathbb C} \circ \mathscr A_{k,j}, D_o\big), \qquad 
D_o = \text{int}(\mathscr S_o)
\]
(see \cite{AED}, Subsection 9.3.3), it suffices to show that
\[
\mathfrak k_{k,j} =  \mathfrak t_j(\alpha_o,k\beta_o,0).
\]
To this end, put  
\[
\Omega:=\{(\alpha,\beta,\tau)\in \br^3: |\alpha-\alpha_o|<\ve,\; |\beta-\beta_o|<\ve, \; 0<\tau< \ve\}
\]
(see \eqref{eq:setU} and \eqref{eq:lamb}), and
define $\Upsilon:\overline{\Omega}\to \bc$ by 
\[
\Upsilon_{k,j}(\alpha,\beta,\tau):= {\det}_\bc\Big(\triangle_\alpha^j(\tau + ik\beta )\Big), \quad (\alpha,\beta,\tau)\in \overline{\Omega}.
\]
Also, put
\[
D_\pm:=\{(\alpha_\pm,\beta,\tau): |\beta-\beta_o|<\ve,\; 0<\tau<\ve\}.
\]
By \eqref{eq:lamb}, for any $(\alpha,\beta,\tau)\in \partial \Omega$, if $\Upsilon_{k,j}(\alpha,\beta,\tau)=0$, then $(\alpha,\beta,\tau)\in D_-\cup D_o\cup D_+$.  Therefore, by standard argument (see \cite{AED}, Subsection 9.2.6), one obtains:
\begin{equation}\label{eq:crucial1}
\deg(\Upsilon_{k,j}|_{D_-}, D_-)-\deg(\Upsilon_{k,j}|_{D_+}, D_+)=\deg(\Upsilon_{k,j}|_{D_o}, D_o)
\end{equation}
(provided that $\partial \Omega$ and $D_\pm\subset \partial \Omega$ are oriented positively, so that the Brouwer degrees in \eqref{eq:crucial1} are well-defined).
On the other hand, 
\begin{align*}
\deg(\Upsilon_{k,j}|_{D_-}, D_-)&=\mathfrak t_j^-(\alpha_o,k\beta_o, 0),\\
\deg(\Upsilon_{k,j}|_{D_+}, D_+)&=\mathfrak t_j^+(\alpha_o,k\beta_o, 0),\\
\deg(\Upsilon_{k,j}|_{D_o}, D_o)&= \deg\big({\det}_{\mathbb C} \circ \mathscr A_{k,j}, D_o\big)=\mathfrak k_{k,j},
\end{align*}
and the conclusion of Proposition \ref{pro-cross-inv} follows.
\end{proof}

\vs

We will use the notation 
\[\zhong:=\{(\alpha,\beta,u)\in \mathscr \br^2_+\times V: (\alpha,u) \text{ is a center  for \eqref{eq:Hopf} with limit frequency $\beta$}\}\]
%\end{definition}

\vs
%\subsection{Global Behavior of Branches of Solutions for \eqref{eq:Hopf}}
Following the classical result of P. Rabinowitz the subsequent theorem is  called the {\it Rabinowitz Alternative}.
\vs

%\vs
\begin{theorem}\label{th:Rab}{\rm (Rabinowitz Alternative)} 
Suppose   \ref{i} and { ($C$)}  are satisfied and  assume the map $F:\br\times W\to H$  satisfies  condition \ref{c0}.  Let   $\mathscr U\subset \overline{\mathscr U}\subset \br^2_+\times \mathscr E$ be an open bounded  $G$-invariant set such that:
\begin{itemize}
\item[(i)] $\zhong\cap \partial \mathscr U=\emptyset$;
\item[(ii)] for every  $(\alpha_o,\beta_o,0)\in \zhong \cap \mathscr U$, $(\alpha_o,0)$ satisfies conditions \ref{c1}---\ref{c4}.
\end{itemize}
Let $\mathscr C$ be a branch of non-trivial periodic solutions to \eqref{eq:Hopf1} such that
$\zhong \cap \mathscr U\cap \overline{\mathscr C}\not=\emptyset$. Then, one has the following alternative:
\begin{itemize}
\item[(a)] either $\overline{\mathscr C} \cap \partial \mathscr U\not=\emptyset$, 
\item[(b)] or the set $\zhong \cap \mathscr U\cap \overline{\mathscr C}$ is finite, i.e.
\[
\zhong \cap \mathscr U\cap \overline{\mathscr C}=\{(\alpha_1,\beta_1,u_1),(\alpha_2,\beta_2,u_2),\dots,(\alpha_m,\beta_m,u_m)\},
\]
and
\begin{equation}\label{eq:Rab}
\sum_{k=1}^m \omega(\alpha_k,\beta_k,u_k)=0  
\end{equation}  
\end{itemize}
(cf. \eqref{eq:inv}).
\end{theorem}

\begin{proof} Suppose that $\overline{\mathscr C}\cap\partial \mathscr U=\emptyset$. Since $\mathscr F$ is a completely continuous vector field and $\mathscr U$ is bounded, the set $\mathscr F^{-1}(0)\cap \overline{\mathscr U}$ is compact. By conditions (ii), %; \ref{c1}-\ref{c4}), 
the set $\zhong \cap \mathscr U\cap \overline{\mathscr C}$ is finite and for every 
$(\alpha_k,\beta_k, u_k)\in \zhong \cap \mathscr U\cap \overline{\mathscr C}$, the invariant $\omega(\alpha_k,\beta_k,u_k)$ is well-defined. 

\vs

Let us show that relation \eqref{eq:Rab} is satisfied.
%First, notice that since $\overline{\mathscr U}$ is bounded, thus by Lemma \ref{lem:Vid2},  $F(\alpha, \ii(u))$ is globally  Lipschitz on $\overline{\mathscr U}$. Therefore, by Lemma \ref{lem:Vid1}, there exists a constant $c>0$ such that if $(\alpha,\beta,u)$ is a non-constant solution to \eqref{eq:Hopf2} such that $(\alpha,u)\in \overline{\mathscr U}$ then $0<\beta<c$. We put  $X:=\overline{\mathscr W}$, where $\Omega:=\{(\alpha,\beta,u): (\alpha,u)\in \mathscr U, \; 0< \beta<c\}$. Then, we 
To this end, consider the sets
\[
 K:=(\overline{\mathscr P}\setminus \overline{ \mathscr C})\cap\overline{\mathscr U}, 
 \quad  
 A_0:=\partial \mathscr U \cup K, 
 \quad 
A_1:=\{(\alpha_k,\beta_k,u_k): \; k=1,2,\dots,m\}=\zhong\cap \overline{\mathscr C}. 
%\quad K:=(\mathscr %P\setminus \mathscr C)\cap\overline{\mathscr U}.
\]
Notice that $K$ is compact (recall, $\overline{\mathscr C}$ is a connected component of 
$\overline{\mathscr P}$)  and $A_0\cap A_1=\emptyset$. 
% Notice that the set $\mathscr L:=\overline{\Omega}\cap \overline{\mathscr P }$ is compact. We denote by $K$ the union of all connected components $L$ of   $\mathscr L$ such that $L\subset \Omega$. Then clearly $K$ is compact and $K\subset \Omega$. 
 Then, by Kuratowski's Lemma \ref{th:G-Kur}, there exist two disjoint  $G$-invariant open sets $\mathscr U_0$ and $\mathscr U_1$ in $\overline{\mathscr U}$ such that $A_0\subset \mathscr U_0$, $A_1\subset \mathscr U_1$ and
\[
A_0\cup A_1\cup \overline{\mathscr C}\subset \mathscr U_0\cup \mathscr U_1.
\]

\begin{figure}
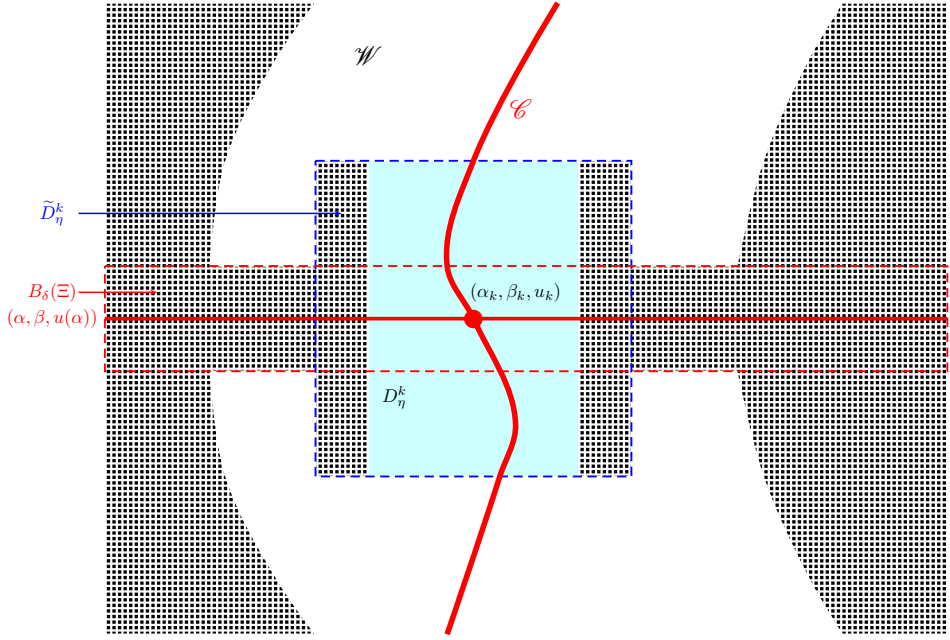

\vskip4.cm
\hskip7cm
\scalebox{.7}{\psline[linecolor=white,fillstyle=crosshatch*,fillcolor=black,
hatchcolor=white,hatchwidth=1.2pt,hatchsep=1.8pt, hatchangle=0](-7,6)(9,6)(9,-6)(-7,-6)(-7,6)
\pscurve[linecolor=white,fillstyle=solid,fillcolor=white](-3,6)(-3.7,5)(-4.7,3)(-5,1)(-5,1)(-3,1)(-3,1)(-3,3)(-3,3)(3,3)(3,3)(3,1)(3,1)(5,1)(5,1)(5.5,3)(6.4,5)(7,6)(7,6)(-3,6)
\pscurve[linecolor=white,fillstyle=solid,fillcolor=white](-3,-6)(-3.7,-5)(-4.7,-3)(-5,-1)(-5,-1)(-3,-1)(-3,-1)(-3,-3)(-3,-3)(3,-3)(3,-3)(3,-1)(3,-1)(5,-1)(5,-1)(5.5,-3)(6.4,-5)(7,-6)(7,-6)(-3,-6)
\psline[linecolor=white,fillstyle=solid,fillcolor=lightblue](-2,3)(2,3)(2,-3)(-2,-3)(-2,3)
\psline[linewidth=2pt,linecolor=red](-7,0)(9,0)
\psline[linewidth=1pt,linecolor=blue,linestyle=dashed](-3,3)(3,3)(3,-3)(-3,-3)(-3,3)
\pscurve[linewidth=3pt,linecolor=red](1.6,6)(0,3)(-.5,1)(0,0)(.8,-2)(.5,-3)(-.5,-6)
\psdots[linewidth=4pt,linecolor=red](0,0)
\psline[linewidth=1pt,linecolor=red,linestyle=dashed](-7,1)(9,1)(9,-1)(-7,-1)(-7,1)
\rput(.9,4){\Large\color{red}$\mathscr C$}
\rput(-2,5){\Large$\mathscr W$}
\rput(-8,.5){\red$B_\delta(\Xi)$}\psline[linecolor=red]{->}(-7.5,.5)(-6.,.5)
\rput(-8,2){\blue$\wt D_\eta^k$}\psline[linecolor=blue]{->}(-7.5,2)(-2.5,2)
\rput(-1.5,-1.5){$D_\eta^k$}
\rput(.8,.5){$(\alpha_k,\beta_k,u_k)$}
\rput(-8,0){\red$(\alpha,\beta,u(\alpha))$}
}
\vskip4cm
\caption{Construction of the set $\mathscr W$.}\label{fig:W-set}
\end{figure}

\vs
Choose sufficiently small numbers $\eta>0$, $\wt \ve>\ve>0$, and  for every center $(\alpha_k,\beta_k,u_k)$ consider the sets
\[
\mathscr D_\eta^k\subset \wt{\mathscr D}_\eta^k,
\]
where (see \eqref{eq:setD})
\begin{align*}
\wt{\mathscr D}_\eta^k&:=\{(\alpha,\beta,v): (\alpha,\beta)\in \wt{\mathscr S}_o^k, \; \|v\|<\eta\},\\
\mathscr D_\eta^k&:=\{(\alpha,\beta,v): (\alpha,\beta)\in \mathscr S_o^k, \; \|v\|<\eta\},
\end{align*}
and (see \eqref{eq:mathscr So})
\begin{align*}
\wt{\mathscr S}_o^k&:=\{(\alpha,\beta): \max\{|\alpha-\alpha_k|,|\beta-\beta_k|<\wt \ve \},\\
\mathscr S_o^k&:=\{(\alpha,\beta): \max\{|\alpha-\alpha_k|,|\beta-\beta_k|<\ve \}.
\end{align*}
Here, we assume that $ \wt{\mathscr D}_\eta^k$ are isolating neighborhoods, i.e. 
\begin{equation}\label{eq:isolating-neighborhood}
\overline{\wt{\mathscr D}_\eta^k}\cap \overline{\wt{\mathscr D}_\eta^{k'}}=\emptyset \quad \text{ for } \;\; k\not=k'.
\end{equation}
Moreover, one can choose the numbers $\wt\ve$ and $\eta$ to be so small that 
$ \wt{\mathscr D}_\eta^k\subset \mathscr U_1$ for all $k=1,2,\dots,m$.
Put (see \eqref{eq:equi-set}) 
\begin{equation}\label{eq:set-Xi-centres}
\Xi:=\{(\alpha,\beta, u)\in \overline{\mathscr U_1},\; (\alpha, u)\in \Xi_o\}
\end{equation}
and choose a sufficiently small $\delta>0$ (with $\eta/2>\delta$) such that 
\begin{equation}\label{eq:exclude-centres}
(\alpha,\beta,u)\in \overline{B_\delta(\Xi)}\cap \overline{\mathscr C}\;\;\Rightarrow \;\; (\alpha,\beta,u)\in \overline{\wt{\mathscr D}_\eta^k} \;\text{ for some } \; k\in \{1,2,\dots,m\},
\end{equation}
where $B_\delta(\Xi)$ stands for the $\delta$-neighborhood of $\Xi$.
Put
\begin{equation}\label{eq:set-W}
\mathscr W:=\text{int}\left(\left(\mathscr U_1\setminus\Big(\bigcup_{k=1}^m\overline{\wt{\mathscr D}_\eta^k}\cup  \overline{B_\delta(\Xi)}\Big)\right) \cup \bigcup_{k=1}^m \overline{\mathscr D_\eta^k}\right)
\end{equation}
(see Figure \ref{fig:W-set}). By construction (cf. \eqref{eq:set-Xi-centres}, \eqref{eq:exclude-centres} and \eqref{eq:set-W}), the map
$\mathscr F$ (see \eqref{eq:F}) has zeros only on $\partial \mathscr W \cap \bigcup_{k=1}^m \overline{\mathscr D_\eta^k}$. Take a $G$-invariant function $\vp: \overline{\mathscr W}\to \br$ given by
%\begin{align}
\begin{equation}\label{eq:auxil-eta}
\vp(\alpha,\beta,u)=\begin{cases}
\|v\|-\frac\eta 2, &\text{ if }\; (\alpha,\beta,u)\in \overline{\mathscr D_\eta^k},\; u=v,\\
\frac\eta 2, &\text{ if } \;\; (\alpha,\beta,u)\in\overline{ \mathscr W} \setminus  \bigcup_{k=1}^m \mathscr D_\eta^k,
\end{cases}
\end{equation}
%\end{align}
and define $\mathscr F^\vp:\overline{\mathscr W} \to \br\times \mathscr E$ by
\begin{equation}\label{eq:last-F-vp}
\mathscr F^\vp(\alpha,\beta,u):=\big(\vp(\alpha,\beta,u), \mathscr F(\alpha,\beta,u)\big).
\end{equation}
Then, $\mathscr F^\vp$ is a $G$-equivariant $\mathscr W$-admissible compact field 
on $\overline{\mathscr W}$. Therefore, the $G$-equivariant degree $\gdeg(\mathscr F^\vp,\mathscr W)$ is well-defined. Also, by construction and definition of the local invariant,  
\begin{align*}
\gdeg(\mathscr F^\vp,\mathscr W)&=\gdeg\left(\mathscr F^\vp, \bigcup_{k=1}^m\mathscr D_\eta^k\right)\\
&=\sum_{k=1}^m\gdeg(\mathscr F^\vp, \mathscr D_\eta^k)=\sum_{k=1}^m\omega(\alpha_k,\beta_k,u_k).
\end{align*} 

\begin{figure}[H]
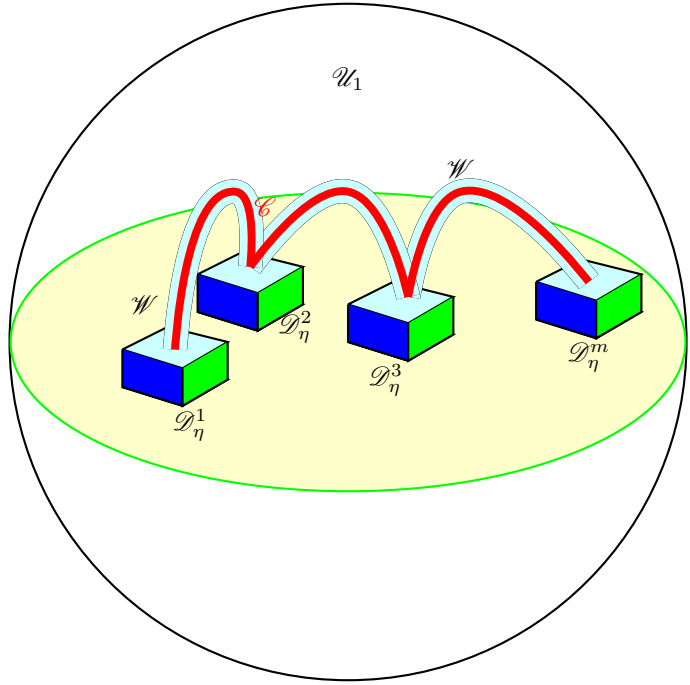

\vglue5cm\hskip7.5cm
\pscircle(0,0){4.5}
\psellipse[linecolor=green,fillstyle=solid,fillcolor=lightyellow](0,0)(4.5,2)
\psline(0,0)(.8,-.25)(1.4,0.1)(.6,0.3)(0,0)
\rput(0,.5){\psline[fillcolor=lightblue,fillstyle=solid](0,0)(.8,-.25)(1.4,0.1)(.6,0.3)(0,0)}
\psline[fillcolor=blue,fillstyle=solid](.8,.25)(.8,-.25)(0,0)(0,.5)
\psline[fillcolor=green,fillstyle=solid](.8,.25)(.8,-.25)(1.4,0.1)(1.4,0.6)
\rput(-2,.4){\psline(0,0)(.8,-.25)(1.4,0.1)(.6,0.3)(0,0)
\rput(0,.5){\psline[fillcolor=lightblue,fillstyle=solid](0,0)(.8,-.25)(1.4,0.1)(.6,0.3)(0,0)}
\psline[fillcolor=blue,fillstyle=solid](.8,.25)(.8,-.25)(0,0)(0,.5)
\psline[fillcolor=green,fillstyle=solid](.8,.25)(.8,-.25)(1.4,0.1)(1.4,0.6)
}
\rput(-3,-.6){\psline(0,0)(.8,-.25)(1.4,0.1)(.6,0.3)(0,0)
\rput(0,.5){\psline[fillcolor=lightblue,fillstyle=solid](0,0)(.8,-.25)(1.4,0.1)(.6,0.3)(0,0)}
\psline[fillcolor=blue,fillstyle=solid](.8,.25)(.8,-.25)(0,0)(0,.5)
\psline[fillcolor=green,fillstyle=solid](.8,.25)(.8,-.25)(1.4,0.1)(1.4,0.6)
}
\rput(2.5,.3){\psline(0,0)(.8,-.25)(1.4,0.1)(.6,0.3)(0,0)
\rput(0,.5){\psline[fillcolor=lightblue,fillstyle=solid](0,0)(.8,-.25)(1.4,0.1)(.6,0.3)(0,0)}
\psline[fillcolor=blue,fillstyle=solid](.8,.25)(.8,-.25)(0,0)(0,.5)
\psline[fillcolor=green,fillstyle=solid](.8,.25)(.8,-.25)(1.4,0.1)(1.4,0.6)
}
\pscurve[linewidth=9.2pt,linecolor=black](.8,.6)(1.5,2)(3.2,.8)
\pscurve[linewidth=9pt,linecolor=lightblue](.8,.6)(1.5,2)(3.2,.8)
\pscurve[linewidth=9.2pt,linecolor=black](.8,.6)(0,2)(-1.3,1)
\pscurve[linewidth=9pt,linecolor=lightblue](.8,.6)(0,2)(-1.3,1)
\pscurve[linewidth=9.2pt,linecolor=black](-1.3,1)(-1.5,2)(-2.3,-.1)
\pscurve[linewidth=9pt,linecolor=lightblue](-1.3,1)(-1.5,2)(-2.3,-.1)
\pscurve[linewidth=3pt,linecolor=red](.8,.6)(0,2)(-1.3,1)
\pscurve[linewidth=3pt,linecolor=red](-1.3,1)(-1.5,2)(-2.3,-.1)
\pscurve[linewidth=3pt,linecolor=red](.8,.6)(1.5,2)(3.2,.8)
\rput(3.2,-.2){$\mathscr D_\eta^m$}
\rput(-2.1,-1.1){$\mathscr D_\eta^1$}
\rput(0.5,-.5){$\mathscr D_\eta^3$}
\rput(-0.7,.2){$\mathscr D_\eta^2$}
\rput(0,3.5){$\mathscr U_1$}
\rput(-1.1,1.8){\red$\mathscr C$}
\rput(1.5,2.3){$\mathscr W$}
\rput(-2.7,.5){$\mathscr W$}
\vskip5cm
\caption{The branch $\mathscr C$ and its neighborhood $\mathscr W$.}\label{fig:K}
\end{figure}

\vs
Finally, due to \eqref{eq:auxil-eta},  one has  the following $\mathscr W$-admissible homotopy
$\mathscr H:[0,1]\times \overline{\mathscr W}\to \br\times \mathscr E$:
\[
\mathscr H(\lambda,\alpha,\beta,u):=\big((1-\lambda)\vp(\alpha,\beta,u)-\lambda, \mathscr F(\alpha,\beta,u)\big), \quad \lambda\in [0,1].
\]
Clearly, $\mathscr F^\vp(\cdot,\cdot,\cdot) =  \mathscr H(0,\cdot,\cdot,\cdot)$ and    $\mathscr H(1,\cdot,\cdot,\cdot)$ does not admit zeros in $\overline{\mathscr W}$. Hence, 
\[
\gdeg(\mathscr F^\vp,\mathscr W)=\gdeg(\mathscr H_1,\mathscr W)=0,
\]
and the proof is complete.
\end{proof}
\vs
Before giving the proofs of the main results, let us introduce the following notation: for a twisted orbit type $(K)$ in $\mathscr E$, define the function $\text{\rm coeff}^K:A_1^t(G)\to \bz$ by 
\begin{equation}\label{eq:coeff}
\text{\rm coeff}^K(b)=n_K, \quad \text{ where }\; b=\sum_{(L)}n_L(L)
\end{equation}
(see Appendix \ref{appendix:Twist}).
\vs
\subsection{Proof of the Local Hopf Bifurcation Theorem}
In this subsection, we present the proof of Theorem \ref{th:main1}.
\vs
\begin{proof} ($i_1$) and ($i_2$): By Proposition \ref{pro-cross-inv}, one has 
\[
\omega(\alpha_o,\beta_o,0)=a\cdot \sum_{k=1}^\infty \sum_{j=0}^r \mathfrak t_j(\alpha_o,k\beta_o,0)\deg_{\cV_{k,j}},
\]
where $a$ is an invertible element in the Burnside ring $A(\Gamma)$ (see Appendix \ref{appendix:Twist} for more details).
Put
\[
b:=\sum_{k=1}^\infty \sum_{j=0}^r \mathfrak t_j(\alpha_o,k\beta_o,0)\deg_{\cV_{k,j}},
\]
and observe that, due to maximality of $(K)$, one has
\begin{align*}
\text{\rm coeff}^K(b)&=  \sum_{j=0}^r \mathfrak t_j(\alpha_o,k\beta_o,0)\text{\rm coeff}^K(\deg_{\cV_{k,j}}).
\end{align*}
Notice that if $j\not=j'$ are such that $\text{\rm coeff}^K(\deg_{\cV_{k,j}})$, $\text{\rm coeff}^K(\deg_{\cV_{k,j'}})\not=0$, then $\text{\rm coeff}^K(\deg_{\cV_{k,j}})=\text{\rm coeff}^K(\deg_{\cV_{k,j'}})=: c$ (see Appendix \ref{appendix:Twist}).
%{\color{blue}Proposition \ref{pro:twisted-deg}} 
%{\color{red} Which proposition?} 
Therefore, 
\[
\text{\rm coeff}^K(\omega(\alpha_o,\beta_o,0))=\sum_{j=0}^r c \, \mathfrak t_j(\alpha_o, k\beta_o, 0) \zeta_{k,j}(K)= c \, \mathfrak t_K^k(\alpha_o,\beta_o,0)\not=0,
\]
and the statements ($i_1$) and ($i_2)$ follow from Proposition \ref{pro:cont}. 
%($i_3$) and ($i_4$). Compactness of $\overline{\mathscr C}$ implies that $\overline{\mathscr C}$ is closed and bounded in $\br\times \br_+\times \mathscr E$, thus there exists a bounded set $\mathscr U\subset \overline{\mathscr U}\subset \br\times \br_+\times \mathscr E$ such that $\overline{\mathscr C}\subset \mathscr U$, $\mathfrak p(\mathscr U)=\Omega$. It is clear that the set  $\mathscr U$ satisfies the assumptions of Theorem \ref{th:Rab} and, consequently, the conclusions ($i_3$) and ($i_4$) follow.
\end{proof}
\vs

\subsection{Proof of the Global Hopf Bifurcation Theorem}
In this subsection, we present the proof of Theorem \ref{th:main2}, which essentially is a corollary of 
%the proof of 
Theorem  \ref{th:Rab}. 
\vs

\begin{proof} Observe that the results formulated in Theorem \ref{th:Rab} rely only on the properties of the map $\mathscr F$ and the set $\mathscr U$. This implies that one can apply the scheme of the proof of Theorem \ref{th:Rab} in a slightly different setting. More precisely, replace the space $\mathscr E$ by $\mathscr E^{\mathfrak H}$, the map $\mathscr F$ by $\mathscr F^{\mathfrak H}$ and the set $\mathscr U$ by the set $\mathscr W\subset \br^2_+\times \mathscr E^{\mathfrak H}$. In this setting, the set $\Xi$ reduces to 
\[
\Xi^{\mathfrak H} := \{(\alpha,\beta,0): \alpha\in \br, \; \beta>0\}.
\]
In this way, the proof of statements (a) and (b) follows immediately from the proof of Theorem \ref{th:Rab}. 
\end{proof}
\vs

\appendix
\section{Accretive Operators and Isomorphism Theorem}\label{appendix:Acc-Iso}
\subsection{Accretive operators}\label{subsec:abst-spaces} 
Let $H$ be a real Hilbert space and let $|\cdot|$ stand for the norm in $H$ and $\langle\cdot,\cdot\rangle$ 
for the inner product in $H$. The complexification of $H$ 
(denoted $H^c:=H\oplus i H$) is equipped with the complex inner product $[\cdot,\cdot]$ given by 
\begin{equation}\label{eq:complex-inner-product}
[u+iv,x+iy]:=\langle u,x\rangle+\langle v,y\rangle+i(\langle v,x\rangle-\langle  u,y\rangle), \quad u,\, v,\, x,\, y\in H.
\end{equation}
Clearly, the formula 
%\eqref{eq:complex-inner-product}, 
\begin{equation}\label{eq:real-direct-product}
\langle z,w\rangle := \re [z,w], \qquad (z,w\in H^c)
\end{equation}
defines a {\it real} inner product on $H\oplus H$.\\

Assume that $A:{\mathfrak{D}}(A)\subset  H\to  H$ is a real unbounded linear operator. The operator $A$ is said to be {\it accretive}\footnote{$A$ is also called a {\it monotone} operator according to the terminology used by Brezis \cite{Brezis}} if
\begin{equation}\label{eq:accretive}
\forall_{v\in {\mathfrak{D}}(A)} \;\;\; \langle Av,v\rangle \ge 0.
\end{equation} 
\vs

An  {\it accretive} operator  $A:{\mathfrak{D}}(A)\subset  H\to  H$ is called {\it $m$-accretive} if for every $\lambda>0$ one has $\im(A+\lambda \id)=H$. Actually, for an accretive operator, it is sufficient to assume that $\im(A+\id)=H$ to deduce that ${\mathfrak{D}}(A)$ is dense, $A$ is closed and $\im(A+\lambda \id)=H$ for all $\lambda >0$ (see Brezis \cite{Brezis}, Proposition 7.1). In such a case, one can easily deduce that $(A+\lambda \id)^{-1}:H\to H$ exists and 
$\|(A+\lambda\id)^{-1}\|\le \frac 1\lambda$. Also, if $A :{\mathfrak{D}}(A)\subset  H\to  H$ is $m$-accretive and $B : H \to H$ is $m$-accretive and bounded, then $A+B$ defined on ${\mathfrak{D}}(A)$ is $m$-accretive as well. Finally, 
%However, 
if an accretive operator  $A:{\mathfrak{D}}(A)\subset  H\to  H$  is self-adjoint, then it is automatically $m$-accretive. Moreover, the following statement is true.
\vs

%This Proposition is well-known in literature, we should refer to Kato, Problem 3.32, chapter 5, section 3. See also https://www.mat.tuhh.de/veranstaltungen/isem18/Phase_1_The_lectures.html
\begin{proposition}\label{pro:accretive} Let $A:{\mathfrak{D}}(A)\subset  H\to  H$ be an unbounded self-adjoint operator satisfying \eqref{eq:accretive}. Then, for every $\lambda\in \bc$ with $\re \lambda>0$, the complexification $A:=A^c:{\mathfrak{D}}(A^c)\subset H^c\to H^c$ satisfies the following properties (see \eqref{eq:real-direct-product}):
\begin{itemize}
\item[(a)] for all $z\in {\mathfrak{D}}(A^c)$,  one has $[ (A+\lambda \id)z, z] \ge \re \lambda |z|^2$ (in particular,  $A+\lambda\id$ is injective);
\item[(b)] $\im(A+\lambda \id)=H^c$;
\item[(c)] $\left\|(A+\lambda \id)^{-1}\right\|\le \frac 1{\re \lambda}$.% and  is considered as  $(A+\lambda \id)^{-1}:H\to H$).
\end{itemize}
\end{proposition}
\begin{proof} For the unbounded operator $A+\lambda \id:{\mathfrak{D}}(A^c)\subset H^c\to H^c$, if the inverse $(A+\lambda \id)^{-1}$ exists, then it will be considered as an operator from $H^c$ into itself. 

(a) From \eqref{eq:complex-inner-product} and \eqref{eq:real-direct-product} it follows immediately that for every $\lambda\in \bc$ and $z\in {\mathfrak{D}}(A^c)$, one has: 
\begin{equation}\label{eq:eq}
[(A+\lambda \id)z,z]=\langle Az+\re \lambda z,z \rangle+i\text{Im} \lambda |z|^2.
\end{equation}
%(a): Notice that $\langle z,w\rangle:=\re [z,w]$, $z,w\in H^c$, is a real inner product on $H^c=H\oplus H$,
Hence, 
\begin{equation}\label{eq:re-acc}
[ (A+\lambda \id)z, z ] = \langle Az,z\rangle+(\re \lambda) |z|^2\ge \re \lambda |z|^2,
\end{equation}
where for $z=x+iy\in H^c$, $|z|^2=|x|^2+|y|^2$.
%which implies that the operator $A+\lambda\id$ is injective. 
\\
(b) First, let us show that $\im (A+\lambda \id)$ is closed. To this end, suppose $z_n\in {\mathfrak{D}}(A^c)$ is a sequence such that $Az_n+\lambda z_n\to z_o$. We claim that $\{z_n\}$ is bounded. Indeed, if $|z_n|\to \infty$ then
\[
A\left(\frac{z_n}{|z_n|}\right)+\lambda \frac{z_n}{|z_n|}\to 0 \;\;\; \text{ as }\;\;\; n\to\infty,
\]
but on the other hand, by \eqref{eq:re-acc} one has
\[
\left[A\left(\frac{z_n}{|z_n|}\right)+\lambda \frac{z_n}{|z_n|},\frac{z_n}{|z_n|}\right] \ge \re \lambda>0,
\]
which leads to a contradiction. Next, since $\{z_n\}$ is bounded, one can assume without loss of generality  that $z_n\rightharpoonup z_o$ (i.e. $z_n$ converges weakly to $z_o$). By assumption, $A$ is self-adjoint, and as such is closed. Since the graph of $A$ is a closed linear subspace, it is also weakly closed. Since  $Az_n+\lambda z_n\to z_o$, one has
\[
z_n\rightharpoonup z_o, \;\; Az_n+\lambda z_n\rightharpoonup  y_o, \;\;\; \text{ as }\;\; n\to \infty,
\]
it follows that  $z_o\in \mathfrak{D}(A^c)$ and $(A+\lambda \id)z_o=y_o$. 
Since $(A+\lambda \id)^*=A+\overline\lambda\id$, 
\[
\im(A+\lambda\id)=\ker(A+\overline \lambda\id)^\perp =\{0\}^\perp =H^c.
\]
(c) By inequality \eqref{eq:re-acc}, one has
\[
|(A+\lambda\id)z|\,|z|\ge \langle(A+\lambda \id)z,z \rangle\ge( \re \lambda) |z|^2,
\]
which implies 
\begin{equation}\label{eq:ineq}
|(A+\lambda\id)z|\ge (\re \lambda) |z|.
\end{equation}
By putting $u=(A+\lambda \id)z$, we get $z=(A+\lambda\id)^{-1}u$, so (by \eqref{eq:ineq})
\[
\forall_{0\not=u\in H^c}\;\;\; |u|\ge (\re \lambda) |(A+\lambda\id)^{-1}u|\;\; \Leftrightarrow\;\;\;
\frac 1{\re\lambda}\ge \frac{|(A+\lambda \id)^{-1}u|}{|u|},
\]
which leads to 
\[
\frac 1{\re\lambda}\ge \|(A+\lambda \id)^{-1}\|.
\]
\end{proof}
\vs

In what follows, for a self-adjoint accretive operator  $A:{\mathfrak{D}}(A)\subset  H\to  H$, we denote by  $V$ the Hilbert space ${\mathfrak{D}}(A)$ equipped with the {\it graph inner product}, i.e.
\[ \langle u,v\rangle_A:=\langle u,v\rangle +\langle Au,Av\rangle, \quad u,\, v\in {\mathfrak{D}}(A).\]
 We denote by $|\cdot|_A$  the graph norm induced by  $\langle \cdot,\cdot\rangle_A$  and by   $\jj : V \to  H$ the natural embedding of  $V$ into $H$. 
Next, define   the following  two spaces of $p$-periodic functions: 
\begin{equation}\label{eq:per-spaces}
\mathscr E_p := H^1_p(\mathbb R; H)\cap  L^2_{p}(\mathbb R; V),  \quad \mathscr H_p := L^2_{p}(\mathbb R; H),
\end{equation}
where the space $\mathscr E_p$
% := H^1_{p}(\mathbb R; V)$ 
 is equipped with the  inner product  
 \begin{equation}\label{eq:Ep-inner}
 \bm{\langle} u,v\bm{\rangle}_{\mathscr E}:=\int_0^p\left(\langle \dot u(t),\dot v(t)\rangle+\langle u(t),v(t)\rangle +\langle Au(t),Av(t)\rangle\right)dt, \quad u,\, v\in \mathscr E_p,
 \end{equation}
 where $\dot u:=\frac{du}{dt}$.
  For an elementary exposition of the Sobolev spaces for vector-valued functions, we refer to \cite{MarcelKreuter}. Notice that the spaces $ L^2_{p}(\mathbb R; H)$ and $L^2([0,p];H)$ can be identified. 
\vs

\subsection{Isomorphism Theorem}\label{sec:isya}
%Consider the natural embedding
%\begin{equation}\label{eq:embedding-j}
%\bm{j} : \mathscr E_p \to \mathscr H_p.
%\end{equation}
%\vs
We will refer to the following theorem as the so-called {\it Isomorphism Theorem}:
\vs

\begin{theorem}{\rm (Isomorphism Theorem)}\label{th:iso-thm} 
Let  $A:{\mathfrak{D}}(A)\subset  H\to  H$ be 
an accretive self-adjoint 
%operator unbounded 
%an $m$-accretive operator unbounded 
linear operator such that the operator $\mathfrak j:V\to H$ is compact. Then, for every $p>0$,  the operator 
\[
\frac{d}{dt}+A:\mathscr E_p\to \mathscr H_p
\]
is an isomorphism.
\end{theorem}
\begin{proof}  
Then, by Hille-Yoshida theorem (cf. \cite{Brezis}, Thm 7.4), for every $u_o\in {\mathfrak{D}}(A)$, there exists a unique 
continuously differentiable function $u:[0,\infty)\to H$ satisfying $u(t)\in {\mathfrak{D}}(A) $ for all $t\ge 0$, and 
\begin{equation}\label{eq:lin-hom}
\begin{cases}
\frac{du}{dt}+Au=0 \;\; & \text{ on $[0,\infty)$},\\
u(0)=u_o. 
\end{cases}
\end{equation}
Moreover, the function $u(t)$ satisfies $|u(t)|\le |u_o|$ and $\left|\frac{du}{dt}\right|=|Au(t)|\le |Au_o|$ for all $t\ge 0$.  Then, for a fixed $t\ge 0$, one can define the linear operator $T_A(t):H\to H$  by $T_A(t)u_o=u(t)$, where $u(t)$ is the solution to \eqref{eq:lin-hom}. Since $T_A(t)$ is bounded on the dense subset ${\mathfrak{D}}(A)$, it has the unique  continuous extension to $H$, which we denote by the same symbol $T_A(t)$. In this way, one obtains a family $\{T_A(t)\}_{t\ge 0}\subset L(H)$  (here $L(H)$ stands for the space of bounded linear operators from $H$ to $H$) satisfying the following properties:
\begin{itemize}
\item[(a)] for each $t\ge 0$,  $\|T_A(t)\|\leq 1$; 
\item[(b)] for all $t_1$, $t_2\in [0,\infty)$,  one has 
\begin{equation}\label{eq:semigroup}
\begin{cases}
T_A(t_1+t_2)=T_A(t_1)\circ T_A(t_2),\\
T_A(0)=\id,
\end{cases}
\end{equation}
\item[(c)] for every $u_o\in H$, one has
\begin{equation}\label{eq:limnit}
\lim_{t\to 0^+} \left| T_A(t)u_o-u_o\right|=0.
\end{equation}
\end{itemize}
%The family $\{T_A(t)\}_{t\ge 0}$ is called a {\it continuous semigroup of contractions} generated by $A$. 
%Assume that $p>0$ and  let $f:[0,p]\to H$ be a function from $L^2([0,p];H)$. Then, for a given $u_o\in D(A)$, the non-homogenous problem 
%\begin{equation}\label{eq:lin-nonhom}
%\begin{cases}
%\frac{du}{dt}(t)+Au(t)=f(t) \;\; & \text{ for  $t\in [0,p]$},\\
%u(0)=u_o. 
%\end{cases}
%\end{equation}
%has a unique solution $u(t)$ given by
%\[
%u(t)=T_A(t)u_o+\int_0^tT_A(t-s)f(s)ds, \quad t\in [0,p].
%\]
%It is possible to show that $u\in H^1([0,p],H)\cap L^2([0,p],V)$ (see \cite{Brezis}, Theorem 7.10 and related comments).
Then, by applying the result of Kielh\"ofer (cf. \cite{Kiel}, Lemma 1.2 and formula (1.10)) we obtain that $\frac{d}{dt}+A:\mathscr E_p\to \mathscr H_p$ is a Fredholm operator of index zero, and consequently it is an isomorphism.
\end{proof}
\vs
\begin{remark}{\rm The family $\{T_A(t)\}_{t\ge 0}$ mentioned in the proof of Theorem \ref{th:iso-thm} is called a {\it continuous semigroup of contractions} generated by $A$. Its existence, being crucial for the proof of  Lemma 1.2 from \cite{Kiel}, is well-documented for more general classes of operators $A$ than the ones considered in Theorem \ref{th:iso-thm}.} 
\end{remark}
\vs 
\begin{remark}\label{rem:analyt} {\rm 
It follows immediately from Proposition \ref{pro:accretive} that under the condition \ref{i}, provided that  $\re \lambda>-\boldsymbol{\delta}$, the operator $R(A,\lambda):=(A +\lambda\id)^{-1}:H^c\to V^c $ is a well-defined bounded $\Gamma$-equivariant  linear  operator, and since $V^c\subset H^c$,  $R(A,\lambda)$ is  also  bounded as an operator from $H^c$ into $H^c$. Moreover, keeping in mind  that  a resolvent operator of a closed unbounded operator is analytic on the resolvent set (see Kato \cite{Kato}, Ch. IV, Thm 3.11), one concludes that $R(A,\cdot)$  is analytic.}
%We complete this section with a statement which is a direct consequence of the following classical result: a resolvent operator of closed %unbounded operator is analytic on the resolvent set (see Kato \cite{Kato}, Ch. IV, Thm 3.11). More precisely, one has
%\vs
%Under the conditions \ref{o}--\ref{i}, consider the operator $A +\lambda \id :D(A^c) \to H^c$, where $\re \lambda>-\boldsymbol{\delta}$. Then, 
%$R(A,\lambda):=(A +\lambda\id)^{-1}:H^c\to H^c $ is a well-defined bounded $\Gamma$-equivariant  linear  operator. Moreover, the map $R(A,
%\cdot):\{ \lambda\in \bc: \re \lambda>-\boldsymbol{\delta}\}\to L(H^c)$ is analytic.}
\end{remark} 
\vs

\section{Reference Results}\label{appendix:Ref}
\subsection{$G$-Equivariant Kuratowski Lemma}

Let us recall the (non-equivariant) background  from the monograph \cite{Kur}. 
\vs

\begin{definition}\rm\label{def:continuum}
 Let  $\mathscr E$ be a Banach space and suppose  $K\subset \mathscr E$. Then, the set $K$ is called a {\it continuum} if $K$ is compact and connected. Moreover, a continuum $K$ is said to be {\it nontrivial}   if it contains at least two points. 
 \end{definition}
\vs

One can easily notice that any nontrivial continuum is uncountable. Also,  the following concept generalizes   the property of connectedness. 
\vs

\begin{definition}[\rm \cite{Kur}, p.142] \label{def:between}\rm A compact space $K$  is said to be {\it connected between} its  subsets $A_0$ and $A_1$, if there is no closed-open set $F \subset K$ such that $A_0\subset F$ and $F\cap A_1=\emptyset$. 
\end{definition}

\smallskip
The {\it connectedness between} relation is symmetric and it is easy to notice  that if $K$ is not connected between $A_0$ and $A_1$, then there exist two disjoint  open-closed subsets $F_0, F_1 \subset K$ such that $A_0\subset F_0$,  $A_1\subset F_1$ and $F_0\cup F_1=K$. Using this fact, one can deduce the following statement (cf. \cite{Kur}, Theorem  5, p. 144).
\vs

\begin{proposition}\label{pro:between}  Assume that a compact  space $K$ is not connected between its two non-empty subsets $A_0$ and $A_1$. Then, there exists a continuous function $\vp:K\to \{0,1\}$ such that $\vp^{-1}(0)\supset A_0$, $\vp^{-1}(1)\supset A_1$.
\end{proposition}
\vs

The following well-known result (cf. \cite{Kur}, Theorem 3, p. 170) is a starting point for our discussion.
\vs
\begin{theorem}\label{thm:orig-Kurat}
Assume that a compact space $K$ is connected between its two closed non-empty subsets $A_0$ and $A_1$. Then, there exists a connected component  $C$ of $K$ such that 
\[
C\cap A_0\not=\emptyset\not=C\cap A_1.
\]
\end{theorem}
\vs

In what follows, we will use a slight modification of Theorem \ref{thm:orig-Kurat} following below. 
\vs
\begin{theorem}\label{th:Kuratowski} Let $X$ be a metric space, $B_0$, $B_1\subset X$ two disjoint closed subsets, and $K \subset X$ a compact subset  such that $K\cap B_0\not=\emptyset\not= K\cap B_1$.  Assume that the set $K$ does not contain a connected component $C$ such that $C \cap B_0\not=\emptyset\not= B_1\cap C$. Then, there exists a continuous function $\psi:X\to [0,1]$ such that  $B_0 \subset \psi^{-1}(0)$, $B_1\subset \psi^{-1}(1)$ and $ B_0 \cup B_1\cup K\subset \psi^{-1}([0,\delta)\cup(1-\delta,1])$ for some $0< \delta<\frac 12$.
\end{theorem}
\begin{proof} Put $A_0:=B_0\cap K$ and $A_1:=B_1\cap K$. Then, it follows from assumptions and Theorem \ref{thm:orig-Kurat} that  $K$ is not connected between $A_0$ and $A_1$. Then, by Proposition \ref{pro:between}, there exists a continuous function $\vp:K\to \{0,1\}$ such that $\vp^{-1}(0)\supset A_0$, $\vp^{-1}(1)\supset A_1$. Since the sets $B_0$ and $B_1$ are disjoined, one can extend continuously the function $\vp$ to $\wt \vp:K\cup B_0\cup B_1\to \{0,1\}$ so  $\wt \vp^{-1}(0)\supset B_0$, $\vp^{-1}(1)\supset B_1$. Next, by Tietze extension  theorem, there exists an extension $\psi :X\to [0,1]$ of $\wt \vp$. Since $\wt \vp(K\cup B_0\cup B_1)\subset \{0,1\}$, for every $0<\delta<\frac 12$, one has  $ B_0 \cup B_1\cup K\subset \psi^{-1}([0,\delta)\cup(1-\delta,1])$.

\end{proof}
\vs

The following statement will be referred to as {\it equivariant Kuratowski's Lemma}.
\vs
\begin{theorem}\label{th:G-Kur}   Let $\mathscr E$ be an isometric Banach  $G$-representation, $B_0$, $B_1\subset \mathscr E$ two $G$-invariant disjoint closed sets in $\mathscr E$, and $K$ a $G$-invariant compact set in $\mathscr E $ such that $K\cap B_0\not=\emptyset\not= K\cap B_1$.  If the set $K$ does not contain a connected component $C$ such that $C \cap B_0\not=\emptyset\not= C\cap B_1$, then there exist two $G$-invariant disjoint open sets $U_0$, $U_1$ such that $B_0 \subset U_0$, $B_1\subset U_1$ and $ B_0 \cup B_1\cup K\subset U_0\cup U_1$.
\end{theorem}
\begin{proof}
Take a function $\psi$ provided by Theorem \ref{th:Kuratowski} and put
\[
\wt{\psi}(x):= \int_G \psi(gx)d\mu (g), \quad x \in \mathscr E.
\]
Clearly, $U_0 : = \wt{\psi}^{-1}([0,\delta))$ and $U_1:=  \wt{\psi}^{-1}((1-\delta,1])$ are as required.
\end{proof}
\vs

\subsection{Equivariant finite-dimensional reduction}
The following Proposition provides technical means for the equivariant finite-dimensional reductions which are needed for effective computation of the related $G$-equivariant bifurcation invariants:
\vs

\begin{proposition}\label{prop:a-map-finite} Assume that $K$ is a compact metric space and  $\mathscr A : K \to \text{\rm GL}_c^G(\mathbb E)$ is a continuous map. Then, there exists a $G$-invariant decomposition $\mathbb E:=\mathbb E_o\oplus\mathbb E^o$ with $\text{\rm dim}(\mathbb E_o) < \infty$ and a homotopy $h : [0,1]\times K\to  \text{\rm GL}_c^G(\mathbb E)$ such that for all $\lambda \in K$, one has 
\begin{itemize}
\item[(i)] $h(0,\lambda) = \mathscr A(\lambda)$;
\item[(ii)] $h(1,\lambda)(\mathbb E_o)\subset \mathbb E_o$, $h(1,\lambda)(\mathbb E^o)\subset \mathbb E^o$, and  $h(1,\lambda)|_{\mathbb E^o}=\id_{\mathbb E^o}$,
\end{itemize}
where $\lambda\in K$.
\end{proposition} 
\begin{proof} Since this result in the non-equivariant case was proved in \cite{KW} (see Theorem 4.4.5), we only present the main steps of the construction of the homotopy $h$. Also,
in contrast to  \cite{KW}, we avoid the usage of (infinite dimensional) fiber bundles, and make the construction more explicit.
\vs
Put $\bm A(\lambda):=\id-\mathscr A(\lambda)$, $\lambda\in K$, and define $b : [0,1]\times K\to \text{L}_c^G(\mathbb E)$ by 
\[
b(t,\lambda):=\id-t\bm A(\lambda), \quad t\in [0,1],\; \lambda\in K. 
\]
Also, put $X:=[0,1]\times K$
%, $X_o:=\{1\}\times K$ 
and $x:=(t, \lambda)\in [0,1]\times K$. Finally, denote by $\text{Pr}^G(\mathbb E)$ the set of all 
linear bounded equivariant projections in $\mathbb E$.
\vs
The construction of $h : X\to \text{GL}^G_c(\mathbb E)$ is conducted in several steps. 
\vs

{\bf Step \#1:} For every  fixed $x_o\in X$, the operator $b(x_o):\mathbb E\to \mathbb E$ is a Fredholm operator of index zero, thus there exists a $G$-invariant closed subspace $\mathbb E(x_o)$ such that 
\[
\mathbb E=\Ker b(x_o)\oplus \mathbb E(x_o). 
\]
Denote by $P_o$ the $G$-equivariant projection of $\mathbb E(x_o)$ with $\Ker P_o=\Ker b(x_o)$.  Similarly, $\Im(b(x_o))$ is a closed $G$-invariant subspace of finite co-dimension, thus there exists a finite-dimensional $G$-invariant subspace $\mathbb F(x_o)$ such that 
\[
\mathbb E=\Im(b(x_o))\oplus \mathbb F(x_o),
\]
and $\dim \mathbb F(x_o)=\dim \Ker b(x_o)$. 
%Since $\mathbb F(x_o)$ and $ \Ker b(x_o)$ are $G$-equivalent sub-representations of $\mathbb E$, there exists a $G$-equivariant isomorphism $B_o:\Ker(b(x_o))\to \mathbb F(x_o)$. 
Notice that, since $b$ is a continuous map, there exists a neighborhood $V_o$ of $x_o$ in $X$ such that for every $x\in V_o$, one has
\begin{equation*}
 \mathbb E = \mathbb E(x_o)\oplus \Ker\, b(x) = b(x) \mathbb E(x_o)\oplus \mathbb F (x_o).
\end{equation*}
Indeed, since $\mathbb F(x_o)$ and $ \Ker b(x_o)$ are $G$-equivalent sub-representations of $\mathbb E$, there exists a $G$-equivariant isomorphism $B_o : \Ker(b(x_o))\to \mathbb F(x_o)$. Define the map $\psi_o : X \to \text{\rm L}_c ^G(\mathbb E)$ by 
\[
\psi_o(x)v:=b(x)v + B_o {(\id-P_o)}v,\quad x\in X, \; v\in \mathbb E.
\]
Since $\psi_o(x_o)\in \text{\rm GL}_c^G(\mathbb E)$, by continuity of $\psi_o$, there exists an open neighborhood  $V_o$ of $x_o$ in $X$  such that for all $x\in V_o$ one has $\psi_o(x)\in \text{\rm GL}_c^G(\mathbb E)$.
\vs

{\bf Step \#2:}  Combining the argument from Step 1 with the compactness of $X$, one can find 
finitely many points $x_1,\dots,x_n \in X$ and their corresponding neighborhoods 
$V_1,\dots,V_n$ (open in $X$) such that:
\begin{itemize}
\item[(a)] $X =  \bigcup \limits_{i=1}^n V_i$;
%a finite set $\{x_1,x_2,\dots,x_n\}\subset X$ and an open cover $\{V_i\}_1:i=1,2,\dots\}$ of $X$,
\item[(b)] for any $i = 1,\dots,n$, there exist $G$-invariant closed subspaces $\mathbb E(x_i)$ and $\mathbb F(x_i)$ satisfying 
\[
\mathbb E=\mathbb E(x_i)\oplus \Ker b(x_i) = \Im(b(x_i))\oplus \mathbb F(x_i),
\]
as well as a $G$-equivariant isomorphism $B_i :  \Ker(b(x_i))\to \mathbb F(x_i)$;
%with the associate projection $P_i\in \text{Pr}^G(\mathbb E)$ onto $\mathbb E(x_i)$;
\item[(c)] for any  $i = 1,\dots,n$ and $x \in V_i$, one has
%$G$-invariant finite-dimensional subspaces $\mathbb F(x_i)$ satisfying 
%\[
%\mathbb E=\Im(b(x_i))\oplus \mathbb F(x_i),
%\]
\[
\mathbb E=\mathbb E(x_i)\oplus \Ker b(x) = \Im(b(x))\oplus \mathbb F(x_i).
\]
\end{itemize}
%such that for all $x\in V_i$ one has 
%\begin{align*}
%\mathbb E=\mathbb E(x_i)\oplus \Ker b(x),\\
%\mathbb E=\Im(b(x))\oplus \mathbb F(x_i).
%\end{align*}
\vs

{\bf Step \#3:} For each $i = 1,\dots, n$, take $P_i\in \text{Pr}^G(\mathbb E)$ projecting 
$\mathbb E$ onto $\mathbb E(x_i)$ (such a projection can be easily constructed by averaging an arbitrary (bounded) projection onto $\mathbb E(x_i)$). Then, for each $i = 1,\dots, n$, and  $x\in V_i$, define
\[
c_i(x):\mathbb E(x_i)\oplus \Ker b(x_i)\to b(x)\mathbb E(x_i)\oplus \mathbb F(x_i)
\]
by
\[
c_i(x)v:=b(x)P_i(v)+B_i(\id-P_i)v, \quad v\in \mathbb E.
\]
Clearly, $c_i(x)\in \text{GL}_c^G(\mathbb E)$, $x\in V_i$. 

Next, put
\[
\mathbb E^o:=\bigcap_{i=1}^n \mathbb E(x_i)
\]
and observe that $\mathbb E^o \subset \mathbb E$ is a closed $G$-invariant subspace of finite 
co-dimension. Take $Q\in \text{Pr}^G(\mathbb E)$ projecting $\mathbb E$ onto $\mathbb E^o$ and put $\mathbb E_o:=\Ker Q$.   
%Then, for $i\in \{1,2,\dots,n\}$ and  $x\in V_i$, define
%\[
%c_i(x):\mathbb E(x_i)\oplus \Ker b(x_i)\to b(x)\mathbb E(x_i)\oplus \mathbb F(x_i)
%\]
%by
%\[
%c_i(x)v:=b(x)P_i(v)+B_i(\id-P_i)v, \quad v\in \mathbb E.
%\]
%Clearly, $c_i(x)\in \text{GL}_c^G(\mathbb E)$, $x\in V_i$. 
Define $p_i:V_i\to \text{Pr}^G(\mathbb E)$ by
\begin{equation*}
p_i(x):=c_i(x)\circ Q \circ P_i\circ (c_i(x))\one, \quad x\in V_i,
\end{equation*}
and notice that
\[
p_i(x)\mathbb E=b(x)\mathbb E^o, \quad x\in V_i.
\]
\vs

{\bf Step \#4:} Take a partition of unity $\{\vp_i\}_{i=1}^n$ subordinate to the cover $\{V_i\}_{i=1}^n$ and put
\[
q(x):=\sum_{i=1}^n \vp_i(x)p_i(x), \quad x\in X.
\]
Then $q(x)\in \text{Pr}^G(\mathbb E)$ for $x\in X$ and $q(x)\mathbb E=b(x)\mathbb E^o$.
Choose $0<\ve<1$ be such that for every $x=(t,\lambda)\in (\ve,1]\times K=:\wt X$ one has  $b(x)\in \text{GL}_c^G(\mathbb E)$. Therefore,  for every $x\in \wt X$, the linear operator $q(x):\mathbb E\to \mathbb E$, given by
\[
q(x):= b(x)\circ Q\circ b(x)\one, \quad x\in \wt X,
\]
is a well define $G$-equivariant projection on $b(x)\mathbb E^o$. Then,
take   a continuous function $\alpha:[0,1]\to [0,1]$ such that $\text{supp}(\alpha)\in (\ve,1]$ and $\alpha(1)=1$, and  define the projection $p(x)\in \text{Pr}^G(\mathbb E)$ for $x=(t,\lambda)\in [0,1]\times K$ by
\[
p(t,\lambda):=(1-\alpha(t))q(t,\lambda)+\alpha(t)\wt q(t,\lambda).
\]
Clearly, $p(x)$ is a $G$-equivariant projection of $\mathbb E$ onto $b(x)\mathbb E^o$ such that
$p(x)|_{\mathbb E^o} :\mathbb E^o\to b(x)\mathbb E^0$ is a $G$-equivariant isomorphism.  Moreover,
for every $x\in X$
\[
\mathbb E=\Ker p(x)\oplus b(x)\mathbb E^o,
\]
and $\id-p(x)|_{\mathbb E_o}:\mathbb E_o\to \Ker p(x)$ is also a $G$-equivariant isomorphism.
\vs

{\bf Step \#5:} The projection $\id-p(x)\in \text{Pr}^G(\mathbb E)$, $x\in X$, satisfies $\Im(\id-p(x))=\Ker(p(x))$. Since for $x=(1,\lambda)\in X$  the operator
\[d(\lambda):=\id-p(1,\lambda)|_{\mathbb E_o}=b(1,\lambda)(\id-Q)b(1,\lambda)^{-1}|_{\mathbb E_o}\]
restricted to the subspace $\mathbb E_o$ is an isomorphism from $\mathbb E_o$ to $b(1,\lambda)\mathbb E_o$,  thus $d(\lambda)^{-1}b(1,\lambda)|_{\mathbb E_o}$ is a $G$-equivariant endomorphism of $\mathbb E_o$ and  (since $b(1,\lambda)$ is an isomorphism) satisfies
\[
b(1,\lambda)d(\lambda)^{-1}b(1,\lambda)|_{\mathbb E_o}=b(1,\lambda)|_{\mathbb E_o}.
\]
Define $\wt h:X\to \text{GL}_c^G(\mathbb E)$ by
\[
\wt h(t,\lambda):=b(t,\lambda)\circ Q+(\id-p(t,\lambda))\circ d(\lambda)^{-1}\circ   b(1,\lambda)(\id -Q), \quad (t,\lambda)\in X.
\]
Notice that since $\id-p(t,\lambda)|_{\mathbb E_o}:\mathbb E_o\to \ker p(x)$ is a $G$-equivariant isomorphism, thus $\wt h(x)\in \text{GL}_c^G(\mathbb E)$ for all $x\in X$. Moreover, for $(1,\lambda)\in X$ one has
\begin{align*}
\wt h(1,\lambda)&=b(1,\lambda)\circ Q+(\id-p(1,\lambda))\circ d(\lambda)^{-1}  \circ b(1,\lambda)(\id -Q)\\
&=b(1,\lambda)\circ Q+b(1,\lambda)\circ (\id-Q)\circ b(1,\lambda)^{-1}\circ b(1,\lambda)(\id -Q)\\
&=b(1,\lambda)\circ Q+b(1,\lambda)\circ (\id -Q)\\
&=b(1,\lambda)=\mathscr A(\lambda).
\end{align*}
On the other hand, one has the following block-matrix
\begin{align*}
\wt h(0,\lambda)
&=
\mathop{\begin{bmatrix}
\id &0\\
M_1(\lambda) &M(\lambda)
  \end{bmatrix}}\limits^{\mathbb E^o\quad\;\;\;\;\;\;\mathbb E_o}
  \begin{array}{c}\text{\small $\mathbb E^o$} \\ \text{\small $\mathbb E_o$}
  \end{array}
\end{align*}
where $M(\lambda):\mathbb E_o\to \mathbb E_o$ is a $G$-equivariant isomorphism.
Notice that
in order to conclude the construction, one needs to connect the homotopy $\wt h$ with the homotopy $\wh h$ given by
\[
\wh h(s,\lambda)=\mathop{\begin{bmatrix}
\id &0\\
(1-s)M_1(\lambda) &M(\lambda)
  \end{bmatrix}}\limits^{\mathbb E^o\quad\;\;\;\;\;\;\mathbb E_o}
  \begin{array}{c}\text{\small $\mathbb E^o$} \\ \text{\small $\mathbb E_o$}
  \end{array}, \quad s\in [0,1],
\]
and the conclusion follows.
\end{proof}
\vs

\section{Twisted Equivariant Degree}\label{appendix:Twist} 
{ Before passing to the $S^1$-equivariant degree, we recall some basic equivariant jargon and related standard notations.
\vs
For $H$ and $K$ two closed subgroups of a compact Lie group $G$, we say that $H$ is {\it conjugate} to $K$ in $G$ if and only if $H=gKg^{-1}$ for some $g \in G$. The conjugacy relation is an equivalence relation. We will write $H \sim K$ to denote that $H$ is conjugate to $K$ and will write $(H)$ for the equivalence class of $H$, called the {\it conjugacy class} of $H$ in $G$.
\vs
Denote by $\Phi(G):= \{(H) : H \leq G\}$, the set of all conjugacy classes. Notice that there is a natural {\it partial order relation} ``$\leq$" in $\Phi(G)$ defined as follows:
\[
(H) \leq (K) \;  \Leftrightarrow \; \exists_{g \in G} \; gHg^{-1} \subseteq K \; \text{for} (H), (K) \in \Phi(G).
\]
For $H \leq G$, denote by $N(H):=\{ g \in G : gHg^{-1} =H\}$ the {\it normalizer} of $H$ in $G$ and by $W(H):=N(H)/H$ the {\it Weyl group} of $H$ in $G$.
\vs
For a G-space $X$ and $x\in X$, denote by $G_x:=\{g\in G: gx=x\}$ the {\it isotropy} group of x, by $G(x):=\{gx:g\in G\}$ the {\it orbit} of x, and call $(Gx)$ the {\it orbit type} of $x\in X$. Put $\Phi_{n}(G) := \{ (H) \in \Phi(G) : \dim W(H) =n, n \in \bn \cup \{0\} \}$, $\Phi(G;X) := \{(H) \in \Phi(G) : H = G_x, x \in X\}$, and $\Phi_{n}(G;X) := \Phi_{n}(G) \cap \Phi(G;X) $. 
\vs
For a subgroup $ H \leq G$, the subspace $ X^H := \{ x \in X : G_x \geq H\}$ is called the {\it $H$-fixed-point subspace} of $X$. We also put $X_H := \{ x \in X : G_x = H\}$ and $X_{(H)} := \{ x \in X : (G_x)=(H)\}$. The Weyl group $W(H)$ acts freely on $X_H$. 
\vs
Given two $G$-spaces $X$ and $Y$, a continuous map $f: X \to Y$ is called {\it $G$-equivariant} (in short, {\it $G$-map}) if for all $x \in X, g \in G, f(gx)=gf(x)$ holds. 
If $X$ is a smooth manifold on which $G$ acts by diffeomorphisms, then $X$ is called a {\it $G$-manifold}.
}

\vs

The $S^1$-equivariant degree, which was introduced in 1980s is presently known and widely used as an effective  tool to study the existence and bifurcation of periodic solutions in the autonomous first order differential systems (cf. \cite{DGJM}, see also \cite{AED,Wu}). In short, the $S^1$-degree (as well as the general twisted degree) can be characterized by its property. For this end, we need some notation. 
\vs

By $A_1(S^1)$ we denote the free $\bz$-module generated by $\Phi_1(S^1):=\{(\bz_n): n=1,2,3,\dots\}$. For $G:=\Gamma \times S^1$, where $\Gamma$ stands for  a compact Lie group, and a subgroup $H\le G$ is called {\it twisted}, if there exists a subgroup $H\le \Gamma$ and continuous homomorphism $\theta: H\to S^1$ such that for some $l \in \{0,1,2,\dots\}$ one has
\[
H=\{(\gamma,z)\in K \times S^1: \theta(\gamma)=z^l\}=:K^{\theta,l}.
\]
We restrict our attention to such twisted subgroups $H$ for which the Weyl's group $W_G(H)$ of $H$  in $G$ satisfies  $\dim W_G(\mathcal H) = 1$. Then, $H=K^{\theta,l}$ will be called a {\it $\theta$-twisted $l$-folded subgroup}. We denote by $\Phi_1^t(G)$ the set of all $\theta$-twisted $l$-folded conjugacy classes 
%of {\it twisted} subgroups 
$(K^{\theta,l})$ such that  $\dim W_G(K^{\theta,l}) = 1$, and define $A_1^t(G)$ to be  the $\bz$-module $A_1^t(G)$ generated by $\Phi_1^t(G)$, i.e. 
\[
A_1^t(G):=\bz[\Phi_t^1(G)].
\]
We also define $A(\Gamma)$ as the free $\bz$-module generated by $\Phi_0(\Gamma):=\{(K):\dim\, W_\Gamma (K)=0\}$. Notice that
 $A(\Gamma)$ has a ring structure with the multiplication defined by the generators $(H)$, $(K)\in \Phi_0(\Gamma)$ by
 \begin{align*}
 (H)\cdot (K)&=:\sum_{(L)\in \Phi_0(\Gamma)} n_L\;(L)\\
 n_L&:= \text{number of $(L)$-orbits in }\; \frac\Gamma H\times \frac \Gamma K,
 \end{align*}
 which defines a ring structure on $A(\Gamma)$ called  the {\it Burnside ring} of $\Gamma$ (see \cite{AED}). Similarly, $A_1^t(G)$ has a structure of $A(\Gamma)$-module with the multiplication 
 defined on generators $(K)\in \Phi_0(\Gamma)$ and $(\mathscr H)\in \Phi_1^t(G)$ by 
\begin{equation}\label{eq:multi}
 (K)\cdot (\mathscr H) := \sum_{(\mathscr L)\in \Phi_1^t(G)} n_{\mathscr L}\;(\mathscr L)
\end{equation}
 where
\begin{equation}\label{eq:nl}
 n_{\mathscr L} := \text{number of $(\mathscr L)$-orbits in }\; \frac {\Gamma\times S^1} K\times \frac {\Gamma\times S^1} {\mathscr H}.
\end{equation}
 \vs
 
 There exists a unique function $\s1deg: \mathscr M_1^{S^1}\to  A_1(S^1)$ satisfying the following properties satisfying the standard {\it additivity, homotopy} and {\it multiplicativity properties} including, in particular,  the {\it recurrence property}, which allows effective computations of $S^1$-degree and formal proofs of all other properties. The $S^1$-degree was introduced in \cite{DGJM} and we refer to \cite{AED} for all the details.
 \vs
Assume that $V$ is an orthogonal $G$-representation,  $f:\br\times V\to V$ an $G$-equivariant map and $
\Omega\subset \br\times V$ is an open bounded $G$-invariant set. Then the pair $(f,\Omega)$ is called an {\it admissible $G$-pair} in $\br\times V$  if $f(t,v)\not=0$ for all $(t,v) \in \partial \Omega$. We denote by $\mathscr M_1^G(\br\times V)$ the set of all admissible $G$-pairs in $\br\times V$.
\vs

We define for an admissible $S^1$-pair $(f,\Omega)$ the integer
\begin{equation}\label{eq:s1-cumul-deg}
\deg_{S^1}(f,\Om):=\sum_{(H)\in \Phi_1(S^1)} \text{\rm coeff\,}^H\Big(\s1deg(f,\Om) \Big)
\end{equation}
and call it 
%$\deg_{S^1}(f,\Om)$ is called the 
{\it cumulative} $S^1$-equivariant degree of $f$ in $\Om$.
\vs

Let $(f,\Omega)\in \mathscr M_1^G(\br\times V)$ and $\mathcal H\in \Phi_1(G)$. Then $f^{\mathcal H}:\br\times V^{\mathcal H}\to V^{\mathcal H}$ is $N(\mathcal H)$ equivariant. Since $S^1\le N(\mathcal H)$, $(f^{\mathcal H},\Omega^{\mathcal H})$ is an admissible $S^1$-pair, one can   define the numbers 
$\{d_{\mathcal H}(f,\Omega)\,: \,(\mathcal H)\in \Phi_1^t(G;\Omega)\}$ by the recurrence formula:
\begin{equation}\label{eq:H-twisted}
d_{\mathcal H}(f,\Omega):=\frac{\displaystyle \deg_{S^1}(f^{\mathcal H},\Om^{\mathcal H})-\sum_{(\mathcal K)>(\mathcal  H)}d_{\mathcal K}(f,\Omega)\, n(\mathcal H,\mathcal K)\, 
|W(\mathcal K)/S^1|}{|W(\mathcal H)/S^1|},
\end{equation}
where $ n(\mathcal H,\mathcal K)$ stands for the cardinality of the set $\{\mathcal K'\in (\mathcal K): \mathcal H\le \mathcal K'\}$.
We are now in a position to introduce a concept which is crucial.
 
\begin{definition}\label{def:twisted-degree} For a given $(f,\Om)\in \mathscr M_1^G$, the element $\gdeg(f, \Om)\in A_1^t(G)$ defined by 
\begin{equation}\label{eq:twisted-def}
\gdeg(f,\Om)=\sum_{(\mathcal H)\in \Phi^t_1(G)} d_{\mathcal H}(f,\Om) \,\,(\mathcal H),
\end{equation}
where $ d_{\mathcal H}(f,\Om) \in \bz$ is given by \eqref{eq:H-twisted}, is called the {\it twisted $G$-equivariant degree} of $f$ in $\Om$.
\end{definition} 
Before listing the properties satisfied by the twisted degree, we need to adopt an additional notation. Given an orthogonal  $G$-representation $V$ with the $G$-action 
$\varphi: G \times V \to V$, the homomorphism  $\psi_m:\Gamma\times S^1\to \Gamma\times S^1$ defined by  $\psi_m(\gamma,z)=(\gamma,z^m)$, induces the $G$-action $\varphi_m: G \times V \to V$ defined by $\varphi_m \big((\gamma,z), v\big):= \psi_m(\gamma,z)v$. Denote by
%, that will be used to describe the properties of $\s1deg$. 
 $^mV$ the $G$-representation related to  the $G$-action $\varphi_m$.
Also, for a $G$-invariant  subset $\Om\subset \br\times V$, we will write $^m\Om$ to point out that $\Om$ is considered with the $G$-action $\varphi_m$. 
%. Notice that the homomorphism $\psi_m$ induces the   {\it $m$-folding homomorphism} of $\bz$-modules  $\Psi_m:A^t_1(G)\to A_1^t(G)$ defined by $\Psi_m(H)=\psi\one_m(H)$.  However, the following statement is also true.
\vs

%Notice that, by using the properties  \ref{t1}--\ref{t3}, one obtains 
The following statement holds.
\vs

\begin{proposition}\label{pro:twisted-deg} The function $\gdeg: \mathscr M_1^{G}\to  A^t_1(G)$ defined by \eqref{eq:twisted-def} satisfies the following properties:
\begin{enumerate}[label=($\mathfrak T_{\arabic*}$)]\itemc

\item\label{T1} {\bf (Additivity)}  For two disjoint open $G$-invariant subsets  $\Omega_1$ and $\Omega_2$ of $\Omega$ such that $f^{-1}(0)\cap \Omega\subset \Omega_1\cup\Omega_2$, one has
\begin{equation}\label{eq:add-twisted}
\gdeg(f,\Omega)=\gdeg(f,\Omega_1)+\gdeg(f,\Omega_2);
\end{equation}
\item\label{T2} {\bf (Homotopy)} If $(F_s, \Omega) \in \mathscr M_1^G$, $s\in [0,1]$, 
is an $\Omega$-admissible $G$-homotopy, then $\gdeg(F_s,\Omega)$ is independent of $s$;

%for an $\Omega$-admissible $G$-homotopy $F_s:\br\times V\to V$, $s\in [0,1]$, one has
%\[
%\gdeg(F_s,\Omega)=\text{\rm constant};
%\]
\item\label{T3} {\bf (Normalization)}   Let  $(f,\Om) \in  \mathscr M_1^{G}$ be such that $f$ is regular normal in $\Om$. Then:
\begin{itemize}
\item[(i)] If $f^{-1}(0)\cap \Om = G(w_o)$ for some  $w_o \in \Om$ with
$G_{w_o}={\mathcal H}$ such that $(\mathcal H) \in \Phi_1^t(G; \Om)$,  then,
\[
\gdeg(f,\Om)=\text{\rm sign\,} \det(Df(w_o)|_{S_{w_o}}),
\]
where $S_{w_o}$ is the positively oriented slice to $W(\mathcal H)(w_o)$  at $w_o$ in 
$\Om_{\mathcal H}$;
\item[(ii)] If $\Phi_1^t(G; f^{-1}(0)\cap \Om) = \emptyset$, then $\gdeg(f,\Om) = 0$;
\end{itemize}

%{\bf (Normalization)}  If $f$ is a regular normal map in $\Om$ such that $f^{-1}(0)\cap \Om=G(w_o)$ and $G_{w_o}=H$, $(H)\in \Phi_1^t(G)$, then 
%\[
%\gdeg(f,\Om)=\rho_o\, (H), \quad \rho_o:=\text{\rm sign\,} \det(Df(w_o)|_{S_{w_o}}),
%\]
%where $S_{w_o}$ is the positively oriented slice to $W(H)(w_o)$  at $w_o$ in $\Om_H$.

\item\label{T4} {\bf (Existence)} If $(f,\Om)\in  \mathscr M_1^{G}$ and 
{\rm coeff$^{\mathcal H}\big(\gdeg(f,\Om)\big)\not=0$} for  some $(\mathcal H) \in \Phi_1^t(G;\Om)$, then there exists $(t,x)\in \Om$ such that $G_{(t,x)}\ge {\mathcal H}$ and $f(t,x)=0$;

%\item\label{T4} {\bf (Existence)} If $(f,\Om)\in  \mathscr M_1^{G}$ and,  for  some $H\in %\Phi_1^t(G;\Om)$, one has 
%{\rm coeff$^H\big(\gdeg(f,\Om)\big)\not=0$}, then there exists $(t,x)\in \Om$ such that 
%$G_{(t,x)}\ge H$ and $f(t,x)=0$;

\item\label{T5} {\bf (Excision)} Suppose that $(f,\Omega)\in  \mathscr M_1^{G}$ and $\Omega_o$ is an open $G$-invariant subset of $\Omega$ such that $f\one(0)\cap\Omega\subset \Omega_o$. Then, $\gdeg(f,\Omega)=\gdeg(f,\Omega_o)$;

\item\label{T6}  {\bf (Rouch\'e's Property)} If $(f,\Omega)\in  \mathscr M_1^{G}(\br\times V)$ and $g:\br\times V\to V$ is a $G$-equivariant map such that
\begin{equation}\label{eq:Rouche-prperty}
\sup_{(t,x)\in \partial \Omega}|f(t,x)-g(t,x)|<\inf_{(t,x)\in \partial \Omega}|f(t,x)|,
\end{equation}
then $(g,\Omega)\in \mathscr M_1^{G}(\br\times V)$ and $\gdeg(f,\Omega)=\gdeg(g,\Omega)$;

\item\label{T7} {\bf (Boundary)} If $(f,\Omega)\in  \mathscr M_1^{G}(\br\times V)$  and $g:\br\times V\to V$ is an $S^1$-equivariant  map such that $f(t,x)=g(t,x)$ for all $(t,x)\in \partial \Omega$, then $\gdeg(f,\Omega)=\gdeg(g,\Omega)$;

\item\label{T8} {\bf (Folding}) If $m\in \bn$ and $(f,\Om)\in  \mathscr M_1^{G}$, then 
\[
\gdeg(f,{^m\Om})=\Psi_m\big(\gdeg(f,\Om) \big);
\]

\item\label{T9} {\bf (Suspension)} If   $(f,\Om)$, $(\id, U)\in  \mathscr M_1^{G}$, where $U$ is a neighborhood of zero, then 
\[
\gdeg(f\times \id,\Om\times U)=\gdeg(f,\Om);
\]

\item\label{T10} {\bf (Product)} Let $(f,\Om)\in  \mathscr M_1^{G}$,   $W$ be an orthogonal $\Gamma$-representation  and $U \subset W$  be an open bounded neighborhood of zero and  $(g,U) \in \mathcal M^\Gamma(W)$. Then,
\[
\gdeg(g \times f,U \times \Omega) = \Gamma\text{\rm -deg}(g,U) \cdot \gdeg(f,\Omega),
\]
where the multiplication `$\cdot$' is taken in the $A(\Gamma)$-module $A_1^t(G)$ (cf. 
\eqref{eq:multi}--\eqref{eq:nl}).
\end{enumerate}
\end{proposition}
\vs
The irreducible $G$-representations can be easily described. Consider a complete list of irreducible  $\Gamma$-representations $\{\cV_j\}$. Take an irreducible $\Gamma$-representation $\cV_j$, $j\in \bn$. If $\cV_j$ is of real type than one can consider its complexification $W:=V^c_j$, and if  $\cV_j$ is not of real type, then it already has a complex structure, so we can put $W:=\cV_j$.  Next, for a given $k\in \bn$,  we define the $S^1$-action on $W$ by the $k$-folding, i.e. $zw:=z^k\cdot w$, $z\in S^1$, $w\in W$, where $`\cdot '$ denotes the complex multiplication. Then, the space $W$, equipped with the above $\Gamma\times S^1$-action, becomes an irreducible  $G$-representation which is denoted by  $\cV_{k,j}$. 
\vs

\
\begin{definition}\label{def:basic-twiste-d} 
Assume that $k>0$ and consider the irreducible 
$G = \Gamma \times S^1$-representation $\cV_{k,j}$. Put
\begin{equation}\label{eq:basic-twisted}
\deg_{\cV_{k,j}}:=\gdeg(b,\Omega),
\end{equation}
where $b:\br\times \cV_{k,j}\to \cV_{k,j}$ is given by 
\begin{equation}\label{eq:basic-real}
b(t,v)=(1-|\lambda|+it)\cdot v, \quad (t,v)\in \br\times \cV_{k,j},
\end{equation} 
and
\begin{equation}\label{eq:Om-basic-twisted}
\Omega:=\left\{(t,v)\in \br\times \cV_{k,j}: \tfrac 12<|\lambda|<2,\; |t|<1\right\}.
\end{equation}
Then, the element $\deg_{\cV_{k,j}}\in A_1^t(G)$ is called the {\it twisted $\cV_{k,j}$-basic degree} (or simply the {\it $\cV_{k,j}$-basic degree}).
\end{definition}
\vs
By recurrence formula \eqref{eq:H-twisted}, one has that
\[
\deg_{\cV_{k,l}}=\sum_{(H)\in \Phi_1^t(G;\cV_{k,j})} n_H(H)
\]
where 
\[
n_H=\frac{\frac 12 \dim \cV_{k,j}^H-\sum_{(K)>(H)} n_K\,n(H,K)\,|W(K)/S^1|}{|W(H)/S^1|}.
\]
Clearly, if $(H)$ is a maximal twisted orbit type in $\Phi_1^t(G;\cV_{k,j})$, then
\begin{equation}\label{eq:max-orb-basic}
n_H=\frac{\frac 12 \dim \cV_{k,j}^H}{|W(H)/S^1|}.
\end{equation}
\vs

Let $W$ is an orthogonal $G$-representation and consider its $G$-isotypic decomposition 
\begin{equation}\label{eq:Wiso}
W=\bigoplus\limits_{k,l} W_{k,l}, \quad l=0,1,2,\ldots,r, \; k\geq 0,
\end{equation}
where the $G$-isotypic component $W_{k,l}$ of $W$ is modeled on $\cV_{k,l}$. Put
\begin{equation}
m_{k,l}:=\dim W_{k,l}/ \dim \cV_{k,l}.
\end{equation}
For a map $a : S^1\to \text{\rm GL}^{G}(W)$, denote by $a_{k,l} : S^1\to \text{\rm GL}^{G}(W_{k,l})$ the restriction of $a$ to the isotypic component $W_{k,l}$ and take the complex $m_{k,l} \times m_{k,l}$-matrix $\widetilde{a}_{k,l}(\lambda) : \bc^{m_{k,l}} \to \bc^{m_{k,l}}, \lambda \in S^1$, representing $a_{k,l}(\lambda)$. Define, 
\begin{equation}\label{eq:dkl}
d_{k,l} := \deg(\det_{\bc}(\widetilde{a}_{k,l})), \quad l=0,1,2,\ldots,r, \; k > 0,
\end{equation}
and 
\begin{equation}\label{eq:rhol}
\rho_l := \text{\rm sign} \det (a_{0,l}(\lambda)), \quad a_{0,l}(\lambda) := a(\lambda)\arrowvert_{W_{0,l}} \quad (\lambda \in S^1).
\end{equation}
Based on these notations one has the following important computational result. 

\vs
\begin{theorem}\label{th:comp-twisted}
Suppose that $W$ is an orthogonal $G$-representation admitting the isotypic decomposition \eqref{eq:Wiso}, and $\mathfrak O\subset \bc\times W$ is given by 
\[
\mathfrak O:=\left\{(\lambda,v)\in \bc\times W: |v|<2,\; \frac 12 <|\lambda|<2 \right\}.
\]
Take a continuous map  $a : S^1\to \text{\rm GL}^{G}(W)$ and let $F_a : \bc \times W \to W$ be a $G$-equivariant map, which  on $\mathfrak O$ is given by 
\[
F_a(\lambda,v) = a\left(\tfrac \lambda{|\lambda|}\right)v, \quad (\lambda,v)\in \mathfrak O.
\]
Let $\theta : \bc \times W\to \br $ be a $G$-invariant (continuous) auxiliary function
in $\mathfrak O$  for $F_a$ satisfying the condition $\theta (\lambda,0) > 0$ 
\big(resp. $\theta (\lambda,0) < 0$\big) for 
$| \lambda | = 2$  \big(resp. $| \lambda | = {1 \over 2}$\big),
 and let $f_{\theta,a} : \bc\times W\to \br\times W$ be given by
\[
f_{\theta,a}(\lambda,v)=\big(\theta(\lambda,v), F_a(\lambda,v)\big).
\]
Then,
\begin{equation}\label{eq:twisted-fa}
\gdeg(f_{\theta,a},\mathfrak O) = \prod_{l} {\bm \delta_l}\cdot \mathop{\sum_{k,l}}\limits_{k>0} d_{k,l}\deg_{\cV_{k,l}},
\end{equation}
where $\deg_{\cV_{k,l}}$ is described in Definition \ref{def:basic-twiste-d}, $d_{k,l}$ is given by \eqref{eq:dkl} and 
\begin{equation}\label{eq:factor-j-twisted}
{\bm \delta_l}:=\begin{cases}
\big(\deg_{\cV_l}\big)^{m_{l,0}} & \text{ if }\; \rho_l=-1,\\
(\Gamma) & \text{ if }\; \rho_l=1,
\end{cases}
\end{equation}
(here $\rho_l$ is given by \eqref{eq:rhol}).
\end{theorem}
\vs
Consider the {\it $k$-folding homomorphism} $\psi_k:\Gamma\times S^1\to \Gamma\times S^1$, given by $\psi_k(\gamma,z)=(\gamma,z^k)$, $(\gamma\times z)\in \Gamma\times S^1$. Then $\psi_k$ induces a $A(\Gamma)$-module homomorphism $\Psi_k:A_1^t(G)\to A_1^t(G)$ which is given on the generators $(H)\in \Phi_1^t(G)$ by
\[
\Psi(H):=(\psi_k^{-1}(H)).
\]
Then, we have the following:
\vs
\begin{proposition}\label{pro:basic-k}
Let $\Gamma$ be a compact Lie group and $G:=\Gamma\times S^1$. Then for all $k\in \bn$ and $j=0,1,2,\dots$, one has
\begin{equation}\label{eq:basic-k}
\deg_{\cV_{k,j}}=\Psi_k\left(\deg_{\cV_{1,j}}\right).
\end{equation}
\end{proposition}
\vs
Proposition \ref{pro:basic-k} reduces the computations of basic twisted degrees to the degrees $\deg_{\cV_{1,j}}$. In the next subsection we present some techniques allowing us to identify the non-zero coefficients of $\deg_{\cV_{1,j}}$ corresponding to maximal twisted orbit types in $\Phi_1^t(G;\cV_{1,j})$.
\vs

\section{Computational techniques  for twisted degree}\label{app:D}
%Assume that $\Gamma$ is a compact Lie group and put $G:=\Gamma\times S^1$. 
In order to describe twisted orbit types in $\cV_{1,j}$ consider  an orbit type $(K)$ in $\cV_j$ such that $W_\Gamma(K)$ is finite, i.e. for some $w\in W$ (here we consider $W$ as a complex $\Gamma$-representation) $\Gamma_w=K$ and $(K)\in \Phi_0(\Gamma)$. 
Since the subspace $W_\Gamma^K$ is $N_\Gamma(K)$-invariant, it is also $W_\Gamma(K)$-invariant. Therefore, if for some $(g,z)\in \Gamma\times S^1$ one has 
\[
(g,z)w=w\;\;\; \Leftrightarrow\;\;\; z\cdot gw=w,
\]
then it follows that $gw\in W^K$ and consequently $g\in N_\Gamma(K)$.
\vs
Therefore, if $N_\Gamma(K)=K$, one has that $G_w=K$ and therefore $(K)$ (considered here as a twisted subgroup) is an orbit type in $\cV_{1,j}$.
\vs
On the other hand, if  $N_\Gamma(K)\supsetneq K$, then there exits $g\in N_\Gamma(K)\setminus K$. Suppose that for some $z\in S^1$ one has 
$(g,z)w=w$, i.e. $gw=z^{-1}w$. Put $\pi :N_\Gamma(K)\to W_\Gamma(K)$ is the natural projection and notice that (since $ W_\Gamma(K)$ is finite), the coset $gK$ has a finite order in $W_\Gamma(K)$, i.e. there exists $m\in\bn$ such that $g^mK=K$ (and $g^lK\not=K$ for $1\le l<m$), which implies that 
\[
w=g^{m}w=z^{-m}\cdot w \;\;\; \Rightarrow\;\;\; z^{-m}=1\;\;\; \Rightarrow \;\;\; z\in \bz_m.
\]
In particular, denote by $\langle gK\rangle\le W_\Gamma (K)$ the cyclic group generated by $gK$ (of order $m $), i.e.  $\langle gK\rangle\simeq \bz_m$ and define the homomorphism $h:\langle gK\rangle\to S^1$ by $h(g)=z^{-1}$.  Then, using the twisted notation, we have
\[
G_w\ge K^h:=\{(g,z)\in N_m\times S^1: h(gK)=z\}, \quad N_m=\pi^{-1}(\langle gK\rangle).
\]
Denote by $\pi_1:\Gamma\times S^1\to \Gamma$ the natural projection $\pi_1(g,z)=g$, and put $N_o:=\pi_1(G_w)\le N_\Gamma(K)$. Then, by definition, for every $g\in N_o$ there exists a unique $z\in S^1$ such that $(g,z)w=w$. Indeed, if 
\[
\begin{cases}
(g,z)w=w\;\;\Leftrightarrow\;\; gz=gw=z^{-1}w\\
(g,z')w=w\;\;\Leftrightarrow\;\; gz'=gw={z'}^{-1}w
\end{cases}
\;\; \Rightarrow\;\; z^{-1}={z'}^{-1}\;\; \text{ i.e. }\; \;z=z'.
\]
That means $G_w$ is a graph of the homomorphism $\vp:N_o\to S^1$ defined by $\vp(g)=z$ where  $(g,z)w=w$.  Put $\bz_n:=\vp(N_o)\subset S^1$, and notice that by isomorphism theorem $N_o/Ker(\vp)=\bz_n$ and we have
\[
(g,z)\in G_w\;\;\; \Leftrightarrow\;\;\; \vp(g)z^{-1}=1,
\]
which implies that 
\begin{equation}\label{eq:twisted-orbit-types}
(G_w)=(K^{\vp}).
\end{equation}
\vs
In order to describe the twisted orbit types in $\cV_{1,j}$, define the set 
\[
\mathscr X:=\{(N,K,\vp): K=\Gamma_w\;\text{ for some }w\in W,\; N\le N_\Gamma (K), \vp:N\to S^1\},
\]
where $\vp$ is a homomorphism from $N$ to $S^1$. We introduce the following equivalence relation in $\mathscr X$
\[
(N,K,\vp)\sim (N',K',\vp') \;\; \Leftrightarrow\;\; \exists_{g\in \Gamma} \;\; \begin{cases}
gNg^{-1}=N',\\
gKg^{-1} =K',\\
\vp'(h')=\vp(g^{-1}h'g) & \; \forall_{h'\in N'}.
\end{cases}
\]
\vs
Based on our considerations above, we have the following
\vs

\begin{proposition}\label{pro:twisted-orbit-types}
The twisted obit types in $\cV_{1,j}$ are in one-to-one correspondence with the elements of $\mathscr X/\sim$. To be more precise, for  a representative $(N,K,\vp)$ of an element in $\mathscr X/\sim$, the conjugacy class of $K^\vp$ is an orbit type in $\cV_{1,j}$.
\end{proposition}
\vs
%\section{Dimensions of the fixed point spaces in $\cV_{k,j}$ corresponding to twisted orbit types}

Suppose $(N^\vp)$ is an orbit type in $\cV_{1,j}$ corresponding to the representative $(N,K,\vp)$ in $\mathscr X/\sim$. Then $\vp:N\to SO(2)\le O(2)$ is a representation of $N$. If $\vp(N)=:\bz_m $ is such that $m>1$, this representation is irreducible, and if $m=2$, it is reducible (with two irreducible components). Denote this irreducible representation by $\cU_o$. Notice that $W^K$ is an $N$-representation, thus one has a $N$-isotopic decomposition of $W^K$ given by
\[
W^K=U_0\oplus U_1\oplus \dots\oplus U_{\lfloor \frac m2\rfloor}.
\]
One of these components, say $U_i$,  is modeled on  $\cU_o$. Then,  denote by $m_o$ the real dimension of $U_i$ and notice that 
\[
\dim (\cV_{1,j}^{N^\vp})=m_o.
\] 
\vs
\begin{example}\label{ex:o2xd4xz2} \rm Consider the group $\Gamma:=O(2)\times D_4\times \bz_2$. Let us first consider the group $\Gamma_1:=D_4\times \bz_2$.
In order to identify the conjugacy classes of subgroups in $\Gamma$, first we notice that that we have
\begin{align*}
\Phi(D_4\times \bz_2)&=\{(\bz_1),(\bz_2) , (\wt D_1^z),  (\bz_1^p), (\bz_2^m), (\wt D_1), (D_1), (D_1^z), (D_2^z),(\bz_4), (\bz_4^d),\\
&\quad (\wt D_2),(D_2),(\wt D_2^z), (\bz_2^p), (D_1^p),(D_2^d), (\wt D_1^p), (\wt D_2^d), (D_4),(\bz_4^p),(D_4^z),(D_2^p),\\
&\quad (D_4^d),(\wt  D_2^p,(D_4^{\hat d}),(D_4^p)\},
\end{align*}
where
\begin{align*}
D_4=\{(1,1),(\gamma,1),(\gamma^2,1),(\gamma^3,1),(\kappa,1),(\gamma\kappa,1),(\gamma^2\kappa,1),(\gamma^3\kappa,1)\},\quad
\gamma:=e^{i\frac \pi 2},\;\quad \kappa=\left[\begin{array}{cc} 1&0\\0&-1\end{array}\right].
\end{align*}
and 
\begin{align*}
\bz_4&=\{(1,1),(\gamma,1),(\gamma^2,1),(\gamma^3,1)\}, \quad \bz_2:=\{(1,1),(\gamma^2,1),\\
D_2&=\{(1,1),(\gamma^2,1),(\kappa,1),(\gamma^2\kappa,1)\},\; \wt D_2=\{(1,1),(\gamma^2,1),(\gamma\kappa,1),(\gamma^3\kappa,1)\},\\
D_1&=\{(1,1),(\kappa,1)\},\; \wt D_1=\{(1,1),(\gamma\kappa, 1)\}, \; \bz_1=\{(1,1)\}.
\end{align*}
The remaining subgroups are 
\begin{align*}
D_4^d&=\{(1,1),(\gamma,-1),(\gamma^2,1),(\gamma^3,-1),(\kappa,1),(\gamma\kappa,-1),(\gamma^2\kappa,1),(\gamma^3\kappa,-1)\},\\
D_4^{\hat d}&=\{(1,1),(\gamma,-1),(\gamma^2,1),(\gamma^3,-1),(\kappa,-1),(\gamma\kappa,1),(\gamma^2\kappa,-1),(\gamma^3\kappa,1)\},\\
D_4^z&=\{(1,1),(\gamma,1),(\gamma^2,1),(\gamma^3,1),(\kappa,-1),(\gamma\kappa,-1),(\gamma^2\kappa,-1),(\gamma^3\kappa,-1)\},\\
\bz_4^d&=\{(1,1),(\gamma,-1),(\gamma^2,1),(\gamma^3,-1)\}, \quad \bz_2^m:=\{(1,1),(\gamma^2,-1),\\
D_2^d&=\{(1,1),(\gamma^2,-1),(\kappa,1),(\gamma^2\kappa,-1)\},\; \wt D_2^d=\{(1,1),(\gamma^2,-1),(\gamma\kappa,1),(\gamma^3\kappa,-1)\},\\
D_2^z&=\{(1,1),(\gamma^2,1),(\kappa,-1),(\gamma^2\kappa,-1)\},\; \wt D_2^z=\{(1,1),(\gamma^2,1),(\gamma\kappa,-1),(\gamma^3\kappa,-1)\},\\
D_1^z&=\{(1,1),(\kappa,-1)\},\; \wt D_1^z=\{(1,1),(\gamma\kappa, -1)\},
\end{align*}
and the product sub groups $H^p=H\times \bz_2$, where $H\le D_4$.
In order to manage the conjugacy classes of subgroups in $\Gamma=O(2)\times D_4\times \bz_2$, we assist our computation by GAP systems. The following code is used together with  the package {\tt EquiDeg} (cf. \cite{Pin}): 
\vs
\begin{lstlisting}[language=GAP, frame=single]
LoadPackage( "EquiDeg" );
# Defining O_2, group1:= D_4, group2:= Z_2
o2 := OrthogonalGroupOverReal( 2 );
gr1 := pDihedralGroup( 4 );
gr2 := SymmetricGroup( 2 );
# Computing the product of D_4 and Z_2
g1:= DirectProduct( gr1, gr2);
# Conjugacy classes of subgroups of the product group g1 
# and corresponding name
ccs:=ConjugacyClassesSubgroups( g1 );
SetCCSsAbbrv(g1,["Z1", "Z2", "Dt1z", "D1z", "D1", 
		 "Z2m", "Z1p", "Dt1z", "D2d","Dt1p", 
   		"D1p","Z2p","D2","D2z", "Z4",
 		"Z4d","D2d","Dt2z", "Dt2d", "D4d",
		 "Dt2p","D2hd","D2p","D4z","Dt2p",
		 "D4","D4p"]);
# Defining the final product group
g := DirectProduct( o2, g1 );
# Computing and naming conjugacy classes subgroups of g
ccsg := ConjugacyClassesSubgroups( g );
\end{lstlisting}
\vs
In order to describe the conjugacy classes of subgroups in $G_1\times G_2$ (here $G_1=O(2)$ and $G_2=D_4\times \bz_2$) we use the so-called amalgamated notation
\[
\mathscr H=H^\phi \times _L^\psi K:=\{(h,k): \vp(h)=\psi(k)\}, \quad H\le G_1,\; K\le G_2,
\]
where $\vp:H\to L$ and $\psi:K\to L$ are surjective homomorphisms. In our problem one can identify $\vp$ and $\psi$ by their kernels $H_o:=\Ker(\vp)$ and $K_o:=\Ker(\psi)$ (here we are trying to describe the conjugacy classes $(\mathscr H)$ of $\mathscr H$), i.e. the group $\mathscr H$ is the amalgamated notation can be written as
\[
\mathscr H=H^{H_o} \times _L^{K_o} K.
\]
We represent the set  $\Phi_0(\Gamma)$ as the following  union 
\[
\Phi_0(\Gamma)=\bigcup_{m=0}^\infty \Phi^m,\quad \Phi^m\cap\Phi^{m'}=\emptyset, \;\; m\not=m',
\]
where $\Phi_m$ represents symmetries of the $2\pi$-periodic functions $x:\br\to V$ (here $V$ an arbitrary $D_4\times \bz_2$-representation) in $m$-th Fourier mode. For example, for $m=0$, the set $\Phi^0$  is composed of conjugacy classes of $\mathscr H$ of type
\[
\mathscr H=SO(2)\times  K, \quad \mathscr H=O(2)\times K,\quad \mathscr H=O(2)^{SO(2)}\times ^{K_o}_{\bz_2}K.
\]
Let us point out that $\Phi^0$ has exactly 118 elements. For example, one has
\begin{align*}
 H[0,65]&=O(2)^{SO(2)}\times_{\bz_2}^{\wt  D_2^z}D_4^z,\quad
 H[0,75]=O(2)^{SO(2)} \times_{\bz_2} ^{\wt D_2^z}D_4^d,\\
 H[0,83]&=O(2)^{SO(2)} \times_{\bz_2}^{ D_2}D_2^p,\quad H[0,97]=O(2) \times  D_2^d,\\
 H[0,107]&=O(2)^{SO(2)}\times ^{ D_2^p}_{\bz_2}D_4^p, \quad H[0,114]=O(2) \times D_4^d,\\
 H[0,117]&=O(2) \times  \wt D_2^p,\quad H[0,118]=O(2) \times D_4^p.
 \end{align*}
 On the other hand the set $\Phi^1$ is composed of 341  conjugacy classes of subgroups in $\Gamma$. Let us list some of the representatives of these classes using the amalgamated notation:
 \begin{align*}
 H[1,118]&=D_2^{D_1}\times_{\bz_2}^{\bz_1}D_1^z, \quad H[1,218]=\bz_2^{\bz_1}\times_{\bz_2}^{\wt D_2^z}\wt D_2^p,\\
 H[1,320]&=D_2^{D_1} \times_{\bz_2}^{ D_2}D_4^d,\quad H[1,337]=D_2^{D_1}\times_{\bz_2}^{ D_4^d}D_4^p,\\
 H[1,341]&=D_1\times D_4^p.
\end{align*}
 Then, to conclude one can observe that for every $m\in \bn$ the set $\Phi^m$ has also 341 elements and 
 \begin{align*}\Phi^m=\Psi_m(\Phi^1), 
 \end{align*}
 where $\Psi^m(\mathscr H):=\psi_m^{-1}(\mathscr H)$ and  $\psi_m:O(2)\times \Gamma_1\to O(2)\times \Gamma_1$ is the $m$-folding homomorphism.
 \vs
 Assume that $\Gamma:=O(2)\times D_4\times \bz_2$. The irreducible $\Gamma$-representations are described as $\cV_{k,j}:=\cW_k\otimes \cV_j$, where $\cV_j$ stands for the irreducible $D_4\times \bz_2$-representation (according the top listing provided by GAP) and $\cW_k$ is the irreducible $O(2)$-representation (with $SO(2)$-action by $k$-folding). Therefore, the irreducible $G$-representations are $\cV_{m,k,j}=\cW_m\otimes \cV_{k,j}$ (here $\cW_m$ is the $O(2)$-representation considered with the action of  $S^1=SO(2)$). Now, we can recognize the maximal twisted orbit types in $\cV_{m,k,j}$. We begin with the
 case of the one-dimensional  representation   $\cV_{k,2}$. The maximal orbit type in $\cV_{k,2}\setminus \{0\}$ is 
 \[
 (K):=(D_{2k}^{D_k}\times_{\bz_2}^{D_4^{ z}}D_4^p).
 \]
 and notice that 
 the normalizer $N(K)$ of $K$ in $\Gamma$ is $D_{2k}\times D_4^p$, i.e. $W(K)=\bz_2$.
 Thus the  maximal twisted orbit types in $\cV_{m,k,4}$ is $(\mathscr K)$, where 
 \begin{align*}
  (\mathscr K_o)&:=(N(K)^K\times_{\bz_2}^{\bz_m}\bz_{2m}).
  \end{align*}
  Next, consider $\cV_{k,4}$. The maximal orbit type in $\cV_{k,4}\setminus \{0\}$ is 
 \[
 (H):=(D_{2k}^{D_k}\times_{\bz_2}^{D_4^{d}}D_4^p).
 \]
 Notice that the normalizer of $H$ in $\Gamma$ is $D_{2k}\times D_4^p$, i.e. $W(H)=\bz_2$, therefore we obtain that the maximal twisted orbit type in $\cV_{m,k,4}$ is exactly
 \[
 (\mathscr H_o):=(N(H)^H\times_{\bz_2}^{\bz_m}\bz_{2m}).
 \]
 Similarly, we have the following maximal orbit types in the $\Gamma$-representation $\cV_{k,10}$
 \[
 (H_1):=(D_{2k}^{D_k}\times _{\bz_2}^{D_2^d}D_2^p),\quad  (H_2):=(D_{2k}^{D_k}\times _{\bz_2}^{\wt D_2^d}\wt D_2^p),\quad  (H_3):=(D_{4k}^{\bz_k}\times _{D_4}^{\bz_2^-}D_4^p).
 \]
 with the corresponding normalizers
 \[
 N(H_1)=D_{2k}\times D_2^p, \quad N(H_2)=D_{2k}\times \wt D_2^p, \quad N(H_3):=(D_{4k}^{\bz_k}\times_{D_4} D_4)\times \bz_2.
 \]
 Therefore, the maximal twisted orbit types in $\cV_{m,k,10}$ are $(\mathscr H_l)$, $l=1$, $2$, $3$, where 
 \[
 \mathscr H_1:=N(H_1)^{H_1}\times_{\bz_2}^{\bz_m}\bz_{2m}, \quad \mathscr H_2:=N(H_2)^{H_2}\times_{\bz_2}^{\bz_m}\bz_{2m},\quad \mathscr H_3:=N(H_3)^{H_3}\times_{\bz_2}^{\bz_m}\bz_{2m}.
 \]
 Consequently, by \eqref{eq:max-orb-basic}, we have
 \begin{gather}
 \text{coeff}^{\mathscr H_o}(\deg_{\cV_{m,k,4}})\not=0,\label{eq:max1}\\
  \text{coeff}^{\mathscr H_1}(\deg_{\cV_{m,k,10}})\not=0,\quad  \text{coeff}^{\mathscr H_2}(\deg_{\cV_{m,k,10}})\not=0, \quad  \text{coeff}^{\mathscr H_3}(\deg_{\cV_{m,k,10}})\not=0.\label{eq:max2}
 \end{gather}
\end{example}
\vs

%%%%%%%%%%%%%%%%%%%%%%%%%%%%%%%%

% \bibliographystyle{plain}
% \bibliographystyle{apalike}

\end{document}